\documentclass{amsart}

\usepackage{enumitem}

\usepackage{amssymb,amsmath,latexsym,mathtools,tensor}

\usepackage{graphicx}
\usepackage[font=small,labelfont=bf]{caption} % Required for specifying captions to tables and figures

\usepackage{etoolbox}

\usepackage{booktabs}

\usepackage{hyperref}

\expandafter\def\expandafter\normalsize\expandafter{%
    \normalsize
    \setlength\abovedisplayskip{10pt}
    \setlength\belowdisplayskip{10pt}
    \setlength\abovedisplayshortskip{10pt}
    \setlength\belowdisplayshortskip{10pt}
}

\makeatletter

\usepackage{epigraph}

\patchcmd{\epigraph}{\@epitext{#1}}{\itshape\@epitext{#1}}{}{}
\makeatother

\everymath=\expandafter{\the\everymath\displaystyle}

\usepackage{tikz-cd}

\usepackage{tikz}
\usetikzlibrary{fit}

\usepackage{blkarray}

\usetikzlibrary{patterns}
\usetikzlibrary{calc}

\pgfdeclarelayer{edgelayer}
\pgfdeclarelayer{nodelayer}
\pgfsetlayers{edgelayer,nodelayer,main}

\newtheorem{theorem}{Theorem}[section]
\newtheorem{lemma}[theorem]{Lemma}
\newtheorem{proposition}[theorem]{Proposition}
\newtheorem{corollary}[theorem]{Corollary}
\newtheorem{definition}[theorem]{Definition}
\newtheorem{example}[theorem]{Example}

\newcommand{\image}{{\rm{im} }\,}
\newcommand{\coimage}{{\rm{coim} }\,}
\newcommand{\kernel}{{\rm{ker} }\,}
\newcommand{\cokernel}{{\rm{coker} }\,}
\newcommand{\rank}{{\rm{rank} }\,}

\newcounter{repsection}
\newcounter{repeat} \numberwithin{repeat}{repsection}
\newtheorem{reptheorem}[repeat]{Theorem}
\newtheorem{repproposition}[repeat]{Proposition}

\begin{document}

\title{On the Structural Theorem of Persistent Homology}
%\titlerunning{Persistent Homology}        % if too long for running head

%\authorrunning{Meehan, Pavlichenko, and Segert} % if too long for running head

\author{Killian Meehan}
\address{
         Hiraoka Laboratory\\
		KUIAS, Kyoto University\\
		Kyoto, Japan 606-8501} 
\email[K.~Meehan]{killian.f.meehan@gmail.com}
\author{Andrei Pavlichenko}
\address{Mathematics Department\\
         University of Missouri\\
         202 Math Sci Bldg, MO 65211}
\email[A.~Pavlichneko]{adpmbc@mail.missouri.edu}
\author{Jan Segert}
\address{Mathematics Department\\
         University of Missouri\\
         202 Math Sci Bldg, MO 65211}
\email[J.~Segert]{segertj@missouri.edu}

%\date{September 29, 2018}

\maketitle

\begin{abstract}
We study the categorical framework for the computation of persistent homology, 
without reliance on a particular computational algorithm.   
The computation of persistent homology is commonly summarized as a matrix theorem, 
which we  call the  Matrix Structural Theorem.   Any of the various algorithms for computing 
persistent homology constitutes  a {\it constructive} proof  of the Matrix Structural Theorem.  
We show that  the Matrix Structural Theorem is  equivalent to the Krull-Schmidt property 
of the category of filtered chain complexes.   
We separately  establish the Krull-Schmidt property by abstract categorical methods, yielding  a novel  
{\it nonconstructive}  proof of the Matrix Structural Theorem.   
 \par
 	These results provide the foundation for an alternate categorical framework for 
 decomposition in persitent homology, bypassing the usual  persistence vector 
 spaces and quiver representations.   
%\keywords{Persistent Homology  \and Matrix Reduction \and Krull-Schmidt Category}
%\PACS{PACS code1 \and PACS code2 \and more}
 %\subclass{MSC 55U99 \and MSC 18G99}
\end{abstract}

\epigraph{But the power of homology is seldom of much efficacy, except in those happy dispositions where it is almost superfluous.}{\textit{with apologies to Edward Gibbon}}

\tableofcontents

\section{Manifestations of the Structural Theorem}
\label{section1}

\subsection{Introduction}
\label{subsection1.1}
 During the last decade, persistent homology \cite{ELZ,Car} has achieved great success as a  powerful and versatile tool,  
particularly for Topological Data Analysis (TDA) of point clouds.   The term {\it point cloud}  usually means a finite subset of points 
in a Euclidean space ${\mathbb R}^k$, where the dimension $k$ can be large, and the number of points is often very large.      
Many excellent surveys and introductions are available in the literature  
  \cite{Car2,Oudot,EH,Zom,Ghr,Wei}.    {\it Decomposition} plays a central role both in the theory and in the applications of persistent homology.    The ubiquitous ``barcode diagrams" encode  a decomposition  in terms of the 
 types and multiplicities of  indecomposable summands.    This data is an invariant, independent of the choice 
 of decomposition.  
  The summands represented by long barcodes contain important characteristic information, while the summands represented by short barcodes only contain random ``noise" and may be disregarded.    
  A number of ``stability theorems" \cite{CSEH,Oudot} provide a firm foundation for  this intuitively appealing interpretation of the long and short barcode invariants.     
In this paper we consider the interplay between the algorithmic and the categorical underpinnings for decomposition in persistent homology.  

It is helpful to first review  analogous decomposition issues for the much more familiar context of finite-dimensional vector spaces 
(over a fixed field ${\Bbb F}$).    
The ordinary  Gaussian elimination algorithm can construct a basis for a vector space. 
Any choice of basis then constitutes a decomposition of the vector space, 
wherein  the linear span of each basis element is a one-dimensional vector 
space.  The direct sum  of these one-dimensional  summands is canonically identified 
(naturally isomorphic) to the original vector space.   A  
one-dimensional vector space cannot be further decomposed as a  sum of  nonzero (dimensional) summands.   
This means that  one-dimensional vector spaces  are  the {\it indecomposable objects}, in the  category of vector spaces.  Since all one-dimensional vector spaces are mutually isomorphic,  there is just one type of indecomposable (object) 
in the category of vector spaces.  
The familiar  dimension of a vector space is  just the multiplicity  of the indecomposable (one-dimensional) summands in a decomposition, and this multiplicity is an invariant  independent of the choice of decomposition.  

Now setting aside what we know about Gaussian elimination, we ask more abstractly  {\it why} is it that 
any vector space is actually decomposable?   
We first observe that decomposability is a categorical property, since it involves both objects (vector spaces)  and  morphisms (linear maps). 
The theory of {\it Krull-Schmidt categories} \cite{Krause} provides an appropriate, albeit abstract, categorical setting for 
questions of decomposability.   The axioms of 
a Krull-Schmidt category guarantee that every object admits an essentially unique decomposition 
as a finite sum of indecomposable objects.  For the category of vector spaces,      
this essential uniqueness encodes the familiar fact that the dimension is an invariant independent of the choice of decomposition.    
The goal then becomes to verify (and understand) the Krull-Schmidt property for the category of vector spaces.  
A concrete constructive verification of the Krull-Schmidt axioms for the category of vector spaces follows easily 
from basic properties of Gaussian elimination and linearity, but this is more in line with describing {\it how} 
to perform a decomposition rather than {\it why} vector spaces are decomposable.  
Fortunately there is a complementary abstract  tool available.  
A theorem of Atiyah \cite{Atiyah} dating back to the early years of category theory provides 
a very useful criterion for verifying the Krull-Schmidt property of a category.  
For the category of vector spaces, Atiyah's criterion reduces to checking certain 
elementary properties of linear maps.  So Atiyah's theorem 
nonconstructively answers the abstract question of {\it why} any vector space admits a decomposition, 
complementing  our understanding of {\it how} to constructively decompose a given vector space via Gaussian elimination.   

In this paper we consider analogous questions of {\it how} and {\it why}  decomposition works in persistent homology.  
The following picture   
summarizes one common description of the transformation from point cloud data to barcodes invariants: 
\begin{equation*}
\begin{tikzcd}[row sep=large, column sep=1.05cm] 
     \fbox{\begin{tabular}{@{}c@{}}Point\\Cloud\end{tabular}}  
     \ar{r}{Threshold}  & 
     \fbox{\begin{tabular}{@{}c@{}}Filtered\\Simplicial\\Complex\end{tabular}}
     \ar{r}{Boundary} &  
      \fbox{\begin{tabular}{@{}c@{}}Filtered\\Chain\\Complex\end{tabular}}
      \ar{rr}{Homology} & &
       \fbox{\begin{tabular}{@{}c@{}}\bf{Persistence}\\ \bf{Vector}\\ \bf{Space}\end{tabular}} 
      % \ar{d}{Decomposition}
       \ar[d,  "\bf{Decomposition}" description]
       \\         &  &  
      & &
       \fbox{\begin{tabular}{@{}c@{}} \sl{Barcodes}\end{tabular}} 
     \end{tikzcd}
\end{equation*}
The initial stages,  going from a point cloud to a filtered chain complex, will be briefly reviewed in Section 1.2 below.   
The primary focus of this paper will be the final stages, going from filtered chain complexes to barcodes.  
At the homology step, the homology functor $H_n$ of the chosen dimension/degree $n$ 
takes a {\it filtered chain complex} (which is a diagram of chain complexes)  to 
 a {\it persistence vector space} (which is a diagram of vector spaces).   
The key final step is to compute  barcode invariants by  decomposition of a persistence vector space. 
An important insight \cite{Car} is that 
persistence vector spaces are quiver representations.      
A concrete consequence is the applicability of decomposition algorithms from quiver representation theory, showing 
{ how} to decompose a persistence vector space and compute the barcodes. 
An abstract consequence is that the appropriate category of quiver representations is Krull-Schmidt 
by Atiyah's theorem, showing {\it why} all of this works.   So the Krull-Schmidt property of persistence 
vector spaces nicely ties together the theoretical and computational aspects.    
But there is one problem with this picture.

The standard computational algorithms for persistent homology \cite{ELZ,ZC1,ZC2} do {\it not} 
 work by decomposing a persistence vector space. 
The following picture  summarizes {\it how}  barcodes are normally computed:  
\begin{equation*}
\begin{tikzcd}[row sep=large, column sep=1.05cm] 
     \fbox{\begin{tabular}{@{}c@{}}Point\\Cloud\end{tabular}}  
     \ar{r}{Threshold}  & 
     \fbox{\begin{tabular}{@{}c@{}}Filtered\\Simplicial\\Complex\end{tabular}}
     \ar{r}{Boundary} &  
      \fbox{\begin{tabular}{@{}c@{}}Filtered\\Chain\\Complex\end{tabular}}
      %\ar{d}[swap]{Decomposition} 
       \ar[d,  "Reduction" description]
      & 
         \\         &  &   \fbox{\begin{tabular}{@{}c@{}}Creators\\and\\Destroyers\end{tabular}}
         \ar{rr}[swap]{Selection} & &
       \fbox{\begin{tabular}{@{}c@{}}\sl{Barcodes}\end{tabular}} 
     \end{tikzcd}
\end{equation*}
The initial stages of the picture, going from a point cloud to a filtered chain complex, are unchanged.  
The key {\it reduction} step  \cite{ELZ,ZC1,ZC2} is the construction of a special type of basis. 
Each  basis element is interpreted as either 
 a {\it creator}  or as a   {\it destroyer} of a homology class. 
The {\it selection} step consists of keeping those creators and destroyers that correspond 
to nonzero barcodes of the desired homology dimension/degree $n$, and 
discarding  the remaining basis elements.    The question remains of {\it why} there should 
exist such algorithms operating 
on filtered complexes, rather than on persistence vector spaces.   

In this paper we provide a categorical framework for the 
standard  persistent homology algorithms, using an equivalence of categories to unify the two pictures above: 
\begin{equation*}
\begin{tikzcd}[row sep=large, column sep=1.05cm] 
     \fbox{\begin{tabular}{@{}c@{}}Point\\Cloud\end{tabular}}  
     \ar{r}{Threshold}  & 
     \fbox{\begin{tabular}{@{}c@{}}Filtered\\Simplicial\\Complex\end{tabular}}
     \ar{r}{Boundary} &  
      \fbox{\begin{tabular}{@{}c@{}} \bf{Filtered}\\ \bf{Chain}\\ \bf{Complex}\end{tabular}}
     % \ar{d}[swap]{Congruence} 
       \ar[d,  "Congruence" description]
        \ar{rr}{Homology} && 
       \fbox{\begin{tabular}{@{}c@{}} \bf{Persistence}\\ \bf{Vector}\\ \bf{Space}\end{tabular}} 
     %  \ar{d}{Decomposition}
        \ar[d,  "\bf{Decomposition}" description]
       \ar[dll,  swap, dashed, leftrightarrow, "\bf{Equivalence}" description]
      & 
         \\         &  &   \fbox{\begin{tabular}{@{}c@{}} \bf{Quotient}\\ \bf{Object}\end{tabular}}
         \ar{rr}[swap]{\bf{Decomposition}} & &
       \fbox{\begin{tabular}{@{}c@{}} \sl{Barcodes}\end{tabular}} 
     \end{tikzcd}
\end{equation*}
Our  {\it Categorical Structural Theorem} (Theorem 1.6) is the foundation of the framework.  The theorem 
 asserts that the category of filtered chain complexes is Krull-Schmidt, and provides an intuitive classification of indecomposables.    This leads to an 
 alternate framework for persistent homology, where the barcodes   
 describe the Krull-Schmidt decomposition of an object in a quotient of the category of filtered chain complexes. 
 The barcodes are exactly the same as in the standard framework, because the quotient category is 
 equivalent to the category of persistence vector spaces.  
 This framework gives a unified answer for {\it why} and {\it how} decomposition actually works in  
persistent homology.    We no longer need to rely on the Krull-Schmidt property 
 of the category of persistence vector spaces as an indirect theoretical foundation for decomposition and barcodes,  since 
 we can directly appeal to the Krull-Schmidt property already in the category of filtered chain complexes.

 The Categorical Structural Theorem is the abstract  version of what we  call the 
 {\it Structural Theorem of Persistent Homology}.   
We give a nonconstructive categorical proof of the Categorical Structural Theorem, indirectly using Atiyah's criterion.
Combining the Categorical Sructural Theorem with a classification of 
indecomposable filtered chain complexes then yields a novel {\it nonconstructive} proof of what 
we call the  {\it Matrix Structural Theorem} (Theorem 1.4).  
The Matrix Structural Theorem characterizes 
the output of any of the various standard persistent homology algorithms 
in terms of a matrix factorization rather than the more common 
description in terms  of creators and destroyers for homology.    
In this sense, any of the standard algorithms can be thought of as constituting a 
   {\it constructive} proof of the Matrix Structural Theorem.  
In Appendices B.1 and B.2 we present a detailed mathematical treatment of the matrix reduction approach 
to the  Matrix Structural Theorem.

\subsection{Topological Data Analysis by Example}
\label{prep}
 
 This paper focuses on the final stages of Topological Data Analysis (TDA), going from a filtered chain complex 
 to barcode invariants.  
In this section we present  a simple example to illustrate the stages leading up to the Structural Theorem,  
namely going from a point cloud to a filtered simplicial complex.   A reader familiar with TDA may  skip this section, 
which is similar to material in introductory papers such as \cite{Car,CSEH,Ghr} and textbooks such as \cite{EH,Zom}.     
In our example, we 
use the $\alpha$-complex construction \cite{EH,Ed}, which is suitable for low dimensions.  
We note that for large point clouds in high dimensions,  
the Vietoris-Rips construction \cite{Car2,Oudot}  is often preferable.

\begin{example} The first step is to construct a  Delaunay complex, the second step is to construct a filtration 
of the Delaunay complex.     
%
% Figures Lablevel1.png and Lablevel4a.png
\begin{figure}
 \includegraphics[width=0.45\textwidth]{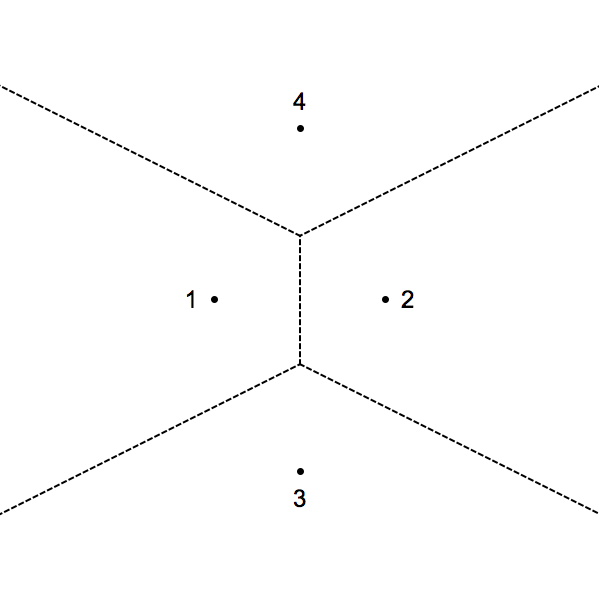}
\caption{A point cloud with its Voronoi cells.}
\label{fig:1}       % Give a unique label
\end{figure}
\begin{figure}
 \includegraphics[width=0.45\textwidth]{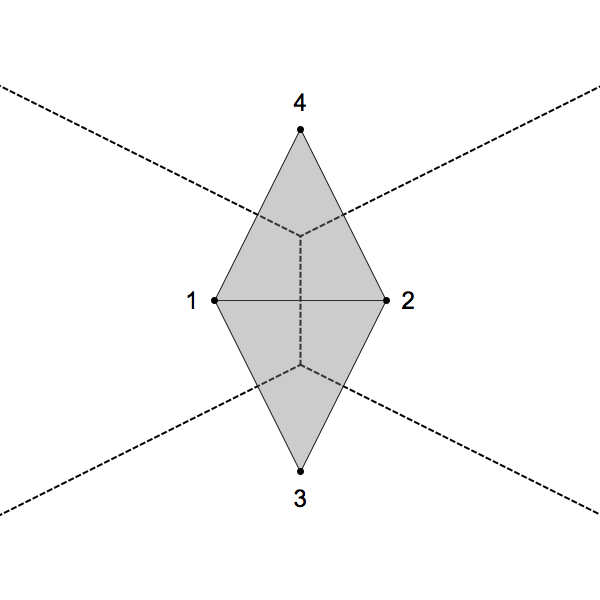}
\caption{Delaunay  Complex. }
\label{fig:4a}       % Give a unique label
\end{figure}
We illustrate the construction of the Delaunay simplicial complex associated to a point cloud.
Figure \ref{fig:1} shows a point cloud consisting of four  points in  in ${\Bbb R^2}$ labeled by $n \in \{1,2,3,4 \} $, together 
with the Voronoi cell $V(n)$ of each labeled point.              
We recall \cite{EH} that a Voronoi cell $V(n)$ contains all the   
points $x \in {\Bbb R^2}$  such that $n$ is the closest labeled point to $x$ (or one of the closest if several are equidistant).  
Figure  \ref{fig:4a} shows the Delaunay simplicial complex encoding the intersections of the Voronoi cells.   
We recall that the simplex    
 $[n_0,\dots,n_k]$, where $n_i \in \{ 1,2,3,4 \}$ and  $n_0 <  \dots < n_k$, is 
included in the Delaunay complex iff $V(n_0) \cap \cdots \cap V(n_k) \ne \emptyset$.  
For example, the simplex $[1,2]$ is included because $V(1) \cap V(2) \ne \emptyset$, but 
the simplex $[3,4]$ is not included because $V(3) \cap V(4) = \emptyset$.

The $\alpha$ construction assigns to each Delaunay simplex  $[n_0,\dots,n_k]$  a 
 real nonnegative ``birth parameter" $b([n_0,\dots,n_k])$.    
Let $B_r(n)$ denote the closed ball of radius $r$ centered at the labeled point $n$, 
and consider the subset  $A_r(n )= B_r(n) \cap V(n)$  of the Voronoi cell $V(n)$. 
The birth parameter of the Delaunay simplex $[n_0,\dots,n_k]$ is defined to be 
the smallest value of $r$ such that $A_r(n_0) \cap \cdots \cap A_r(n_k) \ne \emptyset$. 
A value of $r$ is called a``threshold' if it is the birth parameter for some Delaunay simplex.  
The integer ``level" $p$ indexes the thresholds in increasing order, 
as illustrated in Figures \ref{fig:1rep} through \ref{fig:4}:  
\begin{figure}
 \includegraphics[width=0.45\textwidth]{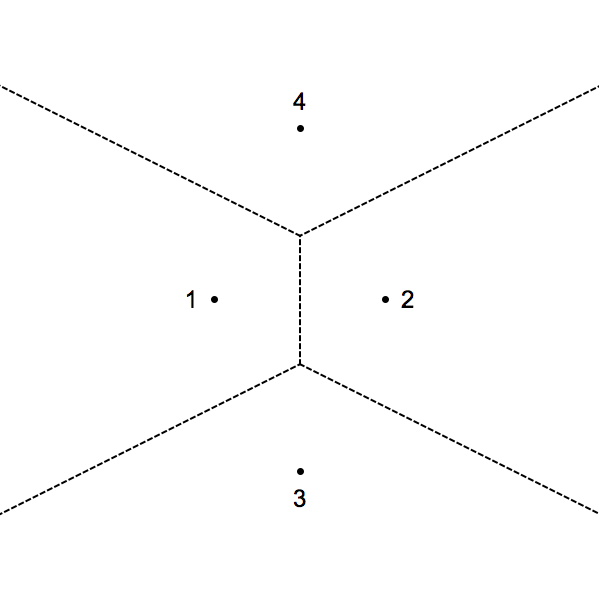}
\caption{Level $p =1$ is the threshold $r= 0 = b([1])=b([2]) =b([3]) = b([4])$. }
\label{fig:1rep}       % Give a unique label
\end{figure}
\begin{figure}
 \includegraphics[width=0.45\textwidth]{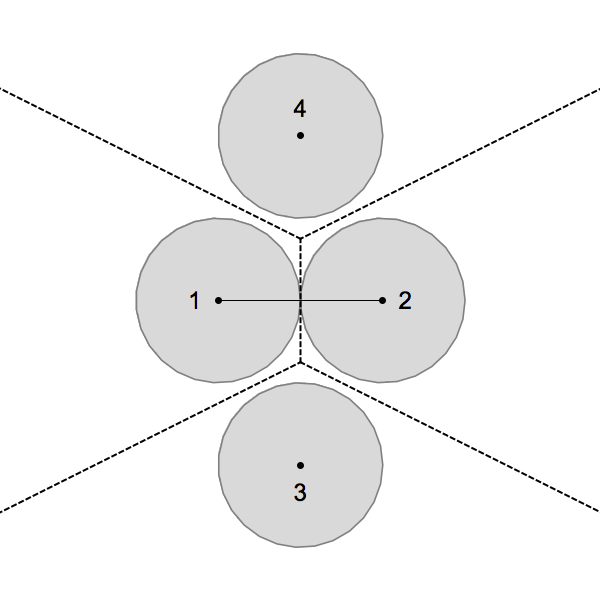}
\caption{Level $p = 2$ is the threshold $r=1 = b([1,2])$. }
\label{fig:2}       % Give a unique label
\end{figure}
\begin{figure}
 \includegraphics[width=0.45\textwidth]{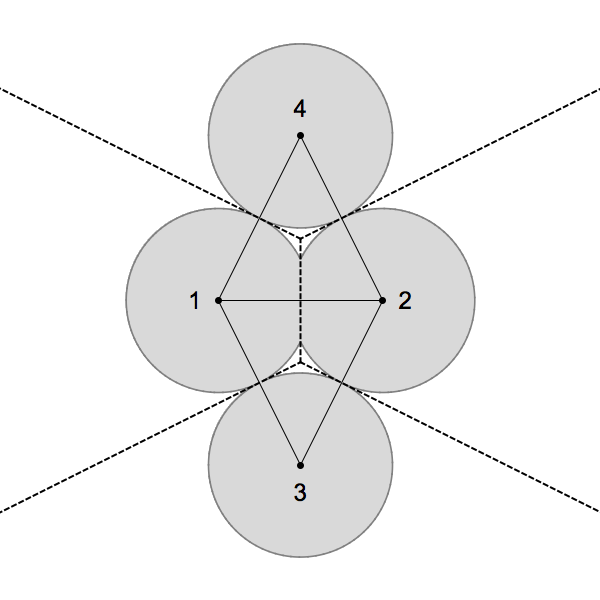}
\caption{Level $p =3$ is the threshold  $r =1.12 = b([1,3]) = b([1,4])= b([2,3]) = b([2,4])$. }
\label{fig:3}       % Give a unique label
\end{figure}
\begin{figure}
 \includegraphics[width=0.45\textwidth]{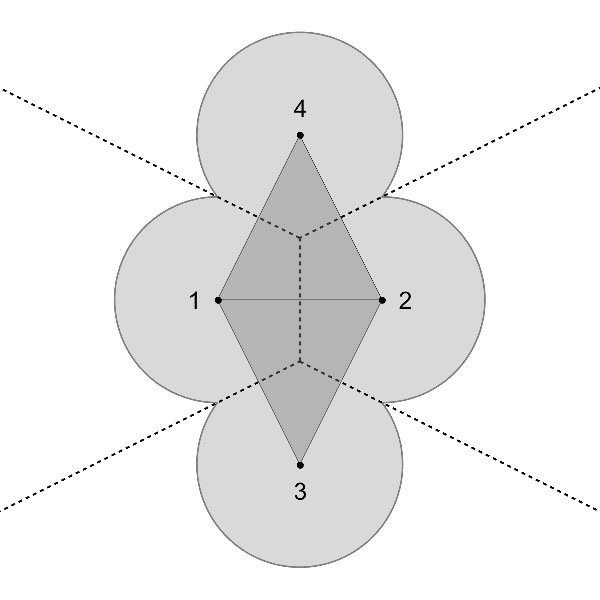}
\caption{Level $p = 4$ is the threshold $r= 1.25 = b([1,2,3]) = b([1,2,4])$. }
\label{fig:4}       % Give a unique label
\end{figure}

The $\alpha$ construction produces a filtration of the Delaunay complex,  
and the simplicial 
homology \cite{Hatcher} of this filtered complex is described in terms of the barcode invariants \cite{Oudot,EH,Zom}.
Conventionally the filtration and the corresponding  barcodes are indexed  by the real-valued threshold parameter $r$, 
which for our example yields the $H_1$ barcode diagram of Figure \ref{fig:bc1}. 
%
% figure bc1.png
\begin{figure*}
  \includegraphics[width=0.8\textwidth]{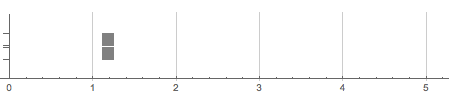}
\caption{$H_1$ barcodes indexed by the real threshold parameter $r$.}
\label{fig:bc1}       % Give a unique label
\end{figure*}
The diagram indicates that the first homology $H_1$ detects two one-dimensional ``holes" that appear  
at $r = 1.12$ and are filled in at $r = 1.25$. 
In this paper we will index filtrations and the corresponding barcodes by the integer-valued level $p$, which 
for our example yields the $H_1$ barcode diagram of Figure \ref{fig:bc2}.  
%
% figure bc2.png
\begin{figure*}
  \includegraphics[width=0.8\textwidth]{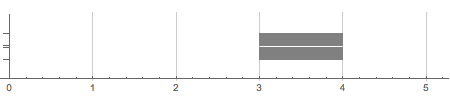}
\caption{$H_1$ barcodes indexed by the integer  level $p$.}
\label{fig:bc2}       % Give a unique label
\end{figure*}
This diagram indicates the same information, namely that the first homology $H_1$ detects two one-dimensional ``holes" that appear  
at $p = 3$ (which corresponds to $r = 1.12$) and are filled in at $p =4$ (which corresponds to $r = 1.25$).    
\end{example}

\subsection{Matrix Structural Theorem} 
\label{subsection1.2}

For simplicity, we start with the ungraded version of the structural theorem.  
 A  {\it differential matrix} is a square matrix $D$ satisfying $D^2 = 0$.   
 We'll say a differential matrix is {\it Jordan} if it is in Jordan normal form,  meaning it decomposes as 
 a block-diagonal matrix built from copies of the two differential Jordan block matrices  
\begin{equation*}
J =   
%************** start matrix ***********
 \begin{tikzpicture}[baseline={([yshift=-0.5ex]current bounding box.center)}]

[decoration=brace]
\tikzset{
    node style ge/.style={circle,minimum size=.75cm},
}
\pgfdeclarelayer{background}
\pgfdeclarelayer{foreground}
\pgfsetlayers{background,main,foreground}
\matrix (A) [matrix of math nodes,
             nodes = {node style ge},
             left delimiter  = {[}, %brackets
             right delimiter = {]}, %brackets
            inner sep=-2pt,
		row sep=-.3cm,
		column sep=-.25cm
             ]
{
%MATRIX ENTRIES
0\\
};

\end{tikzpicture}
%************** end matrix ***********%, 
, \qquad K = 
%************** start matrix ***********
 \begin{tikzpicture}[baseline={([yshift=-0.5ex]current bounding box.center)}]

[decoration=brace]
\tikzset{
    node style ge/.style={circle,minimum size=.75cm},
}
\pgfdeclarelayer{background}
\pgfdeclarelayer{foreground}
\pgfsetlayers{background,main,foreground}
\matrix (A) [matrix of math nodes,
             nodes = {node style ge},
             left delimiter  = {[}, %brackets
             right delimiter = {]}, %brackets
            inner sep=-2pt,
		row sep=-.3cm,
		column sep=-.25cm
             ]
{
%MATRIX ENTRIES
0 & 1 \\
0 & 0 \\
};

\end{tikzpicture}
%************** end matrix ***********%
.
\end{equation*}
We'll say a differential matrix $\underline D$ is {\it almost-Jordan} if there exists a permutation matrix $P$ such that 
the differential matrix $P^{-1} \underline D P$ is Jordan.   Given an almost-Jordan differential matrix $\underline D$, 
it is trivial to construct such a permutation matrix $P$.   
We will say a square matrix $B$ is {\it triangular} if it is upper-triangular and invertible.

The standard algorithm for computing persistent homology is based on the papers \cite{ELZ,ZC1,ZC2}.    
The result of a persistent homology computation, not depending on a choice of  algorithm, is 
conveniently summarized   \cite{DMV,Oudot} as a matrix factorization:

\begin{theorem} 
\label{prenon}
 (Ungraded Matrix  Structural Theorem) 
 Any differential matrix $D$ factors as $D = B {\underline D} B^{-1}$ where  
${\underline D}$ is an almost-Jordan differential matrix  and $B$ is a triangular matrix.   
\end{theorem}

\noindent  It is the triangular condition that makes this interesting: without the triangular condition, 
this would follow immediately from the ordinary Jordan normal form.
Furthermore, the matrix ${\underline D}$ is unique, as we show in Appendix A.
   We'll call  ${\underline D}$  the {\it persistence canonical form} of the differential matrix $D$.  
A column of  the triangular matrix $B$ is in $\ker D$ iff  the corresponding column of ${\underline D}$ is zero. 
We will say that $B$ is {\it normalized} if each such column has diagonal entry equal to $1$.  
It is always possible to  normalize  $B$ by scalar multiplication of columns, but even with normalization  
$B$ is not unique in general. 
A constructive  proof of Theorem \ref{prenon} follows from any of the algorithms for computing persistent homology.        
In Appendices B.1 and B.2 we discuss the matrix reduction approach to computing persistent homology.  

\begin{example} 
\label{explain}
Consider the filtered simplicial complex shown in Figure \ref{fig:sc1}. 
%
% figure simplicial1.pdf
\begin{figure*}
  \includegraphics[width=\textwidth]{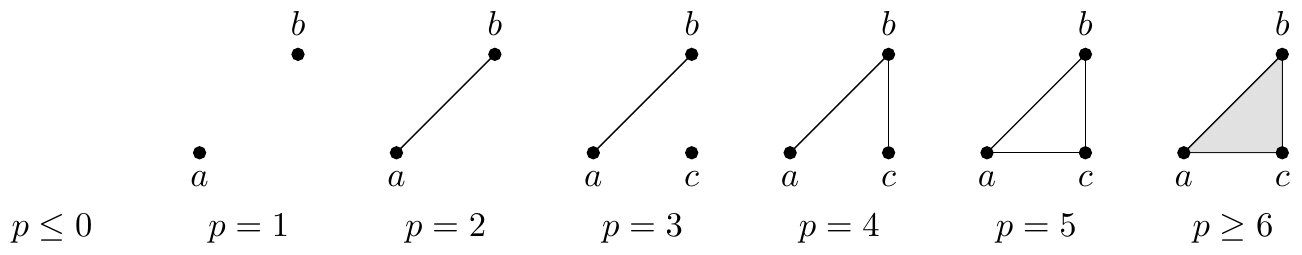}
\caption{A filtered Simplicial Complex.}
\label{fig:sc1}       % Give a unique label
\end{figure*}
With the usual  convention for an adapted basis,   
the ordering of basis elements prioritizes  the level of the filtration over  the degree/dimension of the simplex.   
The initial basis of simplices is  then ordered  so the  level (denoted by prescript) is nondecreasing, and within each  level the 
degree (denoted by postscript) is nondecreasing.   Using lexicographic order to break any remaining ties, the initial adapted 
basis is   $ \tensor[_1]{a}{_0}, \tensor[_1]{b}{_0}, \tensor[_2]{ab}{_1}, \tensor[_3]{c}{_0}, 
\tensor[_4]{bc}{_1},  \tensor[_5]{ac}{_1}, \tensor[_6]{abc}{_2}$, and the boundary operator over the 
field ${\mathbb F} = {\mathbb Q}$ of rationals 
is   represented by the differential matrix
%
%************** start matrix equation ***********
%
\begin{equation*}
D = 
 \begin{tikzpicture}[baseline={([yshift=-1.4ex]current bounding box.center)}]
[decoration=brace]
\tikzset{
    node style ge/.style={circle,minimum size=.75cm},
}
\pgfdeclarelayer{background}
\pgfdeclarelayer{foreground}
\pgfsetlayers{background,main,foreground}
\matrix (A) [matrix of math nodes,
             nodes = {node style ge},
             left delimiter  = {[}, %brackets
             right delimiter = {]}, %brackets
             inner sep=-2pt, %this determines how far the matrix brackets scale outward around the entries.
             row sep=-.3cm, %custom row separation
             column sep=0cm %custom column separation (currently doing nothing)
             ]
{
%MATRIX ENTRIES
0 & 0 & -1 & 0 & 0 & -1 & 0 \\
0 & 0 & 1 & 0 & -1 & 0 & 0 \\
0 & 0 & 0 & 0 & 0 & 0 & 1 \\
0 & 0 & 0 & 0 & 1 & 1 & 0 \\
0 & 0 & 0 & 0 & 0 & 0 & 1 \\
0 & 0 & 0 & 0 & 0 & 0 & -1 \\
0 & 0 & 0 & 0 & 0 & 0 & 0 \\
};
% BASIS across the top. The "above=-1pt" can be edited to bring them lower or higher.
\node[above=-1pt] at (A-1-1.north) {$\tensor[_1]{a}{_0}$};
\node[above=-1pt] at (A-1-2.north) {$\tensor[_1]{b}{_0}$};
\node[above=-1pt] at (A-1-3.north) {$\tensor[_2]{ab}{_1}$};
\node[above=-1pt] at (A-1-4.north) {$\tensor[_3]{c}{_0}$};
\node[above=-1pt] at (A-1-5.north) {$\tensor[_4]{bc}{_1}$};
\node[above=-1pt] at (A-1-6.north) {$\tensor[_5]{ac}{_1}$};
\node[above=-1pt] at (A-1-7.north) {$\tensor[_6]{abc}{_2}$};
%BASIS down the side. The "left=6pt" determines how far out they are.
\node[left=6pt,yshift=-.05cm] at (A-1-1.west) {$\tensor[_1]{a}{_0} $};
\node[left=6pt,yshift=-.05cm] at (A-2-1.west) {$\tensor[_1]{b}{_0} $};
\node[left=6pt,yshift=-.05cm] at (A-3-1.west) {$\tensor[_2]{ab}{_1}$};
\node[left=6pt,yshift=-.05cm] at (A-4-1.west) {$ \tensor[_3]{c}{_0}$};
\node[left=6pt,yshift=-.05cm] at (A-5-1.west) {$\tensor[_4]{bc}{_1} $};
\node[left=6pt,yshift=-.05cm] at (A-6-1.west) {$\tensor[_5]{ac}{_1}$};
\node[left=6pt,yshift=-.05cm] at (A-7-1.west) {$\tensor[_6]{abc}{_2}$};

\end{tikzpicture}
.
\end{equation*}

%************** end matrix equation***********
%
\noindent The persistence canonical form is 
%***  begin matrix equation
\begin{equation*} 
{\underline D} = 
 \begin{tikzpicture}[baseline={([yshift=-1.4ex]current bounding box.center)}]
[decoration=brace]
\tikzset{
    node style ge/.style={circle,minimum size=.75cm},
}
\pgfdeclarelayer{background}
\pgfdeclarelayer{foreground}
\pgfsetlayers{background,main,foreground}
\matrix (A) [matrix of math nodes,
             nodes = {node style ge},
             left delimiter  = {[}, %brackets
             right delimiter = {]}, %brackets
             inner sep=-2pt, %this determines how far the matrix brackets scale outward around the entries.
             row sep=-.3cm, %custom row separation
             column sep=0cm %custom column separation (currently doing nothing)
             ]
{
%MATRIX ENTRIES
0 & 0 & 0 & 0 & 0 & 0 & 0 \\
0 &  { 0} & { 1} & 0 & 0 & 0 & 0 \\
0 & 0 & 0 & 0 & 0 & 0 & 0 \\
0 & 0 & 0 &  { 0} & { 1} & 0 & 0 \\
0 & 0 & 0 & 0 &  0 & 0 & 0 \\
0 & 0 & 0 & 0 & 0 &  { 0} & { 1} \\
0 & 0 & 0 & 0 & 0 & 0 &  0 \\
};
% BASIS across the top. The "above=-1pt" can be edited to bring them lower or higher.
\node[above=-1pt] at (A-1-1.north) {$\tensor[_1]{\underline{a}}{_0} $};
\node[above=-1pt] at (A-1-2.north) {$\tensor[_1]{\underline{b}}{_0} $};
\node[above=-1pt] at (A-1-3.north) {$ \tensor[_2]{\underline{ab}}{_1} $};
\node[above=-1pt] at (A-1-4.north) {$ \tensor[_3]{\underline{c}}{_0}$};
\node[above=-1pt] at (A-1-5.north) {$\tensor[_4]{\underline{bc}}{_1} $};
\node[above=-1pt] at (A-1-6.north) {$\tensor[_5]{\underline{ac}}{_1} $};
\node[above=-1pt] at (A-1-7.north) {$\tensor[_6]{\underline{abc}}{_2}  $};
%BASIS down the side. The "left=6pt" determines how far out they are.
\node[left=6pt,yshift=-.05cm] at (A-1-1.west) {$\tensor[_1]{\underline{a}}{_0} $};
\node[left=6pt,yshift=-.05cm] at (A-2-1.west) {$\tensor[_1]{\underline{b}}{_0} $};
\node[left=6pt,yshift=-.05cm] at (A-3-1.west) {$\tensor[_2]{\underline{ab}}{_1}$};
\node[left=6pt,yshift=-.05cm] at (A-4-1.west) {$\tensor[_3]{\underline{c}}{_0} $};
\node[left=6pt,yshift=-.05cm] at (A-5-1.west) {$\tensor[_4]{\underline{bc}}{_1} $};
\node[left=6pt,yshift=-.05cm] at (A-6-1.west) {$\tensor[_5]{\underline{ac}}{_1}$};
\node[left=6pt,yshift=-.05cm] at (A-7-1.west) {$\tensor[_6]{\underline{abc}}{_2} $};

\end{tikzpicture}
,
\end{equation*}
\noindent   
as verified by checking that $\underline D =  B^{-1} D B$ for the triangular (and normalized) matrix  
\begin{equation*}
B = 
  \begin{tikzpicture}[baseline={([yshift=-1.4ex]current bounding box.center)}]

[decoration=brace]

\tikzset{
    node style ge/.style={circle,minimum size=.75cm},
}
\pgfdeclarelayer{background}
\pgfdeclarelayer{foreground}
\pgfsetlayers{background,main,foreground}
\matrix (A) [matrix of math nodes,
             nodes = {node style ge},
             left delimiter  = {[}, %brackets
             right delimiter = {]}, %brackets
             inner sep=-2pt, %this determines how far the matrix brackets scale outward around the entries.
             row sep=-.3cm, %custom row separation
             column sep=0cm %custom column separation (currently doing nothing)
             ]
{
%MATRIX ENTRIES
1 & -1 & 0 & -1 & 0 & 0 & 0 \\
0 & 1 & 0 & 0 & 0 & 0 & 0 \\
0 & 0 & 1 & 0 & 1 & -1 & 0 \\
0 & 0 & 0 & 1 & 0 & 0 & 0 \\
0 & 0 & 0 & 0 & 1 & -1 & 0 \\
0 & 0 & 0 & 0 & 0 & 1 & 0 \\
0 & 0 & 0 & 0 & 0 & 0 & -1 \\
};
% BASIS across the top. The "above=-1pt" can be edited to bring them lower or higher.
\node[above=-1pt] at (A-1-1.north) {$\tensor[_1]{\underline{a}}{_0}  $};
\node[above=-1pt] at (A-1-2.north) {$\tensor[_1]{\underline{b}}{_0} $};
\node[above=-1pt] at (A-1-3.north) {$\tensor[_2]{\underline{ab}}{_1}$};
\node[above=-1pt] at (A-1-4.north) {$ \tensor[_3]{\underline{c}}{_0} $};
\node[above=-1pt] at (A-1-5.north) {$\tensor[_4]{\underline{bc}}{_1} $};
\node[above=-1pt] at (A-1-6.north) {$\tensor[_5]{\underline{ac}}{_1} $};
\node[above=-1pt] at (A-1-7.north) {$\tensor[_6]{\underline{abc}}{_2} $};
%BASIS down the side. The "left=6pt" determines how far out they are.
\node[left=6pt,yshift=-.05cm] at (A-1-1.west) {$\tensor[_1]{a}{_0}  $};
\node[left=6pt,yshift=-.05cm] at (A-2-1.west) {$\tensor[_1]{b}{_0} $};
\node[left=6pt,yshift=-.05cm] at (A-3-1.west) {$\tensor[_2]{ab}{_1}$};
\node[left=6pt,yshift=-.05cm] at (A-4-1.west) {$\tensor[_3]{c}{_0} $};
\node[left=6pt,yshift=-.05cm] at (A-5-1.west) {$\tensor[_4]{bc}{_1} $};
\node[left=6pt,yshift=-.05cm] at (A-6-1.west) {$\tensor[_5]{ac}{_1} $};
\node[left=6pt,yshift=-.05cm] at (A-7-1.west) {$\tensor[_6]{abc}{_2} $};

\end{tikzpicture}
.
\end{equation*}
\noindent   
The persistence canonical form ${\underline D}$ is  almost-Jordan in general,  and in 
this example it happens to be actually Jordan.    
The  matrix $B$ represents the basis change to the new adapted basis  
$\tensor[_1]{\underline{a}}{_0},  \tensor[_1]{\underline{b}}{_0}, \tensor[_2]{\underline{ab}}{_1}, \tensor[_3]{\underline{c}}{_0},
 \tensor[_4]{\underline{bc}}{_1},  \tensor[_5]{\underline{ac}}{_1},  \tensor[_6]{\underline{abc}}{_2}$.  
 The  level  remains nondecreasing because $B$ is triangular.  
Each basis element retains pure degree, 
although  Theorem \ref{prenon} does not explicitly address issues of degree.  
 The matrix ${\underline D}$ represents 
the boundary operator  relative to the new adapted basis.   
\end{example}

We prefer to prioritize degree over   level in ordering the elements of an adapted basis.      
This has the advantage of encoding the degree in the  block structure of the matrix.  
The following version of the structural theorem is then manifestly compatible with the grading by degree:  

\begin{theorem} 
\label{postnon}
(Matrix  Structural Theorem) 
 Any block-superdiagonal differential matrix $D$ factors as $D = B {\underline D} B^{-1}$ where  
${\underline D}$ is a block-superdiagonal almost-Jordan differential matrix  and $B$ is a block-diagonal triangular matrix.   
\end{theorem} 

\noindent The block-diagonal structure of $B$ ensures that the transformed basis elements retain pure degree. 
The persistence canonical form ${\underline D}$  inherits the block-superdiagonal structure 
of the differential $D$.    
It is always possible to normalize $B$ by scalar multiplication of columns as in the ungraded case.  
Any of the  algorithmic proofs of Theorem \ref{prenon}  \cite{ELZ,ZC1,ZC2}  
can be used to prove Theorem \ref{postnon} by keeping track of degrees.  
We discuss this point for the standard algorithm in Appendix B.2.

\begin{example} 
\label{exfilt}
We again consider the filtered chain complex of Example \ref{explain}, but    
with  basis order prioritizing   degree  over   level.   
Now the degree of basis elements (denoted by postscript) is nondecreasing, and within a degree the  level 
(denoted by prescript) of basis elements is nondecreasing.  
Using lexicographic order to break any remaining ties, the initial 
adapted basis is now   $ \tensor[_1]{a}{_0}, \tensor[_1]{b}{_0}, \tensor[_3]{c}{_0}, \tensor[_2]{ab}{_1}, \tensor[_4]{bc}{_1}, \tensor[_5]{ac}{_1},
 \tensor[_6]{abc}{_2}$, and the boundary operator over the 
field ${\mathbb F} = {\mathbb Q}$ of rationals 
is  now represented by the block-superdiagonal differential matrix
%
%************** start matrix equation ***********
%
\begin{equation*}
D = 
 \begin{tikzpicture}[baseline={([yshift=-1.4ex]current bounding box.center)}]
[decoration=brace]
\tikzset{
    node style ge/.style={circle,minimum size=.75cm},
}
\pgfdeclarelayer{background}
\pgfdeclarelayer{foreground}
\pgfsetlayers{background,main,foreground}
\matrix (A) [matrix of math nodes,
             nodes = {node style ge},
             left delimiter  = {[}, %brackets
             right delimiter = {]}, %brackets
             inner sep=-2pt, %this determines how far the matrix brackets scale outward around the entries.
             row sep=-.3cm, %custom row separation
             column sep=0cm %custom column separation (currently doing nothing)
             ]
{
%MATRIX ENTRIES
0 & 0 &  0 & {-1} & 0 &{-1}  & 0 \\
0 & 0 & 0 & 1 & -1 & 0 & 0 \\
0 & 0 &  0 &  0 & 1 &   1 &  0 \\
0 & 0 & 0 & 0 & 0 & 0 &1 \\
0 & 0 & 0 & 0 & 0 & 0 & 1 \\
0 & 0 & 0 & 0 & 0 & 0 & {-1} \\
0 & 0 & 0 & 0 & 0 & 0 & 0 \\
};
% BASIS across the top. The "above=-1pt" can be edited to bring them lower or higher.
\node[above=-1pt] at (A-1-1.north) {$\tensor[_1]{a}{_0}$};
\node[above=-1pt] at (A-1-2.north) {$\tensor[_1]{b}{_0}$};
\node[above=-1pt] at (A-1-3.north) {$\tensor[_3]{c}{_0}$};
\node[above=-1pt] at (A-1-4.north) {$\tensor[_2]{ab}{_1}$};
\node[above=-1pt] at (A-1-5.north) {$\tensor[_4]{bc}{_1}$};
\node[above=-1pt] at (A-1-6.north) {$\tensor[_5]{ac}{_1}$};
\node[above=-1pt] at (A-1-7.north) {$\tensor[_6]{abc}{_2}$};
%BASIS down the side. The "left=6pt" determines how far out they are.
\node[left=6pt,yshift=-.05cm] at (A-1-1.west) {$\tensor[_1]{a}{_0} $};
\node[left=6pt,yshift=-.05cm] at (A-2-1.west) {$\tensor[_1]{b}{_0} $};
\node[left=6pt,yshift=-.05cm] at (A-3-1.west) {$ \tensor[_3]{c}{_0}$};
\node[left=6pt,yshift=-.05cm] at (A-4-1.west) {$\tensor[_2]{ab}{_1}$};
\node[left=6pt,yshift=-.05cm] at (A-5-1.west) {$\tensor[_4]{bc}{_1} $};
\node[left=6pt,yshift=-.05cm] at (A-6-1.west) {$\tensor[_5]{ac}{_1}$};
\node[left=6pt,yshift=-.05cm] at (A-7-1.west) {$\tensor[_6]{abc}{_2}$};
%These are ANCHORS, invisible nodes in the middle/borders of the matrix that I used to draw the lines and make the shading.
%Just for reference, the "!0.5!" stuff is used to find halfway points between nodes to ensure that the lines are drawn exactly down the middle of two adjacent columns.

%here are the two anchors across the top
\node (top1) at ($(A-1-3.north)!0.5!(A-1-4.north)$) {};
\node (top2) at ($(A-1-6.north)!0.5!(A-1-7.north)$) {};

%anchors across the bottom
\node (bot1) at ($(A-7-3.south)!0.5!(A-7-4.south)$) {};
\node (bot2) at ($(A-7-6.south)!0.5!(A-7-7.south)$) {};

%anchors down the left
\node (left1) at ($(A-3-1.south west)!0.5!(A-4-1.north west)-(.15,0)$) {};
\node (left2) at ($(A-6-1.south west)!0.5!(A-7-1.north west)-(.15,0)$) {};

%anchors down the right
\node (right1) at ($(A-3-7.south east)!0.5!(A-4-7.north east)+(.15,0)$) {};
\node (right2) at ($(A-6-7.south east)!0.5!(A-7-7.north east)+(.15,0)$) {};

%Anchors in the middle of the matrix take two steps, as somehow averaging the center of 4 points at once didn't work.
%So these are just temporary nodes to be used below.
\node (a_1) at ($(A-3-3.south east)!0.5!(A-3-4.south west)$) {};
\node (a_2) at ($(A-4-3.north east)!0.5!(A-4-4.north west)$) {};
\node (b_1) at ($(A-3-6.south east)!0.5!(A-3-7.south west)$) {};
\node (b_2) at ($(A-4-6.north east)!0.5!(A-4-7.north west)$) {};

%the top two "intersection points"
\node (mid11) at ($(a_1)!0.5!(a_2)$) {};
\node (mid12) at ($(b_1)!0.5!(b_2)$) {};

\iffalse

\node (x_1) at ($(A -    i     - (j+1).south west)!0.5!(A-    i     - j.south east)$) {};
\node (x_2) at ($(A - (i+1) - (j+1).north west)!0.5!(A- (i+1) - j.north east)$) {};

\node (mid"ij") at ($(x_1)!0.5!(x_2)$) {};

\fi

%Horizontal and vertical lines

\draw (left1) -- (right1);
\draw (left2) -- (right2);
\draw (top1) -- (bot1);
\draw (top2) -- (bot2);

\begin{pgfonlayer}{background}

%Shaded blocks in a background layer so that they don't cover lines, entries, etc

\fill[black!20!white] (mid11) rectangle (top2.south);
\fill[black!20!white] (mid12) rectangle (right2.west);

\end{pgfonlayer}

\end{tikzpicture}
.
\end{equation*}

%************** end matrix equation***********

\noindent The persistence canonical form inherits the block-superdiagonal structure  
%
%***  begin matrix equation
\begin{equation*} 
{\underline D} = 
 \begin{tikzpicture}[baseline={([yshift=-1.4ex]current bounding box.center)}]
[decoration=brace]
\tikzset{
    node style ge/.style={circle,minimum size=.75cm},
}
\pgfdeclarelayer{background}
\pgfdeclarelayer{foreground}
\pgfsetlayers{background,main,foreground}
\matrix (A) [matrix of math nodes,
             nodes = {node style ge},
             left delimiter  = {[}, %brackets
             right delimiter = {]}, %brackets
             inner sep=-2pt, %this determines how far the matrix brackets scale outward around the entries.
             row sep=-.3cm, %custom row separation
             column sep=0cm %custom column separation (currently doing nothing)
             ]
{
%MATRIX ENTRIES
0 & 0 &  0 & 0 & 0 & 0 & 0 \\
0 & 0 & 0 & 1 & 0 & 0 & 0 \\
0 & 0 &  0 &  0 & 1 &   0 &  0 \\
0 & 0 & 0 & 0 & 0 & 0 & 0 \\
0 & 0 & 0 & 0 & 0 & 0 & 0 \\
0 & 0 & 0 & 0 & 0 & 0 & 0 \\
0 & 0 & 0 & 0 & 0 & 0 & 0 \\
};
% BASIS across the top. The "above=-1pt" can be edited to bring them lower or higher.
\node[above=-1pt] at (A-1-1.north) {$\tensor[_1]{\underline{a}}{_0} $};
\node[above=-1pt] at (A-1-2.north) {$\tensor[_1]{\underline{b}}{_0} $};
\node[above=-1pt] at (A-1-3.north) {$\tensor[_3]{\underline{c}}{_0}  $};
\node[above=-1pt] at (A-1-4.north) {$\tensor[_2]{\underline{ab}}{_1} $};
\node[above=-1pt] at (A-1-5.north) {$\tensor[_4]{\underline{bc}}{_1} $};
\node[above=-1pt] at (A-1-6.north) {$\tensor[_5]{\underline{ac}}{_1} $};
\node[above=-1pt] at (A-1-7.north) {$\tensor[_6]{\underline{abc}}{_2}  $};
%BASIS down the side. The "left=6pt" determines how far out they are.
\node[left=6pt,yshift=-.05cm] at (A-1-1.west) {$\tensor[_1]{\underline{a}}{_0} $};
\node[left=6pt,yshift=-.05cm] at (A-2-1.west) {$\tensor[_1]{\underline{b}}{_0} $};
\node[left=6pt,yshift=-.05cm] at (A-3-1.west) {$\tensor[_3]{\underline{c}}{_0} $};
\node[left=6pt,yshift=-.05cm] at (A-4-1.west) {$\tensor[_2]{\underline{ab}}{_1} $};
\node[left=6pt,yshift=-.05cm] at (A-5-1.west) {$\tensor[_4]{\underline{bc}}{_1} $};
\node[left=6pt,yshift=-.05cm] at (A-6-1.west) {$\tensor[_5]{\underline{ac}}{_1}$};
\node[left=6pt,yshift=-.05cm] at (A-7-1.west) {$\tensor[_6]{\underline{abc}}{_2} $};
%These are ANCHORS, invisible nodes in the middle/borders of the matrix that I used to draw the lines and make the shading.
%Just for reference, the "!0.5!" stuff is used to find halfway points between nodes to ensure that the lines are drawn exactly down the middle of two adjacent columns.

%here are the two anchors across the top
\node (top1) at ($(A-1-3.north)!0.5!(A-1-4.north)$) {};
\node (top2) at ($(A-1-6.north)!0.5!(A-1-7.north)$) {};

%anchors across the bottom
\node (bot1) at ($(A-7-3.south)!0.5!(A-7-4.south)$) {};
\node (bot2) at ($(A-7-6.south)!0.5!(A-7-7.south)$) {};

%anchors down the left
\node (left1) at ($(A-3-1.south west)!0.5!(A-4-1.north west)-(.15,0)$) {};
\node (left2) at ($(A-6-1.south west)!0.5!(A-7-1.north west)-(.15,0)$) {};

%anchors down the right
\node (right1) at ($(A-3-7.south east)!0.5!(A-4-7.north east)+(.15,0)$) {};
\node (right2) at ($(A-6-7.south east)!0.5!(A-7-7.north east)+(.15,0)$) {};

%Anchors in the middle of the matrix take two steps, as somehow averaging the center of 4 points at once didn't work.
%So these are just temporary nodes to be used below.
\node (a_1) at ($(A-3-3.south east)!0.5!(A-3-4.south west)$) {};
\node (a_2) at ($(A-4-3.north east)!0.5!(A-4-4.north west)$) {};
\node (b_1) at ($(A-3-6.south east)!0.5!(A-3-7.south west)$) {};
\node (b_2) at ($(A-4-6.north east)!0.5!(A-4-7.north west)$) {};

%the top two "intersection points"
\node (mid11) at ($(a_1)!0.5!(a_2)$) {};
\node (mid12) at ($(b_1)!0.5!(b_2)$) {};

\iffalse

\node (x_1) at ($(A -    i     - (j+1).south west)!0.5!(A-    i     - j.south east)$) {};
\node (x_2) at ($(A - (i+1) - (j+1).north west)!0.5!(A- (i+1) - j.north east)$) {};

\node (mid"ij") at ($(x_1)!0.5!(x_2)$) {};

\fi

%Horizontal and vertical lines

\draw (left1) -- (right1);
\draw (left2) -- (right2);
\draw (top1) -- (bot1);
\draw (top2) -- (bot2);

\begin{pgfonlayer}{background}

%Shaded blocks in a background layer so that they don't cover lines, entries, etc

\fill[black!20!white] (mid11) rectangle (top2.south);
\fill[black!20!white] (mid12) rectangle (right2.west);

\end{pgfonlayer}

\end{tikzpicture}
,
\end{equation*}

\noindent  
as verified by checking that   $\underline D =  B^{-1} D B$ for the  block-diagonal  triangular (and normalized) matrix 
%************** start matrix ***********

\begin{equation*}
B = 
  \begin{tikzpicture}[baseline={([yshift=-1.4ex]current bounding box.center)}]

[decoration=brace]

\tikzset{
    node style ge/.style={circle,minimum size=.75cm},
}
\pgfdeclarelayer{background}
\pgfdeclarelayer{foreground}
\pgfsetlayers{background,main,foreground}
\matrix (A) [matrix of math nodes,
             nodes = {node style ge},
             left delimiter  = {[}, %brackets
             right delimiter = {]}, %brackets
             inner sep=-2pt, %this determines how far the matrix brackets scale outward around the entries.
             row sep=-.3cm, %custom row separation
             column sep=0cm %custom column separation (currently doing nothing)
             ]
{
%MATRIX ENTRIES
1 & {-1} &  {-1} & 0 & 0 & 0 & 0 \\
0 & 1 & 0 & 0 & 0 & 0 & 0 \\
0 & 0 &  1 &  0 & 0 &   0 &  0 \\
0 & 0 & 0 & 1 & 1 & -1 &0 \\
0 & 0 & 0 & 0 & 1 & -1 & 0 \\
0 & 0 & 0 & 0 & 0 & 1 & 0 \\
0 & 0 & 0 & 0 & 0 & 0 & {-1} \\
};
% BASIS across the top. The "above=-1pt" can be edited to bring them lower or higher.
\node[above=-1pt] at (A-1-1.north) {$\tensor[_1]{\underline{a}}{_0}  $};
\node[above=-1pt] at (A-1-2.north) {$\tensor[_1]{\underline{b}}{_0} $};
\node[above=-1pt] at (A-1-3.north) {$\tensor[_3]{\underline{c}}{_0} $};
\node[above=-1pt] at (A-1-4.north) {$\tensor[_2]{\underline{ab}}{_1} $};
\node[above=-1pt] at (A-1-5.north) {$\tensor[_4]{\underline{bc}}{_1} $};
\node[above=-1pt] at (A-1-6.north) {$\tensor[_5]{\underline{ac}}{_1} $};
\node[above=-1pt] at (A-1-7.north) {$\tensor[_6]{\underline{abc}}{_2} $};
%BASIS down the side. The "left=6pt" determines how far out they are.
\node[left=6pt,yshift=-.05cm] at (A-1-1.west) {$\tensor[_1]{a}{_0}  $};
\node[left=6pt,yshift=-.05cm] at (A-2-1.west) {$\tensor[_1]{b}{_0} $};
\node[left=6pt,yshift=-.05cm] at (A-3-1.west) {$\tensor[_3]{c}{_0} $};
\node[left=6pt,yshift=-.05cm] at (A-4-1.west) {$\tensor[_2]{ab}{_1} $};
\node[left=6pt,yshift=-.05cm] at (A-5-1.west) {$\tensor[_4]{bc}{_1} $};
\node[left=6pt,yshift=-.05cm] at (A-6-1.west) {$\tensor[_5]{ac}{_1} $};
\node[left=6pt,yshift=-.05cm] at (A-7-1.west) {$\tensor[_6]{abc}{_2} $};
%These are ANCHORS, invisible nodes in the middle/borders of the matrix that I used to draw the lines and make the shading.
%Just for reference, the "!0.5!" stuff is used to find halfway points between nodes to ensure that the lines are drawn exactly down the middle of two adjacent columns.

%here are the two anchors across the top
\node (top1) at ($(A-1-3.north)!0.5!(A-1-4.north)$) {};
\node (top2) at ($(A-1-6.north)!0.5!(A-1-7.north)$) {};

%anchors across the bottom
\node (bot1) at ($(A-7-3.south)!0.5!(A-7-4.south)$) {};
\node (bot2) at ($(A-7-6.south)!0.5!(A-7-7.south)$) {};

%anchors down the left
\node (left1) at ($(A-3-1.south west)!0.5!(A-4-1.north west)-(.15,0)$) {};
\node (left2) at ($(A-6-1.south west)!0.5!(A-7-1.north west)-(.15,0)$) {};

%anchors down the right
\node (right1) at ($(A-3-7.south east)!0.5!(A-4-7.north east)+(.15,0)$) {};
\node (right2) at ($(A-6-7.south east)!0.5!(A-7-7.north east)+(.15,0)$) {};

%new corner anchors
\node (corner_tl) at (A-1-1.north west) {};
\node (corner_br) at (A-7-7.south east) {};

%Anchors in the middle of the matrix take two steps, as somehow averaging the center of 4 points at once didn't work.
%So these are just temporary nodes to be used below.
\node (a_1) at ($(A-3-3.south east)!0.5!(A-3-4.south west)$) {};
\node (a_2) at ($(A-4-3.north east)!0.5!(A-4-4.north west)$) {};
\node (b_1) at ($(A-6-6.south east)!0.5!(A-6-7.south west)$) {};
\node (b_2) at ($(A-7-6.north east)!0.5!(A-7-7.north west)$) {};

%the top two "intersection points"
\node (mid11) at ($(a_1)!0.5!(a_2)$) {};
\node (mid22) at ($(b_1)!0.5!(b_2)$) {};
\iffalse

\node (x_1) at ($(A -    i     - (j+1).south west)!0.5!(A-    i     - j.south east)$) {};
\node (x_2) at ($(A - (i+1) - (j+1).north west)!0.5!(A- (i+1) - j.north east)$) {};

\node (mid"ij") at ($(x_1)!0.5!(x_2)$) {};

\fi

%Horizontal and vertical lines

\draw (left1) -- (right1);
\draw (left2) -- (right2);
\draw (top1) -- (bot1);
\draw (top2) -- (bot2);

\begin{pgfonlayer}{background}

%Shaded blocks in a background layer so that they don't cover lines, entries, etc

\fill[black!20!white] (mid11) rectangle (mid22);
\fill[black!20!white] (mid11) rectangle (corner_tl);
\fill[black!20!white] (mid22) rectangle (corner_br);

\end{pgfonlayer}

\end{tikzpicture}
.
\end{equation*}

%************** end matrix ***********
%
\noindent   
The persistence canonical form ${\underline D}$ is  almost-Jordan, but not actually Jordan in this example.     
The  matrix $B$ represents the basis change to the new adapted basis  \hfil\break
$ \tensor[_1]{\underline{a}}{_0}, \tensor[_1]{\underline{b}}{_0}, \tensor[_3]{\underline{c}}{_0}, 
\tensor[_2]{\underline{ab}}{_1}, \tensor[_4]{\underline{bc}}{_1}, \tensor[_5]{\underline{ac}}{_1}, \tensor[_6]{\underline{abc}}{_2} $.  
Since $B$ is block-diagonal, 
each basis element retains pure degree, and the degree remains nondecreasing.  
Since $B$ is triangular, 
 the  level remains nondecreasing within each degree.       
The computation of this particular matrix $B$ via the standard matrix reduction algorithm is worked out in  Appendix B.2.  
\end{example}

\subsection{Categorical Structural Theorem and Structural Equivalence}
\label{subsection1.3}

A Krull-Schmidt category is an additive category where objects decompose nicely as direct 
sums of indecomposable objects.  
  In  Chapter 2,  we  study the additive category of 
filtered chain complexes in the setting of Krull-Schmidt categories, starting with 
a review of Krull-Schmidt categories in section 2.1.    A filtered complex will be called {\it basic} if 
its boundary operator can be represented by differential matrix consisting of a single Jordan block.  
We will use nonconstructive categorical methods to prove the 
following structural theorem for the category of filtered complexes:    

\begin{theorem}
\label{classFilt} 
(Categorical  Structural Theorem)
The category of filtered complexes  is Krull-Schmidt. 
A filtered complex  is indecomposable iff it is basic.  
\end{theorem}

In  chapter 3 we will prove the equivalence of the matrix and the categorical versions of the structural theorem.  
One direction is proved in section 3.1:   
\begin{proposition} 
\label{StructEquiv}  (Forward Structural Equivalence) 
 The Matrix Structural Theorem implies the Categorical Structural Theorem. 
\end{proposition}
\noindent This is followed by a detailed example of a Krull-Schmidt decomposition computation.  
The other  direction is proved in section 3.2:  
\begin{proposition} 
\label{invStructEquiv}  (Reverse Structural Equivalence) 
 The Categorical Structural Theorem implies the Matrix  Structural Theorem. 
\end{proposition}
\noindent Combining the Categorical Structural Theorem \ref{classFilt} and the Reverse Structural Equivalence Proposition \ref{invStructEquiv} 
yields a {\it nonconstructive} categorical proof of the Matrix Structural Theorem \ref{postnon}. 
This contrasts with the various {\it constructive} algorithmic proofs of Theorem  \ref{postnon}, 
which are discussed in Appendix B.2.  
The constructive algorithmic proofs explain {\it how} persistent homology works, the nonconstructive proof explains {\it why} 
persistent homology works.

\section{Proving the Categorical Structural Theorem}
\label{section2}

\subsection{Additive and Krull-Schmidt Categories}
\label{subsection2.1}
This section reviews the relevant background from category theory.    
General references for category theory include \cite{MacL,Awodey,AHS}.  
Additive categories are discussed in  \cite{MacL,HA}.  
Krull-Schmidt categories are discussed in  \cite{Krause,Miya,Atiyah}.  

\begin{definition}
A category is {\it additive} if:   
\begin{enumerate}[leftmargin=*]
\item Each  $Hom(X,Y)$ is an abelian group, and the morphism composition map \hfil\break
$Hom(Y,Z) \times Hom(X,Y) \to Hom(X,Z)$ is biadditive/bilinear.   
\item There exists a zero object $0$. 
\item Any finite collection of objects $X_1, X_2, \dots, X_n$ has a direct sum \hfil\break
$X_1 \oplus X_2 \oplus \cdots \oplus X_n$.  
\end{enumerate}
\end{definition}
\noindent  
An additive category is {\it linear} over the field ${\mathbb F}$  if each $Hom(X,Y)$ is a finite-dimensional  ${\mathbb F}$-vector space, and 
each map $Hom(Y,Z) \times Hom(X,Y) \to Hom(X,Z)$  describing composition of morphisms $(g,f) \mapsto g \circ f$ is ${\mathbb F}$-bilinear.   
All of the categories we will be studying are linear.

The {\it endomorphism ring} of an object $X$ in an additive category is the Abelian group $Hom(X,X)$ of endomorphisms,  
with multiplicative structure given by composition of endomorphisms.    
In a linear category, the endomorphism ring is an ${\mathbb F}$-algebra.  

\begin{definition}  A ring  is {\it local} if: 
\begin{enumerate}[leftmargin=*]
\item $1 \ne 0$. 
\item If  an element $f$ does not have a multiplicative inverse, then the element $1-f$ has a multiplicative inverse.
\end{enumerate}
\end{definition} 

\noindent  The local property is important, because 
a finite direct sum decomposition into summands with a local endomorphism rings is essentially unique:         
 
 \begin{theorem} 
 \label{krause42} 
 (e.g. \cite{Krause} Theorem 4.2)  
Let $X$ be an object in an additive category, and suppose there are two finite decompositions 
\begin{equation*}
 X_1 \oplus \cdots \oplus X_m = X = Y_1 \oplus \cdots \oplus Y_n  
\end{equation*}
into (nonzero) objects with local endomorphism rings. Then $m = n$ and there exists a permutation $\pi$ such that 
$X_i \simeq Y_{\pi(i)}$ for each $1 \le i \le m$.  
\end{theorem}

  An object $X$ in an additive category is  {\it decomposable} if it is the direct sum $X = Y \oplus Z$ of 
  two nonzero objects $Y$ and $Z$.     An {\it indecomposable object}, often abbreviated as an {\it indecomposable}, 
  is a nonzero object that is not decomposable.    
  
  \begin{lemma} 
  \label{locind}
  An object   is indecomposable if it has a local endomorphism ring.  
  \end{lemma}
  
  \begin{proof} We will show that the  endomorphism 
  ring of a decomposable object $X$ is not local.   We may assume that $X = Y \oplus Z$ with 
  $Y$ and $Z$ nonzero.  Then  
neither $f = 1_{Y} \oplus \, 0_{Z}$ nor 
  $1_X - f = 1_{Y} \oplus \, 1_{Z} - f = 0_{Y} \oplus \, 1_{Z}$ has 
  a multiplicative inverse in the ring $Hom(X,X) = Hom(Y \oplus Z,Y \oplus Z)$.        
\qed \end{proof}
 \noindent
A Krull-Schmidt category has properties that guarantee both the existence and essential uniqueness of finite direct sum decompositions of any object, 
see e.g. \cite{Krause,Miya} for more details:

\begin{definition} 
\label{ksdefTwo}
An additive category is Krull-Schmidt if: 
\begin{enumerate}[leftmargin=*]
\item  Every object admits a finite decomposition as a sum of indecomposables.  
\item  Every indecomposable  has a local endomorphism ring.  
 \end{enumerate} 
 \end{definition}

  Recall that  an additive category is {\it Abelian} if every morphism has a kernel and a cokernel,  every monic morphism is  normal 
 (is the kernel of some morphism), and every epic morphism is conormal (is the cokernel of some morphism). 
 Additional information about Abelian categories is outlined in Appendix \ref{subsection5.2}.  
 Note that  Definition \ref{ksdefTwo} of Krull-Schmidt category does not assume that the additive category is Abelian, or even 
 the existence of kernels and cokernels.  
We are primarily interested in the  linear category of filtered chain complexes, which is not Abelian.     
 But we will use Abelian categories and their subcategories to show that this category is nonetheless Krull-Schmidt.  
 Atiyah's Criterion \cite{Atiyah,Krause} provides a very general sufficient condition for an Abelian category to be Krull-Schmidt. 
 Since all of our categories are linear, we will only need the following special case:   
 
 \begin{theorem}
 \label{Atiyah} (Atiyah's Criterion) 
A linear Abelian category is Krull-Schmidt. 
\end{theorem}

\noindent The proof of Atiyah's Criterion is nonconstructive. It 
neither provides an 
algorithm to decompose a given object as a direct sum of indecomposables,
nor 
a  classification of  indecomposables.

\subsection{Persistence Objects and Filtered Objects} 
\label{subsection2.2}

Persistence objects \cite{Car} and filtered objects \cite{HA} are described by categorical diagrams.  
 Suppose that  ${\mathcal X}$ is a linear Abelian category (and therefore Krull-Schmidt by Theorem \ref{Atiyah}).  
 We will study persistence indexed by an integer $p \in {\mathbb Z}$,  with $\le$
 denoting the standard partial order.  
A {\it persistence object} in ${\mathcal X}$ is a diagram ${\prescript{}{\bullet} X}$ in the category ${\mathcal X}$ of  type 
\begin{equation*}
\begin{tikzcd}[row sep=tiny, column sep=normal] 
      \cdots \ar{r}  & \prescript{}{p-1} X \ar{r} &  \prescript{}{p} X   
      \ar{r} & \prescript{}{p+1} X \ar{r} & \cdots.    \\  
     \end{tikzcd}
\end{equation*}
\noindent  
A morphism of persistence objects  $\prescript{}{\bullet} f: \prescript{}{\bullet} X \to \prescript{}{\bullet} X'$ is a commutative diagram of ``ladder"  type  
\begin{equation*}%\begin{tightcenter} 
\begin{tikzcd}[row sep=normal, column sep=normal] 
\cdots \ar{r}  & \prescript{}{p-1} X  \ar{d}[left]{ \prescript{}{p-1} f} \ar{r} &  \prescript{}{p} X \ar{d}[left]{\prescript{}{p} f}
      \ar{r} & \prescript{}{p+1} X  \ar{d}[left]{\prescript{}{p+1} f} \ar{r} & \cdots \\
   \cdots \ar{r} &\prescript{}{p-1} X' \ar{r}   &  \prescript{}{p} X'   
      \ar{r}   & {\prescript{}{p+1} X'} \ar{r} & \cdots\\
     \end{tikzcd}
\end{equation*}%\begin{tightcenter} 
The category of persistence objects in ${\mathcal X}$ is Abelian, with pointwise kernels,  cokernels, and direct sums.  
The set of  morphisms  
$\prescript{}{\bullet} X \to \prescript{}{\bullet} X'$ between two persistence objects is a vector space, 
but not finite-dimensional in general.  
We will say a categorical diagram is {\it tempered} if all but finitely many of its arrows are  iso(morphisms).  
A tempered diagram has a global finiteness property, distinct from the relative finiteness property normally 
conferred by the term ``tame".    
The set of morphisms $\prescript{}{\bullet} X \to \prescript{}{\bullet} X'$ between two tempered persistence objects is a finite-dimensional 
vector space.  This is because $\prescript{}{p\pm 1} f \circ \alpha = \beta \circ \prescript{}{p} f$ in a commutative square with parallel isomorphisms 
$\alpha$ and $\beta$, determining $\prescript{}{p\pm 1} f$  in terms of $ \prescript{}{p} f$.   
  The tempered persistence objects comprise a strictly full Abelian subcategory of the persistence objects.     
Theorem \ref{Atiyah} now yields:   

\begin{proposition}
\label{persKS}  Let ${\mathcal X}$ be a linear Abelian category. 
 The category of tempered persistence objects in ${\mathcal X}$  is Krull-Schmidt.  
\end{proposition}

We next discuss subobjects in a linear Abelian category ${\mathcal X}$.    
We  make  the additional assumption that the category ${\mathcal X}$ is concrete,  meaning 
that an object  in ${\mathcal X}$ is  a set with  some additional features, 
and a morphism in ${\mathcal X}$  is a map of sets compatible with the additional features.  
For example,   the linear Abelian category  ${\mathcal V}$ of (finite-dimensional) vector spaces is a concrete linear Abelian category.       
An {\it inclusion }  $X \hookrightarrow X'$  in ${\mathcal X}$ is an arrow  
that is an inclusion of the underlying sets.   We say $X$ is a {\it subobject} of $X'$ iff such an inclusion arrow 
 exists.  An inclusion arrow is  monic \cite{MacL,Awodey}, and the composition of inclusion arrows is an inclusion arrow.  
Any object $X'$ in ${\mathcal X}$ has a zero  subobject $0 \hookrightarrow X'$, and is its own subobject $X' \hookrightarrow X'$.  
A subobject $X \hookrightarrow X'$ is {\it proper} if $X \ne X'$.  
A nonzero object is said to be {\it simple} if it does not have a  proper nonzero subobject.       
A simple object is obviously indecomposable,  
but an indecomposable object need not be simple.

A filtered object in a concrete linear Abelian category ${\mathcal X}$ is a special type of tempered persistence object in ${\mathcal X}$.    
We say a tempered persistence object $\prescript{}{\bullet} X$ in ${\mathcal X}$  is {\it bounded below} if there exists an 
integer $j$ such that $\prescript{}{j} X = 0$ whenever $p \le j$.  
We say $\prescript{}{\bullet} X$  is a   
  {\it filtered object} if it is bounded below and if   
every arrow is an inclusion arrow: 
\begin{equation*}
\begin{tikzcd}[row sep=normal, column sep=normal] 
      \cdots \ar[hook]{r}  & \prescript{}{p-1} X \ar[hook]{r} &  \prescript{}{p} X   
      \ar[hook]{r} & \prescript{}{p+1} X \ar[hook]{r} & \cdots   .\\  
     \end{tikzcd}
\end{equation*}
\noindent The   filtered objects in ${\mathcal X}$ comprise  a strictly full subcategory of the tempered persistence objects. 
The properties of monics  have several consequences.  
A filtered object diagram has a categorical limit and a colimit   \cite{MacL,Awodey}. 
The limit is $0$ since the diagram is bounded below.  
The colimit $X$ is  $\prescript{}{k} X$ for $k$ sufficiently large  (satisfying $\prescript{}{p} X = X$ whenever $k \le p$).  
Finally, any summand of a filtered object is isomorphic to a filtered object. 
Combining these facts with Proposition \ref{persKS} yields: 
\begin{lemma}
\label{crass}
 Let ${\mathcal X}$ be a linear Abelian category.  The category of filtered objects in ${\mathcal X}$ is Krull-Schmidt. 
A filtered object $\prescript{}{\bullet} X$ in ${\mathcal X}$ is indecomposable iff its colimit  $X$ is an indecomposable object in ${\mathcal X}$.    
\end{lemma} 
Here a filtered object $\prescript{}{p}  Z$ with a decomposable colimit $X \oplus Y$ decomposes as the direct sum 
of the filtered objects $\prescript{}{p}  Z \cap X$ and $\prescript{}{p}  Z \cap Y$.

We note that the filtered objects comprise a subcategory of the tempered persistence objects, but 
this subcategory is not  Abelian because a morphism of filtered objects may have a kernel and/or cokernel that is not a filtered object.  
So Lemma \ref{crass} is not merely a corollary of Theorem  \ref{Atiyah}.   Finally we observe that the category of persistence objects in a 
(concrete) linear 
Abelian category is itself a (concrete) linear Abelian category, to which Lemma \ref{crass} applies.

\subsection{Chain Complexes and Filtered Chain Complexes}  
\label{subsection2.3}

A {\it persistence vector space} is a 
persistence object in the (concrete) linear Abelian category  ${\mathcal X} ={\mathcal V}$ of (finite-dimensional) ${\mathbb F}$-vector spaces.  
Tempered persistence vector spaces are well-understood via the theory of quiver representations.    
A nonempty subset $I \subseteq {\mathbb Z}$ will be called an {\it interval} if $c \in I$ 
whenever $a \le c \le b$ with $a \in I$ and $b \in I$.   
We associate to an interval  $I \subseteq {\mathbb Z}$ the {\it interval persistence vector space} 
 $\prescript{}{\bullet} I$ constructed as follows:  
$\prescript{}{p} I = {\mathbb F}$ whenever $p \in I$,  
$\prescript{}{p} I = 0$ whenever $p \notin I$, 
and every arrow ${\mathbb F} \to {\mathbb F}$ is the identity morphism $1$.    
We will often omit the bullet prescript when context allows.  
 For example, the  interval persistence vector space $[1,4) = \prescript{}{\bullet}  [1,4)$  is the diagram of vector spaces 
%
% ok
\begin{equation*}%\begin{tightcenter} 
\begin{tikzcd}[row sep=tiny, column sep=small] 
      \cdots \ar{r}  & 0  \ar{r} &{\mathbb F}   
      \ar{r}{1}& {\mathbb F}   \ar{r}{1} & {\mathbb F}  \ar{r}& 0 \ar{r} & \cdots   \\  
    &  p =0 & p = 1 & p = 2 &  p = 3 & p = 4  \\
     \end{tikzcd}
\end{equation*}%\begin{tightcenter} 
associated to the interval $[1,4) =\{ 1, 2, 3 \} \subseteq {\mathbb Z}$.  
Proposition \ref{persKS} applies to the linear Abelian category of tempered persistence vector spaces.      
Furthermore, the well-studied representation theory of  $A_n$ quivers (see e.g. \cite{Shiffler}) 
carries over by a  limiting argument to prove the following structural theorem for the 
category of tempered persistence vector spaces:   

\begin{theorem}  
\label{quiver} 
The category of tempered persistence vector spaces is Krull-Schmidt.  
A tempered persistence vector space is indecomposable iff it is isomorphic to an interval. 
\end{theorem} 

Theorem \ref{quiver} can be applied to cochain complexes.  
A {\it cochain complex}, or  {\it cocomplex} for short,  is a tempered persistence vector space $\prescript{}{\bullet} V$   
\begin{equation*}
\begin{tikzcd}[row sep=tiny, column sep=normal] 
      \cdots \ar{r}  & \prescript{}{p-1} V \ar{r}{\partial^p} &  \prescript{}{p} V   
      \ar{r}{\partial^{p+1}} & \prescript{}{p+1} V \ar{r} & \cdots     \\  
     \end{tikzcd}
\end{equation*}
with the property that 
the composition of successive arrow $\partial^{p+1} \circ \partial^p$ is zero.  
The kernel of a morphism between cocomplexes is a cocomplex, as is the cokernel, so the 
cocomplexes comprise a strictly full Abelian subcategory of the tempered persistence vector spaces.  
Theorem \ref{Atiyah} and Theorem \ref{quiver} now yield the structural result:  
  
\begin{proposition}
\label{cocomp} The category ${\mathcal C}^{op}\,$of  cocomplexes  is linear and Abelian, and therefore Krull-Schmidt. 
A  cocomplex is indecomposable iff it is isomorphic to an interval cocomplex. 
\end{proposition} 

Chain complexes are dual to cochain complexes. 
A {\it  complex} (short for chain complex)  is a tempered diagram $V_\bullet$ in ${\mathcal V}$ of  type 
%
% bug
\begin{equation*}
\begin{tikzcd}[row sep=tiny, column sep=normal] 
      \cdots  & \ar{l} V_{n-2} &   \ar{l}{\partial_{n-1}}   V_{n-1}  
     &  \ar{l}{\partial_{n}}   V_{n} \ar{l} & \ar{l} \cdots     \\  
     \end{tikzcd}
\end{equation*}
with the property that 
the composition of successive arrows $ \partial_{n-1} \circ \partial_{n}$ is zero.
 A morphism of  complexes $f_\bullet: V_\bullet \to V'_\bullet$ is a commutative ladder diagram.  
 The category ${\mathcal V}$ of vector spaces is isomorphic to its opposite category ${\mathcal V}^{op}$ via the 
 duality  functor that  takes a vector space to its dual and a linear map to 
its transpose/adjoint \cite{MacL,Awodey}.  Duality takes the category ${\mathcal C}^{op}$  of cocomplexes to
the category of  complexes ${\mathcal C}$.  
A complex is called an {\it interval complex} if its dual  is an interval cocomplex, and    
Proposition \ref{cocomp} becomes:  

\begin{proposition}
\label{comp} The category  ${\mathcal C}$ of  complexes  is linear and Abelian, and therefore Krull-Schmidt. 
A  complex is indecomposable iff it is isomorphic to an interval complex. 
\end{proposition} 

\noindent The interval complexes are easily classified.      
An interval complex  $I_\bullet$ is associated to an interval $I \subseteq {\mathbb Z}$ as follows:   
$n \in I$ whenever $I_n = {\mathbb F}$,  and $n \notin I$ whenever $I_n = 0$.      
Since adjacent nonzero arrows in a complex cannot be iso(morphisms), 
 the interval complexes are in bijective correspondence with the intervals $I \subseteq {\mathbb Z}$ of  
cardinality at most two.    
We will often omit the bullet postscript when the context allows.  
We denote by $J[n] = \{ n \} \subseteq {\mathbb Z}$  the intervals of cardinality one.  
 For example, the  complex  $J[1] =J[1]_\bullet$ is the diagram of vector spaces
\begin{equation*}%\begin{tightcenter} 
\begin{tikzcd}[row sep=tiny, column sep=small] 
      \cdots & \ar{l}   0 &  {\mathbb F}   \ar{l}     
      & \ar{l}   0 \ar{l} & \ar{l}   0 \ar{l}& \ar{l} \cdots    \\  
      & n = 0 & n = 1 & n = 2 & n = 3\\
     \end{tikzcd}
     .
\end{equation*}%\begin{tightcenter} 
The indecomposable complex $J[n]$ is simple.
We denote by 
$K[n] =[n,n+1] \subseteq {\mathbb Z}$  the intervals of cardinality two.  
For example, 
the  complex $K[1]=K[1]_\bullet$ is the diagram of vector spaces
%
% bug
\begin{equation*}%\begin{tightcenter} 
\begin{tikzcd}[row sep=tiny, column sep=small] 
      \cdots & \ar{l}   0 &    {\mathbb F} \ar{l} 
      &  {\mathbb F} \ar{l}{1}    & \ar{l}   0 \ar{l}& \ar{l} \cdots     \\  
      & n = 0 & n = 1 & n = 2 & n = 3\\
     \end{tikzcd}
\end{equation*}%\begin{tightcenter} 
The indecomposable complex $K[n]$  has exactly one  nonzero proper subobject \hfil\break
$J[n] \hookrightarrow K[n]$.   For example, the inclusion 
 of complexes $J[1] \hookrightarrow K[1]$ is the commutative ladder diagram 
%
% bug
\begin{equation*}%\begin{tightcenter} 
\begin{tikzcd}[row sep=normal, column sep=small] 
      \cdots & \ar{l}   0  \ar{d} &    {\mathbb F} \ar{l} 
      \ar{d}{1} &  0 \ar{l}   \ar{d}& \ar{l}   0\ar{d} & \ar{l} \cdots     \\  
      \cdots & \ar{l}   0 &    {\mathbb F} \ar{l} 
      &  {\mathbb F} \ar{l}{1}    & \ar{l}   0 \ar{l}& \ar{l} \cdots     \\  
      & n = 0 & n = 1 & n = 3 & n = 4\\
     \end{tikzcd}
\end{equation*}%\begin{tightcenter} 

We now return to the  the Categorical Structural Theorem  \ref{classFilt}.  
 A filtered complex 
  is a diagram in the category ${\mathcal C}$ of complexes  
\begin{equation*}%\begin{tightcenter} 
\begin{tikzcd}[row sep=tiny, column sep=normal] 
      \cdots \ar[hook]{r}  & \tensor[_{p-1}]{V}{_\bullet} \ar[hook]{r} &  \tensor[_{p}]{V}{_\bullet} 
      \ar[hook]{r} & \tensor[_{p+1}]{V}{_\bullet} \ar[hook]{r} & \cdots   .\\  
     \end{tikzcd}
\end{equation*}%\begin{tightcenter} 
We will say a filtered complex is {\it basic} if its colimit $V_\bullet$ is isomorphic to an interval complex.  
The first statement of Proposition \ref{comp} tells us that the category   ${\mathcal C}$ of complexes is linear and Abelian.  
Then  Lemma \ref{crass} tells us that the category of filtered complexes is Krull-Schmidt.  
The second statement of  Proposition \ref{comp} classifies the indecomposable filtered complexes, 
completing the proof of:  

% manually number repeated theorem
\setcounter{repsection}{1} 
\setcounter{repeat}{5}
\begin{reptheorem}
(Categorical  Structural Theorem)
The category of filtered complexes  is Krull-Schmidt. 
A filtered complex  is indecomposable iff it is basic. 
\end{reptheorem} 

\noindent  The basic filtered complexes are easily classified 
since we  know all  proper subobjects of interval complexes, 
namely  $0 \hookrightarrow J[n]$,  
  $0 \hookrightarrow K[n]$, and $J[n] \hookrightarrow K[n]$.   
Details and examples of basic filtered complexes appear in Chapter 4.  

\section{Categorical Frameworks for Persistent Homology} 
\label{section3}

\subsection{Standard Framework using Persistence Vector Spaces}  
\label{subsection3.1}
The structural theorem for  the category of tempered persistence vector spaces,  Theorem \ref{quiver},   is the foundation for the standard 
framework for  persistent homology. 

For each integer $n$, the {\it homology}  
 of degree $n$ is a functor $H_n: {\mathcal C} \to {\mathcal V}$ from the category ${\mathcal C}$ of complexes to the category ${\mathcal V}$ of 
vector spaces.  An object  $C$ in ${\mathcal C}$ is a diagram of vector spaces   
%
% bug
\begin{equation*}
\begin{tikzcd}[row sep=tiny, column sep=normal] 
      \cdots  & \ar{l} V_{n-1} &   \ar{l}{\partial_n}   V_n  
     &  \ar{l}{\partial_{n+1}}   V_{n+1} \ar{l} & \ar{l} \cdots    \\  
     \end{tikzcd}
\end{equation*}
where $\partial_n \circ \partial_{n+1}  = 0$.  Then $\image \partial_{n+1}  \hookrightarrow \ker \partial_n$ is a 
 subobject inclusion  of vector spaces, 
 and  the homology is the quotient vector space (cokernel)      
\begin{equation*}
 H_n(C) = {\ker \partial_n} / {\image \partial_{n+1}}. 
 \end{equation*}
More generally, the  homology functor $H_n$ takes a diagram in  ${\mathcal C}$ to a 
diagram in ${\mathcal V}$.

Denote by ${\mathcal F}$ the category of filtered complexes. 
An object $F$ in ${\mathcal F}$ is a  diagram of complexes 
\begin{equation*}%\begin{tightcenter} 
\begin{tikzcd}[row sep=tiny, column sep=normal] 
      \cdots \ar[hook]{r}  & \tensor[_{p-1}]{V}{_\bullet} \ar[hook]{r} &  \tensor[_{p}]{V}{_\bullet} 
      \ar[hook]{r} & \tensor[_{p+1}]{V}{_\bullet} \ar[hook]{r} & \cdots   ,  \\  
     \end{tikzcd}
\end{equation*}%\begin{tightcenter} 
\noindent which is tempered and bounded below, and which has  monic arrows.  
Denote by ${\mathcal P}$ the category of tempered persistence vector spaces. 
The  homology functor $H_n$ takes  
the diagram $F$ to the diagram of vector spaces 
\begin{equation*}%\begin{tightcenter} 
\begin{tikzcd}[row sep=tiny, column sep=normal] 
      \cdots \ar{r}  & H_n(\prescript{}{p-1} V_\bullet) \ar{r} &  H_n(\prescript{}{p} V_\bullet )
      \ar{r} & H_n(\prescript{}{p+1} V_\bullet) \ar{r} & \cdots ,  \\  
     \end{tikzcd}
\end{equation*}%\begin{tightcenter} 
\noindent  which is  tempered and bounded below, but which need not have monic arrows in general.  
So an object $F$ in ${\mathcal F}$  goes to an  object $P_n(F)$  in ${\mathcal P}$.     
Similarly a morphism in ${\mathcal F}$, which is a commutative ladder diagram of complexes, 
goes to a morphism in ${\mathcal P}$, which is a commutative ladder diagram of vector spaces.   
The resulting functor $P_n:{\mathcal F} \to {\mathcal P}$ is  the {\it persistent homology} of degree $n$.   

The standard framework for studying the   
 persistent homology functors $P_n:{\mathcal F} \to {\mathcal P}$ is  based on the structural theorem for the category ${\mathcal P}$, 
 Theorem \ref{quiver}.  
It suffices to work with an appropriate Krull-Schmidt subcategory %$\image P_n$ 
of the Krull-Schmidt category  ${\mathcal P}$. 
  A filtered complex $F$ is studied 
by decomposing the persistence vector space $P_n(F)$ as a sum of  indecomposables.  
Since the diagram $P_n(F)$ is bounded below, all of its indecomposables are bounded below.  
The persistence vector spaces that are bounded below comprise  a full Abelian subcategory of ${\mathcal P}$, 
which we will denote by $\image P_n$.  
Despite the notation, the category $\image P_n$ does not depend on $n$;  it is always the same subcategory of ${\mathcal P}$.
The isomorphism class of an indecomposable in the Krull-Schmidt category ${\image P_n}$ is  described by the familiar barcode.     
An interval $I \subseteq {\mathbb Z}$ will be called a  {\it barcode} if  it is bounded below.  
A {\it barcode persistence vector space} is a persistence vector space $\prescript{}{\bullet} I$ corresponding to a barcode  
 $I \subseteq {\mathbb Z}$.   
\begin{theorem}
\label{standard}  The persistent homology functor $P_n:{\mathcal F} \to {\mathcal P}$ factors as 
\begin{equation*}
{\mathcal F} \to {\image P_n} \to {\mathcal P}. 
\end{equation*}
The category $\image P_n$ is Krull-Schmidt.  
An object in $\image P_n$ is indecomposable iff it is isomorphic to a barcode persistence vector space.  
 \end{theorem}  
We can now express the standard framework for persistent homology in terms of the functor   
 ${\mathcal F} \to {\image P_n}$ which takes a filtered complex to a persistence vector space in $\image P_n$.    
The Krull-Schmidt property of $\image P_n$ then allows decomposition as a sum of indecomposables. 
Each indecomposable in $\image P_n$ is a barcode persistence vector space, which is specified up to isomorphism by 
its barcode $I \subseteq {\mathbb Z}$.   An object in $\image P_n$ is determined up to isomorphism by its set of barcodes.

\subsection{Alternate Framework using Quotient Categories} 
\label{subsection3.2}
The structural theorem for  the category of filtered complexes,  Theorem \ref{classFilt},   is the foundation for  our alternate  
framework for  persistent homology.    

We will work with an appropriate Krull-Schmidt quotient category %$\coimage P_n$ 
of the Krull-Schmidt category ${\mathcal F}$.   
Recall in general \cite{MacL} that an object of a quotient category of ${\mathcal F}$ is an object of ${\mathcal F}$, 
and a morphism is an equivalence class of morphisms of ${\mathcal F}$.  
Our quotient category ${\coimage P_n}$ is defined via the following equivalence relation (congruence) on morphisms:  
two morphisms $f$ and $f'$ in ${\mathcal F}$ are equivalent iff the morphisms $P_n(f)$ and $P_n(f')$ in ${\mathcal P}$ are equal.   
 Note that the category $\coimage P_n$ now does depend on the integer $n$; each $\coimage P_n$ is a different 
 quotient category of ${\mathcal F}$.

 \begin{example} 
 We return to the filtered simplicial complex of  Example \ref{explain}, as shown in Figure \ref{fig:sc1rep}. 
%
% figure simplicial1.pdf
\begin{figure*}
  \includegraphics[width=\textwidth]{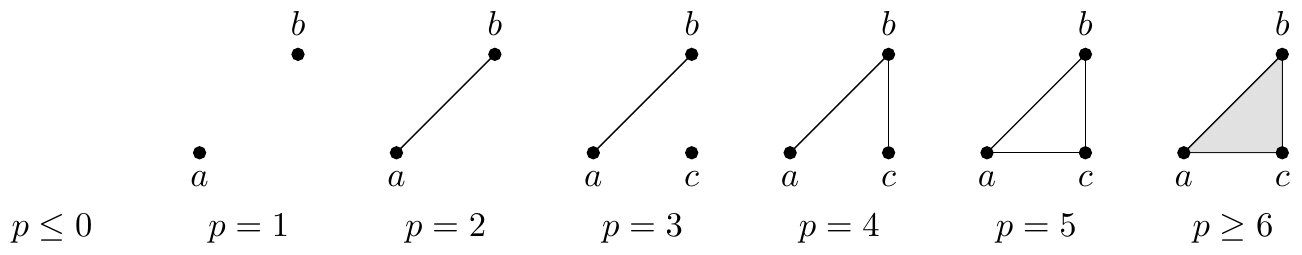}
\caption{A filtered simplicial complex.}
\label{fig:sc1rep}       % Give a unique label
\end{figure*}

We first consider  the subobject shown in Figure \ref{fig:sc2}. 
%
% figure simplicial2.pdf}
\begin{figure*}
  \includegraphics[width=\textwidth]{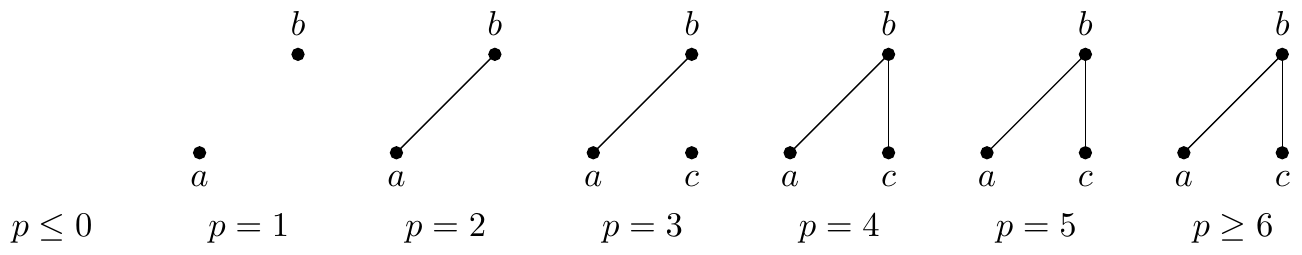}
\caption{A subobject of the filtered simplicial complex.}
\label{fig:sc2}       % Give a unique label
\end{figure*}
In ${\mathcal F}$, this is a proper nonzero subobject.   
  In the quotient category  ${\coimage P_0}$,  the subobject inclusion becomes an isomorphism between 
 nonzero objects.    
 In the quotient category ${\coimage P_1}$, this becomes  a proper zero subobject.

Now consider another subobject as shown in Figure \ref{fig:sc3}.         
%
% figure simplicial3.pdf
\begin{figure*}
  \includegraphics[width=\textwidth]{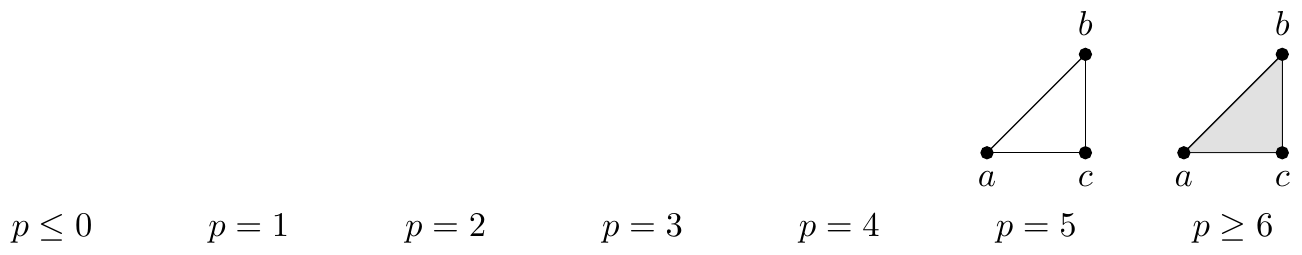}
\caption{Another subobject of the filtered simplicial complex.}
\label{fig:sc3}       % Give a unique label
\end{figure*}
In ${\mathcal F}$, this is  a proper nonzero subobject.   
In the quotient category ${\coimage P_0}$, this remains  a proper nonzero subobject.  
In the quotient category  ${\coimage P_1}$,  the inclusion morphism becomes an isomorphism between 
 nonzero objects. 
 \end{example}

We recall that a quotient of a Krull-Schmidt category  is Krull-Schmidt in general.  This is because an 
 indecomposable  in ${\mathcal F}$ becomes either  a zero object or an indecomposable 
 with a local endomorphism ring in the quotient category (see e.g. \cite{Liu} p. 431).    
 The classification of indecomposables in the quotient category ${\coimage P_n}$ is now easily 
 obtained from Theorem \ref{classFilt}.  This  is independent of the well-known classification of 
 indecomposables in the category of persistence vector spaces (Theorem \ref{quiver}).  
Using the  classification of indecomposables in each of the Krull-Schmidt categories ${\coimage P_n}$ and ${\image P_n}$, 
it is now easy to  
verify that the functor ${\coimage P_n} \to {\image P_n}$ is full, faithful, and essentially surjective. 
Recalling \cite{MacL,Awodey} that a functor satisfying these conditions is an equivalence of categories, we have:  
 \begin{theorem}
 \label{alternate}  The persistent homology functor $P_n:{\mathcal F} \to {\mathcal P}$ factors as 
 \begin{equation*} {\mathcal F} \to {\coimage P_n}\to {\image P_n} \to {\mathcal P}, 
\end{equation*}
where the functor ${\coimage P_n}\to {\image P_n}$ is an equivalence of categories.   
 \end{theorem}  
\noindent   The isomorphism class of an indecomposable object $\prescript{}{\bullet} X$ in the Krull-Schmidt category ${\coimage P_n}$ can be specified  
as $I_n$.   Here the integer $n$ is  the degree/dimension label of the category ${\coimage P_n}$.    
The interval subset $I \subseteq {\mathbb Z}$ is defined by the rule:  
$p \in I$ iff the complex $\prescript{}{p} X$ at level $p$ is isomorphic to $J[n]$.  
We note that the classification of indecomposables in ${\coimage P_n}$ does {\it not} reference homology.  
This point will be illustrated in detail in Example 4.4 below.  
We remark that our naming choices for   
 ${\coimage P_n}$ and  ${\image P_n}$ are intended to emphasize the parallel between Theorem \ref{alternate} and 
 the factorization of a morphism in an Abelian category, see Section \ref{subsection5.2}.    
 Section 5.1 reviews an analogous functor factorization in the simpler setting of plain (not persistent) homology.

We can now express the alternate framework for persistent homology in terms of the functor  
 ${\mathcal F} \to {\coimage P_n}$ which takes a filtered complex in ${\mathcal F}$ to the same filtered complex 
 viewed as an object in  $\coimage P_n$.  
The  Krull-Schmidt property of  $\coimage P_n$ then allows decomposition as a sum of   indecomposables   in ${\coimage P_n}$.    
Each indecomposable  is specified up to isomorphism by  $I_n$.  An object in $\coimage P_n$ is determined up to 
isomorphism by the collection of  intervals $I_n \subseteq {\Bbb Z}$ indexing its decomposition.       
This  framework obviates the need for  auxiliary objects such as persistence vector spaces, while 
providing exactly the same information about  filtered complexes as the standard framework.    
These two frameworks are further compared in Section \ref{subsection5.4}.

 \section{Proving Structural Equivalence}
 \label{section4}
 
 \subsection{Forward Structural Equivalence}   
 \label{4.1}
 
 We now consider in more detail matrix representations of a filtered complex and its automorphisms.    
 The first step is to associate to a filtered complex a  finite-dimensional vector space with an 
 appropriately adapted basis.   
A filtered  complex  is a  diagram of complexes indexed by the integer  level $p$,  
displayed below together with its  colimit:  
\begin{equation*}%\begin{tightcenter} 
\begin{tikzcd}[row sep=tiny, column sep=tiny] 
   \cdots \ar[hook]{r}  & \tensor[_{-1}]{V}{_\bullet}  \ar[hook]{r} & \tensor[_0]{V}{_\bullet}    
      \ar[hook]{r} & \tensor[_1]{V}{_\bullet}  \ar[hook]{r} &  \tensor[_2]{V}{_\bullet} \ar[hook]{r}& \tensor[_3]{V}{_\bullet} 
       \ar[hook]{r} & \cdots & & V_\bullet  \\  
 & \phantom{p = -1}   &  \phantom{p = 0} &  \phantom{p = 1} &  \phantom{p = 2} &  \phantom{p = 3} &  & & {\rm colim}\\ 
     \end{tikzcd}
\end{equation*}%\begin{tightcenter} 
\noindent
A filtered complex becomes a ``lattice" diagram of finite-dimensional vector spaces:  
\begin{equation*}%\begin{tightcenter} 
\begin{tikzcd}[row sep=small, column sep=tiny]
  &  \vdots \ar{d} &  \vdots \ar{d}&  \vdots \ar{d}&  \vdots \ar{d} &   \vdots \ar{d} & & & \vdots \ar{d}\\ 
       \cdots  \ar[hook]{r}  & \tensor[_{-1}]{V}{_2}   \ar{d} \ar[hook]{r} & \tensor[_0]{V}{_2}   \ar{d} \ar[hook]{r} &\tensor[_1]{V}{_2}  \ar{d} \ar[hook]{r} &  \tensor[_2]{V}{_2}  \ar{d} \ar[hook]{r} & \tensor[_3]{V}{_2}  \ar{d} \ar[hook]{r} & \cdots 
      & & V_2 \ar{d}{\partial_2}\\  
      \cdots  \ar[hook]{r}  & \tensor[_{-1}]{V}{_1}   \ar{d} \ar[hook]{r} & \tensor[_0]{V}{_1}   \ar{d} \ar[hook]{r} &\tensor[_1]{V}{_1}  \ar{d} \ar[hook]{r} &  \tensor[_2]{V}{_1}  \ar{d} \ar[hook]{r} & \tensor[_3]{V}{_1}  \ar{d} \ar[hook]{r} & \cdots 
      & &V_1 \ar{d}{\partial_1} \\  
      \cdots  \ar[hook]{r}  & \tensor[_{-1}]{V}{_0}   \ar{d} \ar[hook]{r} & \tensor[_0]{V}{_0}   \ar{d} \ar[hook]{r} &\tensor[_1]{V}{_0}  \ar{d} \ar[hook]{r} &  \tensor[_2]{V}{_0}  \ar{d} \ar[hook]{r} & \tensor[_3]{V}{_0}  \ar{d} \ar[hook]{r} & \cdots 
      & &V_0 \ar{d}{\partial_0}\\  
        \cdots  \ar[hook]{r}  & \tensor[_{-1}]{V}{_{-1}}   \ar{d} \ar[hook]{r} & \tensor[_0]{V}{_{-1}}   \ar{d} \ar[hook]{r} &\tensor[_1]{V}{_{-1}}  \ar{d} \ar[hook]{r} &  \tensor[_2]{V}{_{-1}}  \ar{d} \ar[hook]{r} & \tensor[_3]{V}{_{-1}}  \ar{d} \ar[hook]{r} & \cdots 
      & &V_{-1} \ar{d}\\      
  &  \vdots &  \vdots&  \vdots &  \vdots &  \vdots & & & \vdots \\ 
&  \phantom{p = -1}   &  \phantom{p = 0} &  \phantom{p = 1} &  \phantom{p = 2} &  \phantom{p = 3} &  & & {\rm colim}\\ 
     \end{tikzcd}
\end{equation*}%\begin{tightcenter} 
\noindent In the colimit complex, the  composition $\partial_{n-1} \circ \partial_{n}: V_n \to V_{n-2}$   
is  zero  for all $n$.  Since the diagram is tempered,    $\partial_{n-1} \circ \partial_{n}$ is an isomorphism for all 
but finitely many $n$.    It follows that  the complex is bounded, 
meaning that the vector space $V_n$ is zero-dimensional for all but finitely many $n$.    
The direct sum   $V = \oplus_{n \in {\mathbb Z}} \,  V_n$ is then a finite-dimensional vector space associated to the filtered complex.       
A vector $v \in V$ is said to have {\it  pure degree}  iff $v \in V_n \subseteq V$ for some integer $n$.    
The integer $n$ is then called the {\it degree} of the pure degree vector $v$, and is encoded by a postscript $v_n$.   
The {\it  (filtration) level} of a degree $n$ vector $v_n \in V$ is the smallest integer $p$ such that 
$v_n \in \tensor[_p]{V}{_n} \subseteq V_n$.    
The  level of the degree $n$ vector $v_n$ is encoded by a prescript $\tensor[_p]{v}{_n}$.    

Gaussian elimination constructs an adapted basis for a filtered vector space.  Summing over degrees, we obtain   
an {\it adapted basis} of a filtered complex, meaning a basis of the vector space $V$ satisfying the three conditions:  
\begin{itemize}[leftmargin=*]
\item	Every basis element has pure degree.     
\item	For each $n$ and $p$, the vector space $\, \tensor[_p]{V}{_n}$ is spanned by 
the basis vectors with degree equal to $n$ and  level less than or equal to $p$. 
\item The basis elements are ordered so that degree is nondecreasing, and within each degree the 
 level is nondecreasing.   
\end{itemize}
\noindent
A block-diagonal  triangular matrix $B$  transforms an adapted basis to a new adapted basis, 
representing an automorphism of the filtered complex.   Here we assume that the block structure of the matrix is 
compatible with degrees of the basis elements.

The {\it colimit boundary}  $\partial  = \oplus_{n \in {\mathbb Z}} \,  \partial_n$ of a filtered complex is      
a linear endomorphism $\partial: V \to V$. 
The colimit boundary $\partial$ is represented by a matrix $D$ relative to an adapted basis.  
The matrix representative $D$ is  block-superdiagonal  because $\partial$ is homogeneous of degree $-1$, 
and $D^2 = 0$ because  $\partial^2 = 0$.  
If additionally the matrix representative is almost-Jordan, we will say the  
adapted basis is {\it special}.  
The Matrix Structural Theorem \ref{postnon}   yields:

\begin{proposition}
\label{forego}  A filtered complex admits a special adapted basis.  
\end{proposition} 

\begin{proof} Choose an adapted basis.  Let $D$ be the block-superdiagonal differential matrix 
representing $\partial$ relative to the adapted basis.   Theorem \ref{postnon} provides 
a block-diagonal triangular matrix $B$ such that  ${\underline D} = B^{-1} D B$ is almost-Jordan.  
So the matrix $B$ transforms the original adapted basis to a special adapted basis.  
\qed \end{proof} 

\begin{corollary} 
\label{deco}
A filtered complex admits a finite decomposition as a sum of basic filtered complexes.  
\end{corollary}

\begin{proof}  Choose a special adapted basis, and denote by ${\underline D}$ the 
corresponding almost-Jordan block-superdiagonal differential matrix representative.    
Let $P$ be a permutation matrix such that the matrix $P^{-1} {\underline D} P$ is Jordan.  
Each Jordan block of this matrix represents a basic subobject of the filtered complex.  
The decomposition into Jordan blocks represents the decomposition of the 
 filtered complex as a direct sum of basic filtered complexes.  
\qed \end{proof} 

To verify the Krull-Schmidt property, we will also need:  

\begin{lemma}   
\label{locring}
A basic filtered complex has   local endomorphism ring. 
\end{lemma} 

\begin{proof}  We first show that  the colimit complex of a basic filtered complex has local endomorphism ring.
The colimit complex is isomorphic to an interval complex.  
An interval complex is an indecomposable in the linear Abelian category of complexes, 
so it has local endomorphism ring by  Atiyah's Criterion \ref{Atiyah}.   
(Or less abstractly, it is easy to check that the endomorphism ring of an interval complex is isomorphic to the field ${\mathbb F}$.) 

The proof is completed by checking that  the endomorphism ring of a basic filtered complex 
maps isomorphically to the endomorphism ring of its colimit interval complex.   
In general, the endomorphism ring of a filtered object maps {\it injectively} 
to the endomorphism ring of its colimit.  
We need to show  that the endomorphism ring of a basic filtered complex maps 
{\it surjectively} to the endomorphism ring of its colimit.  It suffices to show   
that an endomorphism of an interval complex restricts to an endomorphism of any subobject.   
  There are two types of interval complexes to consider.   
If the interval complex is isomorphic to $J[n]$, then the 
 subobjects  are $0$ and $J[n]$, and any  endomorphism restricts. 
If the interval complex is isomorphic to $K[n]$, then the subobjects are 
 $0$, $J[n]$, and $K[n]$, and any  endmorphism restricts.    
\qed \end{proof} 

\noindent  
Assembling the pieces  proves the main result of this section:      

 % manually number repeat proposition
 \setcounter{repsection}{1} 
\setcounter{repeat}{6}
\begin{repproposition}
 (Forward Structural Equivalence) 
 The Matrix Structural Theorem implies the Categorical Structural Theorem. 
\end{repproposition}

\begin{proof}  We first prove that a filtered complex is indecomposable iff it is basic.  
A basic filtered complex has a local endomorphism ring by Lemma \ref{locring}, 
so it is indecomposable by 
Lemma \ref{locind}.  
An indecomposable filtered complex is a finite direct sum of basic filtered complexes by Corollary \ref{deco}.  
The direct sum cannot have more than one summand, because that would contradict the indecomposability.   
So an indecomposable filtered complex is  basic.  

Now it remains to check the two conditions of Definition \ref{ksdefTwo}.   
Since a basic filtered complex is indecomposable, Corollary \ref{deco} asserts that every filtered complex 
admits a finite decomposition as a sum of indecomposables.  
Since an indecomposable filtered complex is basic,  Lemma \ref{locring} asserts that every indecomposable 
has a local endomorphism ring.   
\qed \end{proof}

\begin{example} 
\label{nuex} Let $F$ be the filtered complex of Example \ref{exfilt}.  The initial 
adapted basis  consists of appropriately ordered simplices:  
$\tensor[_1]{a}{_0},   \tensor[_1]{b}{_0}, \tensor[_3]{c}{_0},   
 \tensor[_2]{ab}{_1},   \tensor[_4]{bc}{_1},  \tensor[_5]{ac}{_1},  \tensor[_6]{abc}{_2}$.   
The block-superdiagonal differential matrix $D$ represents the colimit boundary operator relative to the initial adapted basis.  

The triangular block-diagonal  matrix $B$ represents an automorphism of the filtered complex.   
This automorphism takes the initial adapted basis to the transformed adapted basis 
$  \tensor[_1]{\underline{a}}{_0}, \tensor[_1]{\underline{b}}{_0}, \tensor[_3]{\underline{c}}{_0}, 
\tensor[_2]{\underline{ab}}{_1}, \tensor[_4]{\underline{bc}}{_1}, \tensor[_5]{\underline{ac}}{_1}, \tensor[_6]{\underline{abc}}{_2}$. 
This transformed adapted basis is special, because 
 the  block-superdiagonal differential matrix  representative ${\underline D} = B^{-1} D B$ is  almost-Jordan:  
%%

%************** start matrix ***********
 
 \begin{equation*}
 {\underline D} =
 \begin{tikzpicture}[baseline={([yshift=-1.4ex]current bounding box.center)}]

[decoration=brace]

\tikzset{
    node style ge/.style={circle,minimum size=.75cm},
}
\pgfdeclarelayer{background}
\pgfdeclarelayer{foreground}
\pgfsetlayers{background,main,foreground}
\matrix (A) [matrix of math nodes,
             nodes = {node style ge},
             left delimiter  = {[}, %brackets
             right delimiter = {]}, %brackets
             inner sep=-2pt, %this determines how far the matrix brackets scale outward around the entries.
             row sep=-.3cm, %custom row separation
             column sep=0cm %custom column separation (currently doing nothing)
             ]
{
%MATRIX ENTRIES
{\mathbf 0} & 0 & 0 & 0 & 0 & 0 & 0 \\
0 & {\mathbf 0} & 0 & {\mathbf 1} & 0 & 0 & 0 \\
0 & 0 & {\mathbf 0} & 0 & {\mathbf 1} & 0 & 0 \\
0 & 0 & 0 & 0 & 0 & 0 & 0 \\
0 & 0 & 0 & 0 & 0 & 0 & 0 \\
0 & 0 & 0 & 0 & 0 & {\mathbf 0} & {\mathbf 1} \\
0 & 0 & 0 & 0 & 0 & 0 & 0 \\
};
% BASIS across the top. The "above=-1pt" can be edited to bring them lower or higher.
\node[above=-1pt] at (A-1-1.north) {$ \tensor[_1]{\underline{a}}{_0}$};
\node[above=-1pt] at (A-1-2.north) {$\tensor[_1]{\underline{b}}{_0} $};
\node[above=-1pt] at (A-1-3.north) {$ \tensor[_3]{\underline{c}}{_0}$};
\node[above=-1pt] at (A-1-4.north) {$\tensor[_2]{\underline{ab}}{_1} $};
\node[above=-1pt] at (A-1-5.north) {$\tensor[_4]{\underline{bc}}{_1} $};
\node[above=-1pt] at (A-1-6.north) {$\tensor[_5]{\underline{ac}}{_1} $};
\node[above=-1pt] at (A-1-7.north) {$\tensor[_6]{\underline{abc}}{_2} $};
%BASIS down the side. The "left=6pt" determines how far out they are.
\node[left=6pt,yshift=-.05cm] at (A-1-1.west) {$\tensor[_1]{\underline{a}}{_0} $};
\node[left=6pt,yshift=-.05cm] at (A-2-1.west) {$\tensor[_1]{\underline{b}}{_0}  $};
\node[left=6pt,yshift=-.05cm] at (A-3-1.west) {$ \tensor[_3]{\underline{c}}{_0}   $};
\node[left=6pt,yshift=-.05cm] at (A-4-1.west) {$\tensor[_2]{\underline{ab}}{_1} $};
\node[left=6pt,yshift=-.05cm] at (A-5-1.west) {$\tensor[_4]{\underline{bc}}{_1} $};
\node[left=6pt,yshift=-.05cm] at (A-6-1.west) {$\tensor[_5]{\underline{ac}}{_1} $};
\node[left=6pt,yshift=-.05cm] at (A-7-1.west) {$\tensor[_6]{\underline{abc}}{_2} $};
%These are ANCHORS, invisible nodes in the middle/borders of the matrix that I used to draw the lines and make the shading.
%Just for reference, the "!0.5!" stuff is used to find halfway points between nodes to ensure that the lines are drawn exactly down the middle of two adjacent columns.

%here are the two anchors across the top
\node (top1) at ($(A-1-3.north)!0.5!(A-1-4.north)$) {};
\node (top2) at ($(A-1-6.north)!0.5!(A-1-7.north)$) {};

%anchors across the bottom
\node (bot1) at ($(A-7-3.south)!0.5!(A-7-4.south)$) {};
\node (bot2) at ($(A-7-6.south)!0.5!(A-7-7.south)$) {};

%anchors down the left
\node (left1) at ($(A-3-1.south west)!0.5!(A-4-1.north west)-(.15,0)$) {};
\node (left2) at ($(A-6-1.south west)!0.5!(A-7-1.north west)-(.15,0)$) {};

%anchors down the right
\node (right1) at ($(A-3-7.south east)!0.5!(A-4-7.north east)+(.15,0)$) {};
\node (right2) at ($(A-6-7.south east)!0.5!(A-7-7.north east)+(.15,0)$) {};

%Anchors in the middle of the matrix take two steps, as somehow averaging the center of 4 points at once didn't work.
%So these are just temporary nodes to be used below.
\node (a_1) at ($(A-3-3.south east)!0.5!(A-3-4.south west)$) {};
\node (a_2) at ($(A-4-3.north east)!0.5!(A-4-4.north west)$) {};
\node (b_1) at ($(A-3-6.south east)!0.5!(A-3-7.south west)$) {};
\node (b_2) at ($(A-4-6.north east)!0.5!(A-4-7.north west)$) {};

%the top two "intersection points"
\node (mid11) at ($(a_1)!0.5!(a_2)$) {};
\node (mid12) at ($(b_1)!0.5!(b_2)$) {};

\iffalse

\node (x_1) at ($(A -    i     - (j+1).south west)!0.5!(A-    i     - j.south east)$) {};
\node (x_2) at ($(A - (i+1) - (j+1).north west)!0.5!(A- (i+1) - j.north east)$) {};

\node (mid"ij") at ($(x_1)!0.5!(x_2)$) {};

\fi

%Horizontal and vertical lines

%\draw (left1) -- (right1);
%\draw (left2) -- (right2);
%\draw (top1) -- (bot1);
%\draw (top2) -- (bot2);

\begin{pgfonlayer}{background}

%Shaded blocks in a background layer so that they don't cover lines, entries, etc

\fill[black!20!white] (mid11) rectangle (top2.south);
\fill[black!20!white] (mid12) rectangle (right2.west);

\end{pgfonlayer}

\end{tikzpicture}
.
\end{equation*}
%************** end matrix ***********
\noindent We have retained the shading denoting the super-diagonal blocks, 
and   
we have also boldfaced the nonzero entries and the diagonal entries of zero columns.  
An almost-Jordan differential matrix $P^{-1} {\underline D} P$ is Jordan iff 
the matrix ${\underline D} P$, which is  related to  ${\underline D}$ by a permutation of columns, 
has  each boldfaced ${\mathbf 1}$ immediately following the boldfaced ${\mathbf 0}$ in the same row. 
Permuting columns 3 and 4 suffices for this example, and  
%************** start matrix ***********
 
 \begin{equation*}
P =
 \begin{tikzpicture}[baseline={([yshift=-1.4ex]current bounding box.center)}]

[decoration=brace]

\tikzset{
    node style ge/.style={circle,minimum size=.75cm},
}
\pgfdeclarelayer{background}
\pgfdeclarelayer{foreground}
\pgfsetlayers{background,main,foreground}
\matrix (A) [matrix of math nodes,
             nodes = {node style ge},
             left delimiter  = {[}, %brackets
             right delimiter = {]}, %brackets
             inner sep=-2pt, %this determines how far the matrix brackets scale outward around the entries.
             row sep=-.3cm, %custom row separation
             column sep=0cm %custom column separation (currently doing nothing)
             ]
{
%MATRIX ENTRIES
1 & 0 & 0 & 0 & 0 & 0 & 0 \\
0 & 1 & 0 & 0 & 0 & 0 & 0 \\
0 & 0 & 0 & 1 & 0 & 0 & 0 \\
0 & 0 & 1 & 0 & 0 & 0 & 0 \\
0 & 0 & 0 & 0 & 1 & 0 & 0 \\
0 & 0 & 0 & 0 & 0 & 1 & 0 \\
0 & 0 & 0 & 0 & 0 & 0 & 1 \\
};
% BASIS across the top. The "above=-1pt" can be edited to bring them lower or higher.
\node[above=-1pt] at (A-1-1.north) {$\tensor[_1]{\underline{a}}{_0} $};
\node[above=-1pt] at (A-1-2.north) {$\tensor[_1]{\underline{b}}{_0} $};
\node[above=-1pt] at (A-1-3.north) {$\tensor[_2]{\underline{ab}}{_1} $};
\node[above=-1pt] at (A-1-4.north) {$\tensor[_3]{\underline{c}}{_0} $};
\node[above=-1pt] at (A-1-5.north) {$\tensor[_4]{\underline{bc}}{_1} $};
\node[above=-1pt] at (A-1-6.north) {$\tensor[_5]{\underline{ac}}{_1} $};
\node[above=-1pt] at (A-1-7.north) {$\tensor[_6]{\underline{abc}}{_2} $};
%BASIS down the side. The "left=6pt" determines how far out they are.
\node[left=6pt,yshift=-.05cm] at (A-1-1.west) {$\tensor[_1]{\underline{a}}{_0} $};
\node[left=6pt,yshift=-.05cm] at (A-2-1.west) {$\tensor[_1]{\underline{b}}{_0} $};
\node[left=6pt,yshift=-.05cm] at (A-3-1.west) {$\tensor[_3]{\underline{c}}{_0} $};
\node[left=6pt,yshift=-.05cm] at (A-4-1.west) {$\tensor[_2]{\underline{ab}}{_1} $};
\node[left=6pt,yshift=-.05cm] at (A-5-1.west) {$\tensor[_4]{\underline{bc}}{_1} $};
\node[left=6pt,yshift=-.05cm] at (A-6-1.west) {$\tensor[_5]{\underline{ac}}{_1} $};
\node[left=6pt,yshift=-.05cm] at (A-7-1.west) {$\tensor[_6]{\underline{abc}}{_2}  $};

\end{tikzpicture}
\end{equation*}
%************** end matrix ***********
%
\noindent  produces the Jordan  matrix 
 \begin{equation*}
P^{-1} {\underline D} P =
 \begin{tikzpicture}[baseline={([yshift=-1.4ex]current bounding box.center)}]

[decoration=brace]

\tikzset{
    node style ge/.style={circle,minimum size=.75cm},
}
\pgfdeclarelayer{background}
\pgfdeclarelayer{foreground}
\pgfsetlayers{background,main,foreground}
\matrix (A) [matrix of math nodes,
             nodes = {node style ge},
             left delimiter  = {[}, %brackets
             right delimiter = {]}, %brackets
             inner sep=-2pt, %this determines how far the matrix brackets scale outward around the entries.
             row sep=-.3cm, %custom row separation
             column sep=0cm %custom column separation (currently doing nothing)
             ]
{
%MATRIX ENTRIES
{\mathbf 0} & 0 & 0 & 0 & 0 & 0 & 0 \\
0 & {\mathbf 0} & {\mathbf 1}& 0  & 0 & 0 & 0 \\
0 & 0 & 0 & 0 &0 & 0 & 0 \\
0 & 0 & 0 & {\mathbf 0} & {\mathbf 1} & 0 & 0 \\
0 & 0 & 0 & 0 & 0 & 0 & 0 \\
0 & 0 & 0 & 0 & 0 & {\mathbf 0} & {\mathbf 1} \\
0 & 0 & 0 & 0 & 0 & 0 & 0 \\
};
% BASIS across the top. The "above=-1pt" can be edited to bring them lower or higher.
\node[above=-1pt] at (A-1-1.north) {$\tensor[_1]{\underline{a}}{_0} $};
\node[above=-1pt] at (A-1-2.north) {$\tensor[_1]{\underline{b}}{_0}  $};
\node[above=-1pt] at (A-1-3.north) {$\tensor[_2]{\underline{ab}}{_1} $};
\node[above=-1pt] at (A-1-4.north) {$\tensor[_3]{\underline{c}}{_0} $};
\node[above=-1pt] at (A-1-5.north) {$ \tensor[_4]{\underline{bc}}{_1} $};
\node[above=-1pt] at (A-1-6.north) {$\tensor[_5]{\underline{ac}}{_1}  $};
\node[above=-1pt] at (A-1-7.north) {$\tensor[_6]{\underline{abc}}{_2}  $};
%BASIS down the side. The "left=6pt" determines how far out they are.
\node[left=6pt,yshift=-.05cm] at (A-1-1.west) {$\tensor[_1]{\underline{a}}{_0} $};
\node[left=6pt,yshift=-.05cm] at (A-2-1.west) {$\tensor[_1]{\underline{b}}{_0} $};
\node[left=6pt,yshift=-.05cm] at (A-3-1.west) {$ \tensor[_2]{\underline{ab}}{_1}  $};
\node[left=6pt,yshift=-.05cm] at (A-4-1.west) {$\tensor[_3]{\underline{c}}{_0}  $};
\node[left=6pt,yshift=-.05cm] at (A-5-1.west) {$\tensor[_4]{\underline{bc}}{_1} $};
\node[left=6pt,yshift=-.05cm] at (A-6-1.west) {$\tensor[_5]{\underline{ac}}{_1} $};
\node[left=6pt,yshift=-.05cm] at (A-7-1.west) {$\tensor[_6]{\underline{abc}}{_2} $};

\end{tikzpicture}
.
\end{equation*}
\noindent  
The  decomposition of the Jordan matrix into its Jordan blocks represents the 
decomposition of the filtered complex into indecomposable/basic summands.    
We now  list the indecomposable summands, denoting by $\langle v \rangle$ the 
linear span of a vector $v \in V$:

\begin{itemize}[leftmargin=*]
\item	 The Jordan block matrix  
%
%************** start matrix ***********
 \begin{tikzpicture}[baseline={([yshift=-1.4ex]current bounding box.center)}]

[decoration=brace]
\tikzset{
    node style ge/.style={circle,minimum size=.75cm},
}
\pgfdeclarelayer{background}
\pgfdeclarelayer{foreground}
\pgfsetlayers{background,main,foreground}
\matrix (A) [matrix of math nodes,
             nodes = {node style ge},
             left delimiter  = {[}, %brackets
             right delimiter = {]}, %brackets
             inner sep=-2pt,
		row sep=-.3cm,
		column sep=-.25cm
             ]
{
%MATRIX ENTRIES
\mathbf{0}  \\
};
% BASIS across the top. The "above=-1pt" can be edited to bring them lower or higher.
\node[above=-1pt] at (A-1-1.north) {$\tensor[_1]{\underline{a}}{_0} $};
%BASIS down the side. The "left=6pt" determines how far out they are.
\node[left=6pt,yshift=-.05cm] at (A-1-1.west) {$\tensor[_1]{\underline{a}}{_0} $};

\end{tikzpicture}
%************** end matrix ***********%
represents  the filtered complex  \hfil\break
\begin{equation*}%\begin{tightcenter} 
\begin{tikzcd}[row sep=tiny, column sep=tiny]
&  \vdots \ar{d} &  \vdots \ar{d}&  \vdots \ar{d}&  \vdots \ar{d} &   \vdots \ar{d} 
%&  \vdots \ar{d} 
& \vdots \ar{d} && & \vdots \ar{d}\\ 
 \cdots  \ar[hook]{r}  & 0   \ar{d} \ar[hook]{r} & 0   \ar{d} \ar[hook]{r} &0 \ar{d} \ar[hook]{r} &  0 \ar{d} \ar[hook]{r} 
 &  0  \ar{d} \ar[hook]{r} 
 &  0  \ar{d} \ar[hook]{r} 
 %& 0  \ar{d} \ar[hook]{r} 
 &  \cdots & &0 \ar{d}\\  
\cdots  \ar[hook]{r} &   0 \ar{d} \ar[hook]{r} &  \langle  \tensor[_1]{\underline{a}}{_0}  \rangle \ar{d} \ar[hook]{r}&  \langle  \tensor[_1]{\underline{a}}{_0}  \rangle  \ar{d} \ar[hook]{r} & 
 \langle  \tensor[_1]{\underline{a}}{_0}  \rangle \ar{d} \ar[hook]{r} &  \langle  \tensor[_1]{\underline{a}}{_0}  \rangle \ar{d} \ar[hook]{r} & \langle  \tensor[_1]{\underline{a}}{_0}  \rangle  \ar{d} \ar[hook]{r} 
 %& \langle  \tensor[_1]{\underline{a}}{_0}  \rangle \ar{d} \ar[hook]{r} 
 & \cdots & & \langle  \tensor[_1]{\underline{a}}{_0}  \rangle \ar{d} \\   
 \cdots  \ar[hook]{r}  & 0   \ar{d} \ar[hook]{r} & 0   \ar{d} \ar[hook]{r} &0  \ar{d} \ar[hook]{r} &  0  \ar{d} \ar[hook]{r} & 0  \ar{d} \ar[hook]{r} &
0  \ar{d} \ar[hook]{r} 
%&0  \ar{d} \ar[hook]{r} 
& \cdots & &0 \ar{d}\\  
  &  \vdots &  \vdots&  \vdots &  \vdots &  \vdots &\vdots 
 % &\vdots 
  & & & \vdots \\ 
&  p = 0   &  p = 1 &  p = 2 &  p = 3 &  p = 4 &   p = 5  
%& p = 6 
& & & {\rm colim}\\ 
     \end{tikzcd}
\end{equation*}%\begin{tightcenter} 
\noindent
This filtered complex is basic because in  ${\mathcal F}$ it  is isomorphic to the filtered complex 
\begin{equation*}%\begin{tightcenter} 
\begin{tikzcd}[row sep=tiny, column sep=tiny] 
      \cdots \ar[hook]{r}  & 0  \ar[hook]{r} & J[0]    
      \ar[hook]{r} & J[0]  \ar[hook]{r} &  J[0] \ar[hook]{r}& J[0]  \ar[hook]{r} & J[0]  \ar[hook]{r} 
      %& J[0]  \ar[hook]{r} 
      & \cdots & & J[0],  \\  
 &  p =0 & p = 1 & p = 2 &  p = 3 & p = 4 & p = 5 
 %& p = 6 
 &  & & {\rm colim} \\
     \end{tikzcd}
\end{equation*}%\begin{tightcenter} 
\noindent  which has the interval complex $J[0]$ as  colimit.    
So the filtered complex is an indecomposable  object in the category ${\mathcal F}$,  and also an indecomposable object in the 
quotient category $\coimage P_0$ where it is isomorphic to $[1,\infty)_0$.
Here the subscript $0$ labels the degree/dimension of the quotient category $ \coimage P_0$, and the interval 
subset  $[1, \infty) \subseteq {\mathbb Z}$ encodes the levels $p$ that are isomorphic to $J[0]$.   
The equivalence ${\coimage P_0} \to {\image P_0}$ corresponds to the 
homology functor $H_0$ acting on a diagram of complexes, producing  the 
 indecomposable  barcode persistence vector space $[1,\infty)$: 
\begin{equation*}%\begin{tightcenter} 
\begin{tikzcd}[row sep=tiny, column sep=tiny] 
   \cdots \ar{r}  & 0  \ar{r} & {\mathbb Q}   
      \ar{r} & {\mathbb Q}   \ar{r} & {\mathbb Q} \ar{r}& {\mathbb Q} \ar{r} & {\mathbb Q} \ar{r} 
      %& {\mathbb Q} \ar{r}
      & \cdots  . &&  \\  
  &  p =0 & p = 1 & p = 2 &  p = 3 & p = 4  & p = 5 
%  & p = 6 
  & & \phantom{ {\rm colim}}\\
     \end{tikzcd}
\end{equation*}%\begin{tightcenter} 
\noindent For any $n \ne 0$, the filtered complex is a zero object in the quotient category $ \coimage P_n$.  
  
\item	  The Jordan block matrix 
%************** start matrix ***********
 \begin{tikzpicture}[baseline={([yshift=-1.4ex]current bounding box.center)}]

[decoration=brace]
\tikzset{
    node style ge/.style={circle,minimum size=.75cm},
}
\pgfdeclarelayer{background}
\pgfdeclarelayer{foreground}
\pgfsetlayers{background,main,foreground}
\matrix (A) [matrix of math nodes,
             nodes = {node style ge},
             left delimiter  = {[}, %brackets
             right delimiter = {]}, %brackets
             inner sep=-2pt, %this determines how far the matrix brackets scale outward around the entries.
             row sep=-.3cm, %custom row separation
             column sep=0cm %custom column separation (currently doing nothing)
             ]
{
%MATRIX ENTRIES
\mathbf{0} & \mathbf{1} \\
0 & 0 \\
};
% BASIS across the top. The "above=-1pt" can be edited to bring them lower or higher.
\node[above=-1pt] at (A-1-1.north) {$\tensor[_1]{\underline{b}}{_0} $};
\node[above=-1pt] at (A-1-2.north) {$\tensor[_2]{\underline{ab}}{_1} $};
%BASIS down the side. The "left=6pt" determines how far out they are.
\node[left=6pt,yshift=-.05cm] at (A-1-1.west) {$\tensor[_1]{\underline{b}}{_0} $};
\node[left=6pt,yshift=-.05cm] at (A-2-1.west) {$\tensor[_2]{\underline{ab}}{_1}  $};

\node (corner_tr) at (A-1-2.north east) {};

\node (a_1) at ($(A-1-1.south east)!0.8!(A-1-2.south west)$) {};
\node (a_2) at ($(A-2-1.north east)!0.8!(A-2-2.north west)$) {};

\node (mid11) at ($(a_1)!0.2!(a_2)$) {};

\iffalse

\node (x_1) at ($(A -    i     - (j+1).south west)!0.5!(A-    i     - j.south east)$) {};
\node (x_2) at ($(A - (i+1) - (j+1).north west)!0.5!(A- (i+1) - j.north east)$) {};

\node (mid"ij") at ($(x_1)!0.5!(x_2)$) {};

\fi

\begin{pgfonlayer}{background}

\fill[black!20!white] (mid11) rectangle (corner_tr);

\end{pgfonlayer}

\end{tikzpicture}
%************** end matrix ***********%
represents the filtered complex   
\begin{equation*}%\begin{tightcenter} 
\begin{tikzcd}[row sep=tiny, column sep=tiny]
&  \vdots \ar{d} &  \vdots \ar{d}&  \vdots \ar{d}&  \vdots \ar{d} &   \vdots \ar{d} &  \vdots \ar{d} 
%& \vdots \ar{d} 
&& & \vdots \ar{d}\\ 
 \cdots  \ar[hook]{r}  & 0   \ar{d} \ar[hook]{r} & 0   \ar{d} \ar[hook]{r} &0 \ar{d} \ar[hook]{r} &  0 \ar{d} \ar[hook]{r} 
 &  0  \ar{d} \ar[hook]{r} &  0  \ar{d} \ar[hook]{r} 
 %& 0  \ar{d} \ar[hook]{r} 
 & \cdots & &0 \ar{d}\\  
  \cdots  \ar[hook]{r} &   0 \ar{d} \ar[hook]{r} & 0 \ar{d} \ar[hook]{r}& \langle \tensor[_2]{\underline{ab}}{_1} \rangle   \ar{d} \ar[hook]{r} & 
 \langle \tensor[_2]{\underline{ab}}{_1} \rangle \ar{d} \ar[hook]{r} &  \langle \tensor[_2]{\underline{ab}}{_1} \rangle \ar{d} \ar[hook]{r} &  \langle \tensor[_2]{\underline{ab}}{_1} \rangle  \ar{d} \ar[hook]{r} 
 %& \langle \tensor[_2]{\underline{ab}}{_1} \rangle  \ar{d} \ar[hook]{r} 
 &\cdots & & \langle \tensor[_2]{\underline{ab}}{_1} \rangle \ar{d}\\   
      \cdots  \ar[hook]{r} &   0 \ar{d} \ar[hook]{r} & \langle \tensor[_1]{\underline{b}}{_0} \rangle \ar{d} \ar[hook]{r}& \langle \tensor[_1]{\underline{b}}{_0} \rangle\ar{d} \ar[hook]{r} & 
 \langle \tensor[_1]{\underline{b}}{_0}\rangle \ar{d} \ar[hook]{r} &  \langle \tensor[_1]{\underline{b}}{_0} \rangle \ar{d} \ar[hook]{r} & \langle \tensor[_1]{\underline{b}}{_0}  \rangle\ar{d} \ar[hook]{r} 
 %& \langle \tensor[_1]{\underline{b}}{_0} \rangle \ar{d} \ar[hook]{r} 
 &\cdots & & \langle \tensor[_1]{\underline{b}}{_0} \rangle \ar{d} \\   
\cdots  \ar[hook]{r}  & 0   \ar{d} \ar[hook]{r} & 0   \ar{d} \ar[hook]{r} &0  \ar{d} \ar[hook]{r} &  0  \ar{d} \ar[hook]{r} & 0  \ar{d} \ar[hook]{r} &
0  \ar{d} \ar[hook]{r} 
%&0  \ar{d} \ar[hook]{r} 
& \cdots 
& &0 \ar{d}\\  
&  \vdots &  \vdots&  \vdots &  \vdots &  \vdots &\vdots 
%&\vdots 
& & & \vdots \\ 
 &  p = 0   &  p = 1 &  p = 2 &  p = 3 &  p = 4 &   p = 5  
% & p = 6 
 & & & {\rm colim}\\ 
     \end{tikzcd}
\end{equation*}%\begin{tightcenter} 
This filtered complex is basic because in ${\mathcal F}$ it is isomorphic to the filtered complex 
\begin{equation*}%\begin{tightcenter} 
\begin{tikzcd}[row sep=tiny, column sep=tiny] 
  \cdots \ar[hook]{r}  & 0  \ar[hook]{r} & J[0]    
      \ar[hook]{r} & K[0]  \ar[hook]{r} &  K[0] \ar[hook]{r}& K[0]  \ar[hook]{r} & K[0]  \ar[hook]{r} 
      %& K[0]  \ar[hook]{r} 
      & \cdots & & K[0],  \\  
 &  p =0 & p = 1 & p = 2 &  p = 3 & p = 4 & p = 5 
% & p = 6 
 &  & & {\rm colim} \\
     \end{tikzcd}
\end{equation*}%\begin{tightcenter} 
\noindent  which has the interval complex $K[0]$ as  colimit.    
So the filtered complex is an indecomposable  object in the category ${\mathcal F}$,  and also an indecomposable object in the 
quotient category $\coimage P_0$ where it is isomorphic to $[1,2)_0$.
Here the subscript $0$ labels the degree/dimension of the quotient category $ \coimage P_0$, and the interval 
subset  $[1, 2) \subseteq {\mathbb Z}$ encodes the levels $p$ that are isomorphic to $J[0]$.   
The equivalence ${\coimage P_0} \to {\image P_0}$ corresponds to the 
homology functor $H_0$ acting on a diagram of complexes, producing  the 
 indecomposable  barcode persistence vector space $[1,2)$:    
\begin{equation*}%\begin{tightcenter} 
\begin{tikzcd}[row sep=tiny, column sep=tiny] 
   \cdots \ar{r}  & 0  \ar{r} & {\mathbb Q}   
      \ar{r} &0  \ar{r} &0 \ar{r}&0 \ar{r} &0 \ar{r} 
      %&0 \ar{r}
      & \cdots  .&&  \\  
  &  p =0 & p = 1 & p = 2 &  p = 3 & p = 4  & p = 5 
  %& p = 6 
  & & \phantom{ {\rm colim}} \\
     \end{tikzcd}
\end{equation*}%\begin{tightcenter} 
\noindent For any $n \ne 0$, the filtered complex is a zero object in the quotient category $ \coimage P_n$.  
  
\item  The Jordan block matrix 
%
%************** start matrix ***********
 \begin{tikzpicture}[baseline={([yshift=-1.4ex]current bounding box.center)}]

[decoration=brace]
\tikzset{
    node style ge/.style={circle,minimum size=.75cm},
}
\pgfdeclarelayer{background}
\pgfdeclarelayer{foreground}
\pgfsetlayers{background,main,foreground}
\matrix (A) [matrix of math nodes,
             nodes = {node style ge},
             left delimiter  = {[}, %brackets
             right delimiter = {]}, %brackets
             inner sep=-2pt, %this determines how far the matrix brackets scale outward around the entries.
             row sep=-.3cm, %custom row separation
             column sep=0cm %custom column separation (currently doing nothing)
             ]
{
%MATRIX ENTRIES
\mathbf{0} & \mathbf{1} \\
0 & 0 \\
};
% BASIS across the top. The "above=-1pt" can be edited to bring them lower or higher.
\node[above=-1pt] at (A-1-1.north) {$\tensor[_3]{\underline{c}}{_0}$};
\node[above=-1pt] at (A-1-2.north) {$\tensor[_4]{\underline{bc}}{_1} $};
%BASIS down the side. The "left=6pt" determines how far out they are.
\node[left=6pt,yshift=-.05cm] at (A-1-1.west) {$\tensor[_3]{\underline{c}}{_0}$};
\node[left=6pt,yshift=-.05cm] at (A-2-1.west) {$\tensor[_4]{\underline{bc}}{_1}$};

\node (corner_tr) at (A-1-2.north east) {};

\node (a_1) at ($(A-1-1.south east)!0.8!(A-1-2.south west)$) {};
\node (a_2) at ($(A-2-1.north east)!0.8!(A-2-2.north west)$) {};

\node (mid11) at ($(a_1)!0.2!(a_2)$) {};

\iffalse

\node (x_1) at ($(A -    i     - (j+1).south west)!0.5!(A-    i     - j.south east)$) {};
\node (x_2) at ($(A - (i+1) - (j+1).north west)!0.5!(A- (i+1) - j.north east)$) {};

\node (mid"ij") at ($(x_1)!0.5!(x_2)$) {};

\fi

\begin{pgfonlayer}{background}

\fill[black!20!white] (mid11) rectangle (corner_tr);

\end{pgfonlayer}

\end{tikzpicture}
%************** end matrix ***********
 %
represents the filtered complex    
\begin{equation*}%\begin{tightcenter} 
\begin{tikzcd}[row sep=tiny, column sep=tiny]
%
%&  \vdots \ar{d} 
&  \vdots \ar{d}&  \vdots \ar{d}&  \vdots \ar{d} &   \vdots \ar{d} &  \vdots \ar{d} & \vdots \ar{d} && & \vdots \ar{d}\\ 
 \cdots  \ar[hook]{r}  
%& 0   \ar{d} \ar[hook]{r} 
 & 0   \ar{d} \ar[hook]{r} &0 \ar{d} \ar[hook]{r} &  0 \ar{d} \ar[hook]{r} 
 &  0  \ar{d} \ar[hook]{r} &  0  \ar{d} \ar[hook]{r} & 0  \ar{d} \ar[hook]{r} &
  \cdots 
 & &0 \ar{d}\\  
    \cdots  \ar[hook]{r} 
% &   0 \ar{d} \ar[hook]{r} 
    & 0 \ar{d} \ar[hook]{r}&0 \ar{d} \ar[hook]{r} & 
0 \ar{d} \ar[hook]{r} &  \langle  \tensor[_4]{\underline{bc}}{_1} \rangle  \ar{d} \ar[hook]{r} &  \langle  \tensor[_4]{\underline{bc}}{_1} \rangle  \ar{d} \ar[hook]{r} &  \langle  \tensor[_4]{\underline{bc}}{_1} \rangle  \ar{d} \ar[hook]{r} &
\cdots & & \langle  \tensor[_4]{\underline{bc}}{_1} \rangle \ar{d} \\   
     \cdots  \ar[hook]{r} 
%&   0 \ar{d} \ar[hook]{r} 
& 0 \ar{d} \ar[hook]{r}   &0  \ar{d} \ar[hook]{r} & 
\langle \tensor[_3]{\underline{c}}{_0} \rangle  \ar{d} \ar[hook]{r} & \langle \tensor[_3]{\underline{c}}{_0} \rangle \ar{d} \ar[hook]{r} & \langle \tensor[_3]{\underline{c}}{_0} \rangle  \ar{d} \ar[hook]{r} & \langle \tensor[_3]{\underline{c}}{_0} \rangle  \ar{d} \ar[hook]{r} &
\cdots & &\langle \tensor[_3]{\underline{c}}{_0} \rangle \ar{d} \\   
 \cdots  \ar[hook]{r}  
%& 0   \ar{d} \ar[hook]{r} 
 & 0   \ar{d} \ar[hook]{r} &0  \ar{d} \ar[hook]{r} &  0  \ar{d} \ar[hook]{r} & 0  \ar{d} \ar[hook]{r} &
0  \ar{d} \ar[hook]{r} &0  \ar{d} \ar[hook]{r} & \cdots 
& &0 \ar{d}\\  
%
%&  \vdots 
 &  \vdots&  \vdots &  \vdots &  \vdots &\vdots &\vdots & & & \vdots \\ 
%
%&  p = 0   
&  p = 1 &  p = 2 &  p = 3 &  p = 4 &   p = 5  & p = 6 & & & {\rm colim}\\ 
     \end{tikzcd}
\end{equation*}%\begin{tightcenter} 
This filtered complex is basic because in ${\mathcal F}$ it is isomorphic to the filtered complex 
\begin{equation*}%\begin{tightcenter} 
\begin{tikzcd}[row sep=tiny, column sep=tiny] 
   \cdots \ar[hook]{r}  
 %  & 0  \ar[hook]{r} 
   & 0    \ar[hook]{r} &0  \ar[hook]{r} &  J[0] \ar[hook]{r}& K[0]  \ar[hook]{r} & K[0]  \ar[hook]{r} & K[0]  \ar[hook]{r} & \cdots & & K[0]  \\  
%&  p =0 
& p = 1 & p = 2 &  p = 3 & p = 4 & p = 5 & p = 6 &  & & {\rm colim} \\
     \end{tikzcd}
\end{equation*}%\begin{tightcenter} 
\noindent  which has the interval complex $K[0]$ as  colimit.    
So the filtered complex is an indecomposable  object in the category ${\mathcal F}$,  and also an indecomposable object in the 
quotient category $\coimage P_0$ where it is isomorphic to $[3,4)_0$.
Here the subscript $0$ labels the degree/dimension of the quotient category $ \coimage P_0$, and the interval 
subset  $[3,4) \subseteq {\mathbb Z}$ encodes the levels $p$ that are isomorphic to $J[0]$.   
The equivalence ${\coimage P_0} \to {\image P_0}$ corresponds to the 
homology functor $H_0$ acting on a diagram of complexes, producing  the 
 indecomposable  barcode persistence vector space $[3,4)$:    
\begin{equation*}%\begin{tightcenter} 
\begin{tikzcd}[row sep=tiny, column sep=tiny] 
   \cdots \ar{r}  
%   & 0  \ar{r} 
   & 0    \ar{r} &0  \ar{r} & {\mathbb Q} \ar{r}&0 \ar{r} &0 \ar{r} &0 \ar{r}& \cdots  . &&  \\  
%  &  p =0 
  & p = 1 & p = 2 &  p = 3 & p = 4  & p = 5 & p = 6 & & \phantom{ {\rm colim}} \\
     \end{tikzcd}
\end{equation*}%\begin{tightcenter} 
\noindent For any $n \ne 0$, the filtered complex is a zero object in the quotient category $ \coimage P_n$.  

\item	 The Jordan block matrix 
%************** start matrix ***********
 \begin{tikzpicture}[baseline={([yshift=-1.4ex]current bounding box.center)}]

[decoration=brace]
\tikzset{
    node style ge/.style={circle,minimum size=.75cm},
}
\pgfdeclarelayer{background}
\pgfdeclarelayer{foreground}
\pgfsetlayers{background,main,foreground}
\matrix (A) [matrix of math nodes,
             nodes = {node style ge},
             left delimiter  = {[}, %brackets
             right delimiter = {]}, %brackets
             inner sep=-2pt, %this determines how far the matrix brackets scale outward around the entries.
             row sep=-.3cm, %custom row separation
             column sep=0cm %custom column separation (currently doing nothing)
             ]
{
%MATRIX ENTRIES
\mathbf{0} & \mathbf{1} \\
0 & 0 \\
};
% BASIS across the top. The "above=-1pt" can be edited to bring them lower or higher.
\node[above=-1pt] at (A-1-1.north) {$\tensor[_5]{\underline{ac}}{_1}$};
\node[above=-1pt] at (A-1-2.north) {$\tensor[_6]{\underline{abc}}{_2} $};
%BASIS down the side. The "left=6pt" determines how far out they are.
\node[left=6pt,yshift=-.05cm] at (A-1-1.west) {$\tensor[_5]{\underline{ac}}{_1}$};
\node[left=6pt,yshift=-.05cm] at (A-2-1.west) {$\tensor[_6]{\underline{abc}}{_2} $};

\node (corner_tr) at (A-1-2.north east) {};

\node (a_1) at ($(A-1-1.south east)!0.8!(A-1-2.south west)$) {};
\node (a_2) at ($(A-2-1.north east)!0.8!(A-2-2.north west)$) {};

\node (mid11) at ($(a_1)!0.2!(a_2)$) {};

\iffalse

\node (x_1) at ($(A -    i     - (j+1).south west)!0.5!(A-    i     - j.south east)$) {};
\node (x_2) at ($(A - (i+1) - (j+1).north west)!0.5!(A- (i+1) - j.north east)$) {};

\node (mid"ij") at ($(x_1)!0.5!(x_2)$) {};

\fi

\begin{pgfonlayer}{background}

\fill[black!20!white] (mid11) rectangle (corner_tr);

\end{pgfonlayer}

\end{tikzpicture}
%************** end matrix ***********
%
represents the filtered complex   
\begin{equation*}%\begin{tightcenter} 
\begin{tikzcd}[row sep=tiny, column sep=tiny]
%
%&  \vdots \ar{d} 
&  \vdots \ar{d}&  \vdots \ar{d}&  \vdots \ar{d} &   \vdots \ar{d} &  \vdots \ar{d} & \vdots \ar{d} && & \vdots \ar{d}\\ 
 \cdots  \ar[hook]{r}  
% & 0   \ar{d} \ar[hook]{r} 
 & 0   \ar{d} \ar[hook]{r} &0 \ar{d} \ar[hook]{r} &  0 \ar{d} \ar[hook]{r} 
 &  0  \ar{d} \ar[hook]{r} &  0  \ar{d} \ar[hook]{r} & 0  \ar{d} \ar[hook]{r} &
  \cdots 
 & &0 \ar{d}\\  
       \cdots  \ar[hook]{r} 
   %    &   0 \ar{d} \ar[hook]{r} 
       & 0 \ar{d} \ar[hook]{r}&0 \ar{d} \ar[hook]{r} & 
0 \ar{d} \ar[hook]{r} &  0 \ar{d} \ar[hook]{r} &0  \ar{d} \ar[hook]{r} & \langle \tensor[_6]{\underline{abc}}{_2} \rangle  \ar{d} \ar[hook]{r} &
\cdots & & \langle \tensor[_6]{\underline{abc}}{_2} \rangle  \ar{d} \\   
     \cdots  \ar[hook]{r} 
 %    &   0 \ar{d} \ar[hook]{r} 
     & 0 \ar{d} \ar[hook]{r}&0  \ar{d} \ar[hook]{r} & 
0 \ar{d} \ar[hook]{r} &  0 \ar{d} \ar[hook]{r} &\langle \tensor[_5]{\underline{ac}}{_1} \rangle \ar{d} \ar[hook]{r} & \langle \tensor[_5]{\underline{ac}}{_1} \rangle  \ar{d} \ar[hook]{r} &
\cdots & &\langle \tensor[_5]{\underline{ac}}{_1} \rangle\ar{d} \\   
 \cdots  \ar[hook]{r}  
% & 0   \ar{d} \ar[hook]{r} 
 & 0   \ar{d} \ar[hook]{r} &0  \ar{d} \ar[hook]{r} &  0  \ar{d} \ar[hook]{r} & 0  \ar{d} \ar[hook]{r} &
0  \ar{d} \ar[hook]{r} &0  \ar{d} \ar[hook]{r} & \cdots 
& &0 \ar{d}\\  
%
% &  \vdots 
 &  \vdots&  \vdots &  \vdots &  \vdots &\vdots &\vdots & & & \vdots \\ 
%
%&  p = 0   
&  p = 1 &  p = 2 &  p = 3 &  p = 4 &   p = 5  & p = 6 & & & {\rm colim}\\ 
     \end{tikzcd}
\end{equation*}%\begin{tightcenter} 
This filtered complex is basic because in ${\mathcal F}$ it is isomorphic to the filtered complex 
\begin{equation*}%\begin{tightcenter} 
\begin{tikzcd}[row sep=tiny, column sep=tiny] 
   \cdots \ar[hook]{r}  
%   & 0  \ar[hook]{r} 
   & 0     \ar[hook]{r} &0  \ar[hook]{r} &  0 \ar[hook]{r}& 0  \ar[hook]{r} & J[1] \ar[hook]{r} & K[1]  \ar[hook]{r} & \cdots & & K[1]  \\  
%&  p =0 
& p = 1 & p = 2 &  p = 3 & p = 4 & p = 5 & p = 6 &  & & {\rm colim} \\
     \end{tikzcd}
\end{equation*}%\begin{tightcenter} 
\noindent  which has the interval complex $K[1]$ as  colimit.    
So the filtered complex is an indecomposable  object in the category ${\mathcal F}$,  and also an indecomposable object in the 
quotient category $\coimage P_1$ where it is isomorphic to $[5,6)_1$.
Here the subscript $1$ labels the degree/dimension of the quotient category $ \coimage P_1$, and the interval 
subset  $[5,6) \subseteq {\mathbb Z}$ encodes the levels $p$ that are isomorphic to $J[1]$.   
The equivalence ${\coimage P_1} \to {\image P_1}$ corresponds to the 
homology functor $H_1$ acting on a diagram of complexes, producing  the 
 indecomposable  barcode persistence vector space $[5,6)$: 
\begin{equation*}%\begin{tightcenter} 
\begin{tikzcd}[row sep=tiny, column sep=tiny] 
   \cdots \ar{r}  
 %  & 0  \ar{r} 
   & 0     \ar{r} &0  \ar{r} & 0 \ar{r}&0 \ar{r} &{\mathbb Q} \ar{r} &0 \ar{r}& \cdots  &&  \\  
% &  p =0 
 & p = 1 & p = 2 &  p = 3 & p = 4  & p = 5 & p = 6 & & \phantom{ {\rm colim}} \\
     \end{tikzcd}
\end{equation*}%\begin{tightcenter} 
\noindent For any $n \ne1$, the filtered complex is a zero object in the quotient category $ \coimage P_n$.   

\end{itemize}
This completes the decomposition of the filtered complex $F$ in the category ${\mathcal F}$.  
As an object in the quotient category ${\coimage P_0}$, the filtered complex  $F$ is  isomorphic to $[0,\infty)_0 \oplus [1,2)_0 \oplus [3,4)_0$. 
As an object in the quotient category ${\coimage P_1}$,  the filtered complex  $F$ is  isomorphic to $[5,6)_1$. 
For any other value of $n$, the filtered complex $F$ is a zero object in the quotient category $\coimage P_n$.

 \end{example}

 \subsection{Reverse Structural Equivalence}  
 \label{subsection4.2}  

Special adapted bases help to intermediate between the Matrix Structural Theorem and Categorical Structural Theorem.  
In Proposition \ref{forego},  we established the existence of a special adapted basis using 
the Matrix Structural Theorem \ref{postnon}.       
Now in the reverse direction,  we establish the existence of a special adapted basis using the Categorical Structural Theorem \ref{classFilt}: 
 \begin{proposition}
 \label{ss} 
A filtered complex admits a special adapted basis.   
 \end{proposition} 
 
 \begin{proof} 
 The Categorical Structural Theorem decomposes the filtered complex as a finite direct sum of indecomposables.  
Each indecomposable summand is a basic filtered complex, so it admits a special adapted basis.    
With appropriate ordering, 
the union over the summands of these basis elements  is  
a special adapted basis for the direct sum filtered complex.   
 \qed \end{proof} 

An automorphism of a filtered complex transforms an adapted basis to another adapted basis.  
The change of basis is represented by a matrix $B$, which is block-diagonal because an automorphism preserves the degree 
of basis elements.    But the matrix $B$ need not be triangular in general.    
We call a filtered complex  {\it nondegenerate}  if $\dim( \,  \tensor[_{p+1}]{V} {_n})   \le    \dim  (\, \tensor[_{p}]{V}{_n})+1$ 
for any $p$ and any $n$. 

\begin{lemma} 
\label{trideg}
If a filtered complex is nondegenerate, then any change of adapated basis 
is represented by  a triangular matrix $B$.  
\end{lemma} 

\begin{proof} 
An automorphism takes a basis element of degree $n$ and level $p$ 
to a linear combination of basis elements of degree $n$ and level at most $p$.     
A filtered complex is nondegenerate iff  an adapted basis contains no pair of elements with  the same degree and same  level.  
In this case the linear combination does not contain any basis elements that appear later in the ordering of the basis. 
The matrix $B$ is then triangular, since it has no nonzero entries below the diagonal.   
\qed \end{proof}

\noindent 
We will construct nondegenerate filtered complexes by using the upper-left submatrices 
of a differential matrix.  We illustrate  submatrices with an example:     
\begin{example}   
The  upper-left submatrices are indicated below for the block-superdiagonal differential matrix $D: {\mathbb Q}^7 \to {\mathbb Q}^7$ 
given by: 
% 

%************** start matrix ***********
\begin{equation*}
D =
 \begin{tikzpicture}[baseline={([yshift=-1.4ex]current bounding box.center)}]

[decoration=brace]
\tikzset{
    node style ge/.style={circle,minimum size=.75cm},
}
\pgfdeclarelayer{background}
\pgfdeclarelayer{foreground}
\pgfsetlayers{background,main,foreground}
\matrix (A) [matrix of math nodes,
             nodes = {node style ge},
             left delimiter  = {[}, %brackets
             right delimiter = {]}, %brackets
             inner sep=-2pt,
		row sep=-.3cm,
		column sep=-.25cm
             ]
{
%MATRIX ENTRIES
0 & 0 &  0 & {-1} & 0 &{-1}  & 0 \\
0 & 0 & 0 & 1 & -1 & 0 & 0 \\
0 & 0 &  0 &  0 & 1 &   1 &  0 \\
0 & 0 & 0 & 0 & 0 & 0 &1 \\
0 & 0 & 0 & 0 & 0 & 0 & 1 \\
0 & 0 & 0 & 0 & 0 & 0 & {-1} \\
0 & 0 & 0 & 0 & 0 & 0 & 0 \\
};
% BASIS across the top. The "above=-1pt" can be edited to bring them lower or higher.
%BASIS down the side. The "left=6pt" determines how far out they are.
%These are ANCHORS, invisible nodes in the middle/borders of the matrix that I used to draw the lines and make the shading.
%Just for reference, the "!0.5!" stuff is used to find halfway points between nodes to ensure that the lines are drawn exactly down the middle of two adjacent columns.

%here are the two anchors across the top
\node (top1) at ($(A-1-1.north)!0.5!(A-1-2.north)$) {};
\node (top2) at ($(A-1-2.north)!0.5!(A-1-3.north)$) {};
\node (top3) at ($(A-1-3.north)!0.5!(A-1-4.north)$) {};
\node (top4) at ($(A-1-4.north)!0.5!(A-1-5.north)$) {};
\node (top5) at ($(A-1-5.north)!0.5!(A-1-6.north)$) {};
\node (top6) at ($(A-1-6.north)!0.5!(A-1-7.north)$) {};

%anchors across the bottom
%\node (bot1) at ($(A-7-1.south east)!0.5!(A-7-2.south west)$) {};

%anchors down the left
\node (left1) at ($(A-1-1.south west)!0.5!(A-2-1.north west)-(.15,0)$) {};
\node (left2) at ($(A-2-1.south west)!0.5!(A-3-1.north west)-(.15,0)$) {};
\node (left3) at ($(A-3-1.south west)!0.5!(A-4-1.north west)-(.15,0)$) {};
\node (left4) at ($(A-4-1.south west)!0.5!(A-5-1.north west)-(.15,0)$) {};
\node (left5) at ($(A-5-1.south west)!0.5!(A-6-1.north west)-(.15,0)$) {};
\node (left6) at ($(A-6-1.south west)!0.5!(A-7-1.north west)-(.15,0)$) {};

%anchors down the right
\node (right3) at ($(A-3-7.south east)!0.5!(A-4-7.north east)+(.15,0)$) {};

%Anchors in the middle of the matrix take two steps, as somehow averaging the center of 4 points at once didn't work.
%So these are just temporary nodes to be used below.
\node (a_1) at ($(A-1-1.south east)!0.5!(A-1-2.south west)$) {};
\node (a_2) at ($(A-2-1.north east)!0.5!(A-2-2.north west)$) {};

\node (b_1) at ($(A-2-2.south east)!0.5!(A-2-3.south west)$) {};
\node (b_2) at ($(A-3-2.north east)!0.5!(A-3-3.north west)$) {};

\node (c_1) at ($(A-3-3.south east)!0.5!(A-3-4.south west)$) {};
\node (c_2) at ($(A-4-3.north east)!0.5!(A-4-4.north west)$) {};

\node (d_1) at ($(A-4-4.south east)!0.5!(A-4-5.south west)$) {};
\node (d_2) at ($(A-5-4.north east)!0.5!(A-5-5.north west)$) {};

\node (e_1) at ($(A-5-5.south east)!0.5!(A-5-6.south west)$) {};
\node (e_2) at ($(A-6-5.north east)!0.5!(A-6-6.north west)$) {};

\node (f_1) at ($(A-6-6.south east)!0.5!(A-6-7.south west)$) {};
\node (f_2) at ($(A-7-6.north east)!0.5!(A-7-7.north west)$) {};

%the top two "intersection points"
\node (mid11) at ($(a_1)!0.5!(a_2)$) {};
\node (mid22) at ($(b_1)!0.5!(b_2)$) {};
\node (mid33) at ($(c_1)!0.5!(c_2)$) {};
\node (mid44) at ($(d_1)!0.5!(d_2)$) {};
\node (mid55) at ($(e_1)!0.5!(e_2)$) {};
\node (mid66) at ($(f_1)!0.5!(f_2)$) {};

\iffalse

\node (x_1) at ($(A -    i     - (j+1).south west)!0.5!(A-    i     - j.south east)$) {};
\node (x_2) at ($(A - (i+1) - (j+1).north west)!0.5!(A- (i+1) - j.north east)$) {};

\node (mid"ij") at ($(x_1)!0.5!(x_2)$) {};

\fi

%Horizontal and vertical lines

\draw (top1) -- (mid11.center) -- (left1);
\draw (top2) -- (mid22.center) -- (left2);
\draw (top3) -- (mid33.center) -- (left3);
\draw (top4) -- (mid44.center) -- (left4);
\draw (top5) -- (mid55.center) -- (left5);
\draw (top6) -- (mid66.center) -- (left6);

\begin{pgfonlayer}{background}

%Shaded blocks in a background layer so that they don't cover lines, entries, etc

\fill[black!20!white] (mid33) rectangle (top6.south);
\fill[black!20!white] (mid66) rectangle (right3.west);

\end{pgfonlayer}

\end{tikzpicture}
.
\end{equation*}
%************** end matrix ***********
%
Note that for  each integer $0 <  p < 7$, the  
upper-left submatrix $\tensor[_p]{D}{}: {\mathbb Q}^p \to {\mathbb Q}^p$ 
is itself a block-superdiagonal differential matrix.  
We remark that the matrix $D$ had appeared previously in  Example \ref{exfilt}, 
representing  the degenerate (not nondegenerate) 
filtered complex of Example \ref{explain}.  
\end{example}

\begin{lemma}   
 \label{revlem} 
 Any block-superdiagonal differential matrix $D$ represents the colimit boundary  of some nondegenerate filtered complex. 
\end{lemma} 
 
 \begin{proof}  
 Let $D: {\mathbb F}^m \to {\mathbb F}^m$ be a  block-superdiagonal differential matrix.   
 We  construct a filtered complex 
\begin{equation*}%\begin{tightcenter} 
\begin{tikzcd}[row sep=tiny, column sep=tiny] 
   \cdots \ar[hook]{r}  & \tensor[_{-1}]{V}{_\bullet}  \ar[hook]{r} & \tensor[_0]{V}{_\bullet}    
      \ar[hook]{r} & \tensor[_1]{V}{_\bullet}  \ar[hook]{r} &  \tensor[_2]{V}{_\bullet} \ar[hook]{r}& \tensor[_3]{V}{_\bullet} 
       \ar[hook]{r} & \cdots & & V_\bullet  \\  
 & \phantom{p = -1}   &  \phantom{p = 0} &  \phantom{p = 1} &  \phantom{p = 2} &  \phantom{p = 3} &  & & {\rm colim}\\ 
     \end{tikzcd}
\end{equation*}%\begin{tightcenter} 
by  specifying for each integer $p$ the complex $\tensor[_p]{V}{_\bullet}$ at  level $p$: 
\begin{itemize}[leftmargin=*]
\item For $p \le 0$, the complex is the zero complex.    
\item For  $1 <  p < m$,  the complex  is 
%
% added 1-11-17
specified by the block-superdiagonal differential submatrix
$\tensor[_p]{D}{}: {\mathbb F}^p \to {\mathbb F}^p$. 
\item For $m \le p$, the complex  is 
%
% added 1-11-17
specified by the initial block-superdiagonal differential matrix 
$D:{\mathbb F}^m \to {\mathbb F}^m$. 
\end{itemize} 
The arrows are the subobject inclusions $\tensor[_p]{V}{_\bullet}  \hookrightarrow \tensor[_{p+1}]{V}{_\bullet}$. 
Then the diagram is a filtered complex since the zero complex is a limit and  the complex $D:{\mathbb F}^m \to {\mathbb F}^m$ is a colimit. 
It only remains to observe that the filtered complex is nondegenerate, and that  the matrix $D$ represents its colimit boundary.   
  \qed \end{proof} 
%
% added 1-11-17
\noindent Note that the block structure of the differential matrix $D$ is important in the preceding proof.   
If a differential matrix does not have block-superdiagonal structure, then an upper-left submatrix 
need not be a differential matrix in general.   

Now we have assembled the ingredients to prove:  
 
 % manually number repeat proposition
 \setcounter{repsection}{1} 
\setcounter{repeat}{7}
\begin{repproposition}
 (Reverse Structural Equivalence) 
 The Categorical Structural Theorem implies the Matrix  Structural Theorem. 
\end{repproposition}

\begin{proof}      Let $D$ be a block-superdiagonal differential matrix.  
Lemma \ref{revlem} lets us choose a nondegenerate filtered complex that is represented by $D$. 
Proposition \ref{ss} lets us make a change of basis to a special adapted basis.  
The block-diagonal matrix $B$ representing the basis change is 
triangular by  Lemma \ref{trideg}.  
Finally, the  block-superdiagonal differential ${\underline D} = B^{-1} D B$ is 
almost-Jordan because the adapted basis is special.  
\qed \end{proof} 

%***** start revision 

%\vfill\eject

\section{Concluding Remarks and Directions for Further Development}
\label{section5}

\subsection{Encoding Homology}
\label{subsection5.1}

Section \ref{subsection3.2} presents  our {\it alternate framework} for persistent homology, 
based on the Krull-Schmidt decomposition of a filtered complex  afforded    
by the Categorical Structural  Theorem \ref{classFilt}.     
In this section we outline the analogous alternate framework for  the homology of   
``plain" ({\it i.e.} not filtered) complexes (see Section \ref{subsection2.3}).    The encoding of homology within a decomposition is 
 easier to explain in this  simpler setting, 
and the explanation carries over {\it mutatis mutandis} to the more complicated persistent homology framework.  
The basic idea is to ``compute" homology (or persistent homology) by discarding from a decomposition those summands that 
are {\it a priori} known to have zero homology (respectively  persistent homology).     
Note that this idea cannot be implemented in 
all situations.   Even for plain homology, it  works with coefficients in a field ${\mathbb F}$, but  
fails in the fundamental case of integer coefficients.  
For persistent homology, it  works for the ``ordinary" case of filtered complexes as discussed previously, but fails for 
  for zigazag persistent homology as discussed in Section 5.3 below.

The category  ${\mathcal C}$ of chain complexes is   Krull-Schmidt (by Proposition \ref{comp}),   and 
any indecomposable complex is isomorphic  to $J[m]$ or to $K[m]$ for some integer $m$.  
Denote by $H_n: {\mathcal C} \to {\mathcal V}$  the  degree-$n$ homology functor  from  ${\mathcal C}$ to the category ${\mathcal V}$ of 
vector spaces (see Section \ref{subsection3.1}).   
The following result is the analogue of Theorem \ref{alternate} for this simpler setting:     
 \begin{proposition}
\label{dumalt}  The (plain) homology functor $H_n: {\mathcal C} \to {\mathcal V}$ factors as 
 \begin{equation*} {\mathcal C} \to {\coimage H_n}\to {\image H_n} \to {\mathcal V}, 
\end{equation*}
where the functor ${\coimage H_n}\to {\image H_n}$ is an equivalence of categories.   
\end{proposition} 
In this setting ${\image H_n}$ is just another name for $ {\mathcal V}$, and ${\image H_n} \to {\mathcal V}$ is the 
identity.    
  The interesting part is the quotient functor ${\mathcal C} \to {\coimage H_n}$, 
  where the  quotient category ${\coimage H_n}$ is defined via the following equivalence relation (congruence) on morphisms:  
two morphisms $f$ and $f'$ in ${\mathcal C}$ are equivalent iff the morphisms $H_n(f)$ and $H_n(f')$ in ${\mathcal V}$ are equal. 
By the Krull-Schmidt property, any complex $C$ in ${\mathcal C}$ is isomorphic to a direct sum with appropriate 
multiplicities of the  indecomposable complexes $J[m]$ and $K[m]$ for various  $m$.   
The key property required to encode homology in this framework is:  an 
 indecomposable complex goes to zero under the quotient functor iff it 
goes to zero under the homology functor $H_n$.  Namely, $J[m]$ goes to zero unless $m = n$, and 
$K[m]$ goes to zero for all $m$.  
So working with a complex $C$ in the quotient category ${\coimage H_n}$ amounts to discarding from a decomposition 
of $C$  those indecomposable summands that are {\it a priori} known to go to zero under the homology functor $H_n$.     
Each indecomposable summand that remains is canonically isomorphic to $J[n]$, and the set of 
these isomorphisms contains the data for the usual  ``basis of homology cycles" of the homology vector space $H_n(C)$.

\subsection{Kernels and Cokernels}
\label{subsection5.2}

It is well-known that  the representations of a quiver constitute an Abelian category, see 
for example \cite{Shiffler}.   This means that Abelian categories are 
relevant to persistent homology,  and this has been studied in the paper  \cite{CHEM}.   
We now rapidly review the fundamental constructs in an  Abelian category, 
referring to Freyd's classic \cite{Freyd} or the more modern approach of \cite{HA} for details. 
Recall that an additive category is  {\it pre-Abelian} if any morphism 
$f: X \to Y$ admits a kernel, $\kernel f \to X$,  and a cokernel, $Y \to \cokernel f$, 
each characterized by standard universal properties.   Then the 
 image, $\image f \to Y$, is  defined as the kernel of the cokernel, 
and the coimage, $X \to \coimage f$ as the cokernel of the kernel.  
Any morphism $f: X \to Y$ in a pre-Abelian category  factors uniquely as  (\cite{HA} Lemma 3.12): 
\begin{equation*} X \to {\coimage f}\to {\image f} \to Y. 
\end{equation*}
Finally, a  pre-Abelian category is Abelian iff  ${\coimage f}\to {\image f}$ is always an isomorphism.  
(\cite{HA} Definition 5.1; this is widely known as the ``rank theorem" for   
the Abelian category ${\mathcal V}$ of finite-dimensional vector spaces.)

The standard framework for persistent homology (Section \ref{subsection3.1})  focuses on the 
 category ${\image P_n}$.  
 The category ${\image P_n}$ is Abelian, 
so each morphism $f: X \to Y$ has a kernel, 
 cokernel, image, and coimage.   Furthermore the category ${\image P_n}$ is  Krull-Schmidt, 
 so the objects $\kernel f$, $\cokernel f$, $\image f$, and $\coimage f$ can 
 be decomposed in terms of barcodes.  The paper  \cite{CHEM} presents 
 algorithms for computing the barcode invariants of these objects for the case when $f: X \hookrightarrow Y$ is the inclusion of a subobject, 
 and discusses  the case of general $f$ in terms of mapping cylinders.    

Our alternate framework for persistent homology (Section \ref{subsection3.2}) focuses on the quotient category ${\coimage P_n}$.  
Theorem \ref{alternate} asserts that the  persistent homology functor 
$P_n: {\mathcal F} \to  {\mathcal P}$ 
factors as 
\begin{equation*} {\mathcal F} \to {\coimage P_n}\to {\image P_n} \to {\mathcal P}, 
\end{equation*}
where  the functor ${\coimage P_n}\to {\image P_n}$ is an equivalence 
of categories.  Since ${\image P_n}$ is Abelian, the equivalence immediately implies  that the quotient category ${\coimage P_n}$ is also Abelian.  
In a forthcoming paper \cite{PS}, we  study algorithms for constructing 
$\kernel f \to X$ and $Y \to \cokernel f$ for a general morphism $f: X \to Y$ in the quotient category ${\coimage P_n}$.  
(We note  that the category ${\mathcal F}$ of  filtered complexes is pre-Abelian  
but not Abelian  \cite{HA},  but the quotient functor ${\mathcal F} \to {\coimage P_n}$ does not preserve kernels and cokernels.)

%\vfill\eject

\subsection{Zigzag Persistent Homology}
\label{subsection5.3}

Zigzag persistent homology was introduced in \cite{CdS,CdSM} and further studied in \cite{SV-J,CO,CO2}.
In this section we  apply the  categorical techniques of Chapters \ref{section2} and \ref{section3} to 
the general zigzag case.   A reader who is not interested in the zigzag case may skip this section, and 
continue to the concluding discussion in Section 5.4.  
 Our main result Theorem \ref{alternate} applies to ``ordinary persistent homology" ({\it i.e.} not the zigzag generalization).  
This result is also  informally outlined in the flowchart diagrams of Section \ref{subsection1.1}.   
We  will  show below that this result only partially generalizes to the zigzag case.     
Theorem \ref{alternate} asserts that the  ordinary persistent homology functor factors as 
\begin{equation*} {\mathcal F} \to {\coimage P_n}\to {\image P_n} \to {\mathcal P}, 
\end{equation*}
where each category is Krull-Schmidt.   This assertion generalizes to the zigzag case. 
Theorem \ref{alternate} further asserts that the functor ${\coimage P_n}\to {\image P_n}$ is an equivalence 
of categories.  This assertion does {\it not} generalize to the zigzag case.  
For the general zigzag case, 
the indecomposables of the Krull-Schmidt category ${\image P_n}$ are still classified by 
intervals  as in \cite{CdS}.  But now the classification of 
indecomposables is more complicated for the Krull-Schmidt category ${\coimage P_n}$,  
as illustrated by the example at the end of the section.

We proceed to an  outline of the categorical framework for the general  zigzag case.     
Let ${\mathcal X}$ be a linear Abelian category.   
We describe the sources of the leftward-directed arrows of a zigzag diagram as a  subset $L \subseteq {\mathbb Z}$.  For 
any $L \subseteq {\mathbb Z}$, we define an  
 {\it $L$-persistence object} in ${\mathcal X}$ to be a diagram ${\prescript{}{\bullet} X}$ in the category ${\mathcal X}$ of  type 
\begin{equation*}
\begin{tikzcd}[row sep=tiny, column sep=normal] 
      \cdots  \arrow[r,leftarrow,shift left=0.5ex]  \arrow[r,rightarrow,shift right=0.5ex] & 
      \prescript{}{p-1} X \arrow[r,leftarrow,shift left=0.5ex]  \arrow[r,rightarrow,shift right=0.5ex] & 
       \prescript{}{p} X     \arrow[r,leftarrow,shift left=0.5ex]  \arrow[r,rightarrow,shift right=0.5ex] & 
       \prescript{}{p+1} X \arrow[r,leftarrow,shift left=0.5ex]  \arrow[r,rightarrow,shift right=0.5ex] & 
       \cdots,    \\  
     \end{tikzcd}
\end{equation*}
where the arrow directions are specified by the rule:  $\prescript{}{p-1} X   \longleftarrow \prescript{}{p} X$ if $p \in L$ and 
$\prescript{}{p-1} X   \longrightarrow \prescript{}{p} X$ if $p \notin L$.   
For example,  an $L$-persistence object with  $L = \{ 0, 2, 3 \} \subseteq {\mathbb Z}$ is: 
\begin{equation*}%\begin{tightcenter} 
\begin{tikzcd}[row sep=tiny, column sep=small] 
      \cdots  \arrow[r,rightarrow]   &  \prescript{}{-1} X \arrow[r,leftarrow]  & \prescript{}{0} X
      \arrow[r,rightarrow] &  \prescript{}{1} X  \arrow[r,leftarrow] &  \prescript{}{2} X  \arrow[r,leftarrow] & 
       \prescript{}{3} X \arrow[r,rightarrow] & \cdots  . \\  
       \phantom{p=3}    &  p =-1 & p = 0 & p = 1 &  p = 2 & p = 3\phantom{X}  & \\
     \end{tikzcd}
\end{equation*}%\begin{tightcenter} 
We recover the previous definition of  
 ``ordinary" persistence object from Section  \ref{subsection2.2} by choosing $L = \emptyset$.    
A morphism of $L$-persistence objects  $\prescript{}{\bullet} f: \prescript{}{\bullet} X \to \prescript{}{\bullet} X'$ is a 
commutative diagram of ``ladder"  type:  
\begin{equation*}%\begin{tightcenter} 
\begin{tikzcd}[row sep=normal, column sep=normal] 
\cdots  
\arrow[r,leftarrow,shift left=0.5ex]  \arrow[r,rightarrow,shift right=0.5ex]  & \prescript{}{p-1} X \ar{d} { \prescript{}{p-1} f}  
\arrow[r,leftarrow,shift left=0.5ex]  \arrow[r,rightarrow,shift right=0.5ex] &  \prescript{}{p} X \ar{d}[left]{\prescript{}{p} f}
  \arrow[r,leftarrow,shift left=0.5ex]  \arrow[r,rightarrow,shift right=0.5ex] 
       & \prescript{}{p+1} X  \ar{d}[left]{\prescript{}{p+1} f}   
      \arrow[r,leftarrow,shift left=0.5ex]  \arrow[r,rightarrow,shift right=0.5ex]  & \cdots \\
   \cdots  
  \arrow[r,leftarrow,shift left=0.5ex]  \arrow[r,rightarrow,shift right=0.5ex]  &\prescript{}{p-1} X'  
 \arrow[r,leftarrow,shift left=0.5ex]  \arrow[r,rightarrow,shift right=0.5ex]    &  \prescript{}{p} X'   
      \arrow[r,leftarrow,shift left=0.5ex]  \arrow[r,rightarrow,shift right=0.5ex]   & {\prescript{}{p+1} X'}  
      \arrow[r,leftarrow,shift left=0.5ex]  \arrow[r,rightarrow,shift right=0.5ex]  & \cdots .\\
     \end{tikzcd}
\end{equation*}%\begin{tightcenter} 
Recall that a categorical diagram is {\it tempered} if all but finitely many of its arrows are  iso(morphisms).  
The proof of Proposition \ref{persKS} readily generalizes to:  

\begin{proposition}
\label{zzKS}  Let ${\mathcal X}$ be a linear Abelian category.  Then for any $L \subseteq {\mathbb Z}$, the 
  category of tempered $L$-persistence objects in ${\mathcal X}$  is Krull-Schmidt.  
\end{proposition} 

An ``$L$-filtered object" is the generalization of a filtered object to the zigzag case. 
For any  $L \subseteq {\mathbb Z}$,  
an $L$-filtered object in a concrete linear Abelian category ${\mathcal X}$ is a special type of tempered $L$-persistence object in ${\mathcal X}$.  
Recall that in a concrete linear Abelian category  we denote by  $X \hookrightarrow X'$ the inclusion of a subobject. 
We say a tempered $L$-persistence object $\prescript{}{\bullet} X$  is an   
$L$-{\it filtered object} if it is bounded below and if   every arrow is an inclusion arrow: 
\begin{equation*}
\begin{tikzcd}[row sep=tiny, column sep=normal] 
      \cdots  \arrow[r,hookleftarrow,shift left=0.65ex]  \arrow[r,hookrightarrow,shift right=0.65ex] & 
      \prescript{}{p-1} X \arrow[r,hookleftarrow,shift left=0.65ex]  \arrow[r,hookrightarrow,shift right=0.65ex] & 
       \prescript{}{p} X    \arrow[r,hookleftarrow,shift left=0.65ex]  \arrow[r,hookrightarrow,shift right=0.65ex]  & 
       \prescript{}{p+1} X \arrow[r,hookleftarrow,shift left=0.65ex]  \arrow[r,hookrightarrow,shift right=0.65ex]  & 
       \cdots.    \\  
     \end{tikzcd}
\end{equation*}
For example,  an $L$-filtered object with  $L = \{ 0, 2, 3 \} \subseteq {\mathbb Z}$ is: 
\begin{equation*}%\begin{tightcenter} 
\begin{tikzcd}[row sep=tiny, column sep=small] 
      \cdots  \arrow[r,hookrightarrow]   &  \prescript{}{-1} X \arrow[r,hookleftarrow]  & \prescript{}{0} X
      \arrow[r,hookrightarrow] &  \prescript{}{1} X  \arrow[r,hookleftarrow] &  \prescript{}{2} X  \arrow[r,hookleftarrow] & 
       \prescript{}{3} X \arrow[r,hookrightarrow] & \cdots  . \\  
       \phantom{p=3}    &  p =-1 & p = 0 & p = 1 &  p = 2 & p = 3\phantom{X}  & \\
     \end{tikzcd}
\end{equation*}%\begin{tightcenter} 
We recover the previous definition of  
``ordinary" filtered object from Section  \ref{subsection2.2} by choosing $L = \emptyset$.    
 It is important to note that an $L$-filtered object need not admit a categorical limit and colimit  
in the general  case $L \ne \emptyset$.   Consequently only the first assertion of Lemma \ref{crass} generalizes: 
 \begin{lemma}
\label{lcrass}
 Let ${\mathcal X}$ be a linear Abelian category.   Then for any $L \subseteq {\mathbb Z}$, the category of 
 $L$-filtered objects in ${\mathcal X}$ is Krull-Schmidt.  
\end{lemma} 
But the classification of filtered objects in terms of colimits is not available for the general zigzag case $L \ne \emptyset$.  

For any $L \subseteq {\mathbb Z}$ we can now introduce $L$-persistent homology,  commonly known as ``zigzag persistent homology."  
 ${\mathcal F}$ now denotes the Krull-Schmidt (by Lemma \ref{lcrass}) category of $L$-filtered complexes, where an object is a diagram of complexes.    
${\mathcal P}$ now denotes the Krull-Schmidt (by Proposition \ref{zzKS}) category of $L$-persistence vector spaces, where an object is a diagram of vector spaces. 
The homology functor $H_n$ takes a diagram of complexes to a diagram of vector spaces, resulting 
in a functor $ P_n:{\mathcal F} \to {\mathcal P}$ which we  call 
 the {\it $L$-persistent homology} of degree $n$.  This is the  functor that is  commonly known as {\it zigzag persistent homology}.    
We recover the ``ordinary" ({\it i.e.} not ``zigzag") persistent homology functor  by choosing $L = \emptyset$.  

We now consider the factorization properties of the $L$-persistent homology functors for an arbitrary subset $L \subseteq {\mathbb Z}$, obtaining 
a partial generalization of Theorem \ref{alternate} for the  ordinary   case where $L = \emptyset$.  
The category $\image  P_n$, comprised of the  
$L$-persistence vector spaces that are bounded below, is a full Abelian subcategory of ${\mathcal P}$  and 
therefore Krull-Schmidt.   
(The category $\image P_n$ does not depend on $n$;  it is always the same subcategory of $ {\mathcal P}$.)  
The well-studied representation theory of  $A_n$ quivers (see e.g. \cite{CdS,Shiffler}) 
still carries over by a  limiting argument to classify the indecomposable $L$-persistence vector spaces in terms of 
 intervals    $I \subseteq {\mathbb Z}$ that are bounded below.   
The category $\coimage P_n$ is   
a categorical quotient of $ {\mathcal F}$, defined  via the following equivalence relation (congruence) on morphisms:  
two morphisms $f$ and $f'$ in ${\mathcal F}$ are equivalent iff the morphisms $P_n(f)$ and $P_n(f')$ in 
${\mathcal P}$ are equal.   For any $L \subseteq {\mathbb Z}$,  the category 
 ${\mathcal F}$ and its quotient $\coimage P_n$ are both Krull-Schmidt.  
But the classification of their indecomposable summands in terms of colimits is not 
available for the general zigzag case $L \ne \emptyset$.  
Consequently for arbitrary $L \subseteq {\mathbb Z}$, we only have the 
following partial generalization of Theorem \ref{alternate}:  
 \begin{theorem}
 \label{lalternate}  For any $L \subseteq {\mathbb Z}$, 
 the $L$-persistent homology functor $P_n:{\mathcal F} \to{\mathcal P}$ factors as 
 \begin{equation*} {\mathcal F} \to {\coimage P_n}\to {\image P_n} \to{\mathcal P}.  
\end{equation*}
 \end{theorem}  
\noindent  Recall that for the  ``ordinary" case $L = \emptyset$, we compared  the classification of indecomposables in 
the two categories to  prove that  ${\coimage P_n}\to {\image P_n}$ is 
an equivalence of categories.    
But for the general zigzag case where $L \ne \emptyset$, the following example shows that 
the functor ${\coimage P_n}\to {\image P_n}$ is {\it not}  an equivalence in general: 

\begin{example} 
\label{lexy} We consider $L$-persistent homology with $L =\{0,2,3 \}  \subseteq  {\mathbb Z}$.   
Start with the  $L$-filtered simplicial complex  shown in Figure \ref{fig:exp1}.  
%
% figure simplicial1.pdf
\begin{figure*}
  \includegraphics[width=\textwidth]{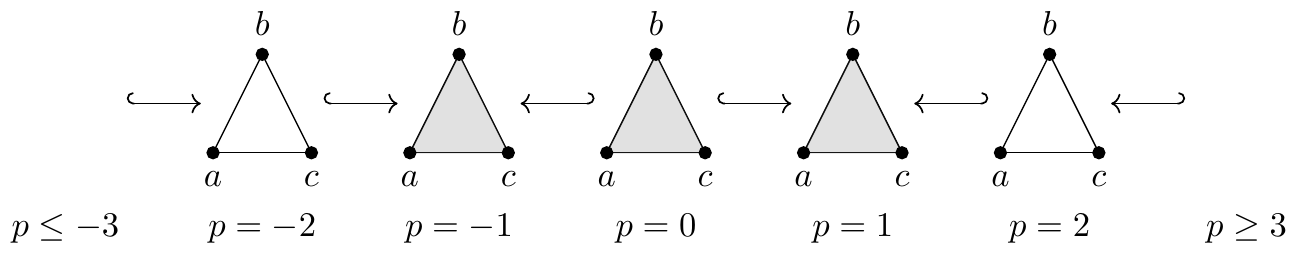}
\caption{An $L$-filtered Simplicial Complex with $L =\{0,2,3 \}  \subseteq  {\mathbb Z}$. }
\label{fig:exp1}       % Give a unique label
\end{figure*}
In the quotient category  ${\coimage P_1}$ this becomes the (indecomposable) object 
\begin{equation*}%\begin{tightcenter} 
\begin{tikzcd}[row sep=tiny, column sep=small] 
      0  \arrow[r,hookrightarrow]   & J[1] \arrow[r,hookrightarrow]  &K[1]
      \arrow[r,hookleftarrow] &  K[1]  \arrow[r,hookrightarrow] & K[1]  \arrow[r,hookleftarrow] & 
       J[1] \arrow[r,hookleftarrow] & 0  . \\  
         p\le -3  &  p =-2 & p = -1 & p = 0 &  p = 1& p = 2  & p\ge 3 \\
     \end{tikzcd}
\end{equation*}%\begin{tightcenter} 
Compare with the $L$-filtered simplicial complex shown in Figure \ref{fig:exp2}. 
%
% figure simplicial1.pdf
\begin{figure*}
  \includegraphics[width=\textwidth]{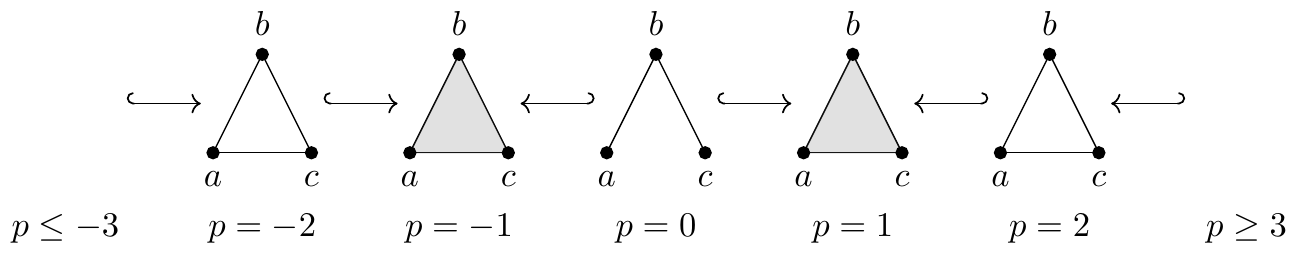}
\caption{Another $L$-filtered Simplicial Complex  with $L = \{ 0,2,3 \} \subseteq  {\mathbb Z}$. }
\label{fig:exp2}       % Give a unique label
\end{figure*}
In the quotient category  ${\coimage P_1}$ this becomes  the (decomposable) object
\begin{equation*}%\begin{tightcenter} 
\begin{tikzcd}[row sep=tiny, column sep=small] 
      0  \arrow[r,hookrightarrow]   & J[1] \arrow[r,hookrightarrow]  &K[1]
      \arrow[r,hookleftarrow] &  0 \arrow[r,hookrightarrow] & K[1]  \arrow[r,hookleftarrow] & 
       J[1] \arrow[r,hookleftarrow] & 0  . \\  
         p\le -3  &  p =-2 & p = -1 & p = 0 &  p = 1& p = 2  & p\ge 3 \\
     \end{tikzcd}
\end{equation*}%\begin{tightcenter} 
This pair of objects is not  isomorphic in the quotient category  ${\coimage P_1}$.  
But these non-isomorphic objects become isomorphic in ${\image P_n}$,  since 
both go to the same (decomposable) object 
\begin{equation*}%\begin{tightcenter} 
\begin{tikzcd}[row sep=tiny, column sep=small] 
      0  \arrow[r,rightarrow]   &  {\mathbb F} \arrow[r,rightarrow]  & 0
      \arrow[r,leftarrow] &  0  \arrow[r,rightarrow] & 0 \arrow[r,leftarrow] & 
       {\mathbb F}  \arrow[r,leftarrow] & 0  . \\  
         p\le -3  &  p =-2 & p = -1 & p = 0 &  p = 1& p = 2  & p\ge 3 \\
     \end{tikzcd}
\end{equation*}%\begin{tightcenter} 
It follows that the functor  ${\coimage P_n}\to {\image P_n}$ cannot be an 
equivalence of categories; 
an equivalence would not take a non-isomorphic pair to an isomorphic pair.    (Furthermore, an equivalence 
 would not take an indecomposable object to a decomposable object.)  
\end{example}

%\vfill\eject

\subsection{What is the Best  Framework for Persistent Homology?}
\label{subsection5.4}

To conclude,  we make a few general remarks about comparing  categorical frameworks.    
A functor $F: {\mathcal A} \to {\mathcal B}$ transforms objects and morphisms in a category ${\mathcal A}$ 
to objects and morphisms  in another category ${\mathcal B}$.   
The usefulness of a functor for studying objects and morphisms in the category ${\mathcal A}$ 
depends on various criteria for the category ${\mathcal B}$.   
 Such criteria are discussed in many Algebraic Topology textbooks, for example  
\cite{Hatcher}, in the context of the ``plain" homology functors $H_n$.   Here we briefly consider some key  
criteria in the context of the persistent homology functors $P_n$, including the zigzag case of 
 Appendix \ref{subsection5.3}.

A functor $F: {\mathcal A} \to {\mathcal B}$ may be useful if the category ${\mathcal B}$ has additional structure.   
We will call this the {\it structural criterion} for the category ${\mathcal B}$.  
 Krull-Schmidt categories and Abelian categories are relevant examples of categories with additional structure.  
A categorical structure tends to be useful in applications if it is amenable to algorithmic computation.  
The most important applications of persistent homology are based on algorithmic decompositions of objects in various Krull-Schmidt categories.  
As we have shown,   the algorithms for  persistent homology  actually compute a  Krull-Schmidt decomposition of a filtered complex in ${\mathcal F}$.   
In the standard framework (Section \ref{subsection3.1}) the functor ${\mathcal F} \to {\image P_n}$ takes this decomposition 
of a filtered complex in ${\mathcal F}$ 
to a decomposition of a persistence vector space in the category ${\image P_n}$ of persistence vector spaces.  
In our alternate framework (Section \ref{subsection3.2}) the functor ${\mathcal F} \to {\coimage P_n}$ takes this decomposition 
of a filtered complex in ${\mathcal F}$ to a decomposition of 
a filtered complex in the quotient category ${\coimage P_n}$.  
For ``ordinary" persistent homology we have the equivalence of categories ${\coimage P_n} \to {\image P_n}$ (Theorem \ref{alternate}),  
so  the alternate framework and the standard framework perform equally well on the structural criterion.  

But in the general zigzag case the categories ${\coimage P_n}$ and ${\image P_n}$ are {\it not} 
equivalent (Appendix \ref{subsection5.3}).   
Decomposition algorithms in the Krull-Schmidt category ${\image P_n}$ are known, and 
furthermore this category is known to be Abelian.   So the standard 
framework ${\mathcal F} \to {\image P_n}$ performs well on the structural criterion.  
But decomposition algorithms in the more complicated Krull-Schmidt category ${\coimage P_n}$ do not appear to be known at present, and furthermore it is not clear whether this category is Abelian.   So at present 
a putative alternate framework  ${\mathcal F} \to {\coimage P_n}$ for the general zigzag case does rather badly on the 
structural criterion.

An important countervailing consideration is whether the functor $F: {\mathcal A} \to {\mathcal B}$ loses, or forgets, too much 
of the information present in the original category ${\mathcal A}$.    
For example, the functor may be losing too much information if it takes
 non-isomorphic objects in ${\mathcal A}$ to isomorphic objects in ${\mathcal B}$ (as we saw in Example \ref{lexy} for 
a zigzag case of  the functor ${\coimage P_n}\to {\image P_n}$). 
 Typically we work with concrete categories, where 
objects are represented as sets with additional features, such as algebraic or topological features.     
Then we would like  the sets representing objects of ${\mathcal B}$ 
to retain features of the sets representing objects of ${\mathcal A}$.  
We will call this the {\it representational criterion} for the category ${\mathcal B}$.  
Our alternate framework for persistent homology (Section \ref{subsection3.2}) is based on the 
quotient functor  ${\mathcal F} \to {\coimage P_n}$.    Our  alternate framework performs very well on 
the representational criterion, because a filtered complex in ${\mathcal F}$ goes to the very same  
filtered complex in the quotient category ${\coimage P_n}$.   
The standard framework (Section \ref{subsection3.1})  is based on the 
 functor  ${\mathcal F} \to {\image P_n}$.    The standard framework does not perform  well 
 on the representational criterion, because a persistence vector space in ${\image P_n}$ does not 
 retain  the algebraic features of a filtered complex in ${\mathcal F}$.  
 These arguments carry over to the general zigzag case, where a putative alternate 
 framework would also perform better on the representational criterion.

%In conclusion, we argue that our alternate framework for persistent homology is better than the standard framework. 
In conclusion, we argue that our alternate framework for persistent homology is better than the standard framework based on the representational criterion and structural criterion described here, in the setting of ``ordinary" persistent homology.
%This is because the alternate framework performs much better on the representational criterion, and performs 
%equally well on the structural criterion.     This argument is limited to ``ordinary" persistent homology, and does 
%{\it not} extend to the general zigzag case,  
%because the bad performance on the structural criterion renders other considerations moot.  

% ****** end revision

%\appendix
\appendix\normalsize

\section{Bruhat Uniqueness Lemma}
\label{appendixA}

Here we establish the uniqueness of the persistence canonical form $\underline{D}$  appearing  in 
the Matrix Structural Theorem \ref{postnon}, as well as in the ungraded version Theorem \ref{prenon}.  
Our result generalizes the  uniqueness statement for  the usual Bruhat factorization of 
 an invertible matrix \cite{AB,Geck}.   
 
It is convenient to make the following definitions.  
 We call an (upper) triangular matrix  $U$ \emph{unitriangular} if it  is unipotent, meaning that  each diagonal entry is $1$. 
We call a matrix $M$ \emph{quasi-monomial} if each row has at most one nonzero entry and each column has at most one nonzero entry. 
We remark that a unitriangular matrix is always square, but a quasi-monomial matrix need not be square.   
The key to proving uniqueness  is:     

\begin{lemma} 
\label{firstA}
Suppose  $M_1\,U = V \,M_2  $, where  $M_1$ and $M_2$ are quasi-monomial and $U$ and $V$ are unitriangular.
Then $M_2=M_1$.
\end{lemma}

\noindent In the following proof, the term  \emph{row-pivot}  denotes a matrix entry that 
is the leftmost nonzero entry in its row, and  \emph{column-pivot} denotes a matrix 
entry that is the bottommost nonzero entry in its column.

\begin{proof}
The first half of the proof consists of showing that every nonzero entry of $M_2$ is also an entry of $M_1$.      
A nonzero entry of the quasi-monomial matrix $M_2$ is a column-pivot. 
Similarly a nonzero entry of the quasi-monomial matrix $M_1$ is a row-pivot. 
It now suffices to show that a column-pivot of $M_2$
is a row-pivot of $M_1$.
Since $V$ is unitriangular, $VM_2$ has the same column-pivots as $M_2$.  
Similarly  since $U$ is unitriangular,  $M_1U$ has the same row-pivots as $M_1$. 
It now suffices to prove that a column-pivot of $S = VM_2$ is a  row-pivot of $S =M_1U$.  
Suppose to the contrary that some column-pivot of $S$  is not a row-pivot of $S$.   
Let $x$ be the leftmost such column-pivot.     Since $x$ is not a row-pivot, there exists a row-pivot $y$ to the left of $x$ in the same row. 
If $y$ were a column-pivot of $S = VM_2$,  then it would be a column-pivot of $M_2$.
But the quasi-monomial matrix $M_2$ cannot have two nonzero entries $y$ and $x$ in the same row.
So $y$ is not a column-pivot of $S$, and there exists a column-pivot $z$ below $y$ in the same column. 
If $z$ were a row-pivot of $S = UM_1$, then it would be a row-pivot of $M_1$. 
But the quasi-monomial matrix $M_1$ cannot have two nonzero entries $z$ and $y$ in the same column. 
So $z$ is a column-pivot of $S$ that  is not a row-pivot of $S$, and $z$ is to the left of (and below) $x$.
This is a contradiction, because $x$ is the leftmost such column-pivot.   

The second half of the proof consists of  showing that every nonzero entry of $M_1$ is also an entry of $M_2$.  
This is analogous to the first half, and we omit the details.   
The two matrices then have the same nonzero entries, so they must also have the same zero entries.
Since all the entries of the two matrices are the same, we have proved  $M_2=M_1$. 
\qed \end{proof} 

Recall that a matrix $M$ is \emph{Boolean} if every non-zero entry is $1$.    
An almost-Jordan differential matrix ${\underline D}$ is  Boolean and quasi-monomial.    

\begin{proposition}
\label{cobb}
Suppose $P_1A=BP_2$ where $P_1$ and $P_2$ are Boolean quasi-monomial and $A$ and $B$ are invertible triangular. Then $P_2=P_1$.
\end{proposition}

\begin{proof}  
Factor $A = T_1U$ as the product of an invertible diagonal  matrix $T_1$ and a unitriangular matrix $U$.     
Factor $B = VT_2$ as the product of a unitriangular matrix  $V$ and an invertible diagonal matrix $T_2$.  
Then $(P_1 T_1)U = V(T_2 P_2)$,  with $(P_1 T_1)$ and $(T_2 P_2)$  quasi-monomial.     
 Lemma \ref{firstA} then gives  the $P_1 T_1 =T_2 P_2$.   
Since the quasi-monomial matrices 
$P_1$ and $P_2$ are Boolean, the conclusion follows.     
\qed \end{proof}

\noindent We remark that any permutation matrix $P$ is   Boolean and quasi-monomial, so     
Proposition \ref{cobb}  generalizes the standard uniqueness result for Bruhat factorization of an 
invertible matrix \cite{AB,Geck}. 

The uniqueness of the persistence canonical form $\underline D$ appearing in 
Theorem \ref{prenon} and in the Matrix Structural Theorem \ref{postnon} now follows easily: 

\begin{corollary}
\label{eunuch}
Suppose $D$ is a differential matrix and $B_1$ and $B_2$ are invertible triangular matrices.  
If both differential matrices ${\underline D}_1 = B_1^{-1} D B_1$ and ${\underline D_2} = B_2^{-1} D B_2$ are almost-Jordan, then    
  ${\underline D}_2 ={\underline D}_1$.  
\end{corollary} 

\begin{proof}
 ${\underline D_1} (B_1^{-1} B_2) = (B_1^{-1} B_2) {\underline D_2}$,  and 
 the result follows from Proposition \ref{cobb}. 
\qed \end{proof}

\section{Constructively Proving  the Matrix Structural Theorem}
\label{appendixB}

\subsection{Linear Algebra of Reduction} 
\label{apendixB.1}
In this section we discuss  {column-reduction} of a matrix $M:{\mathbb F}^m \to {\mathbb F}^n$, including 
its application to describing 
the kernel and image of the matrix.   
Column-reduction of a differential matrix $D$ is a standard tool in the computation of persistent homology, 
where it is usually just called {\it reduction} \cite{ELZ,ZC1,ZC2,CSEM}.  
We prefer the more precise terminology in order to maintain the distinction with   row-reduction, since 
both are used for Bruhat factorization \cite{AB,Geck,Lusztig}.

As in Appendix A, the term  \emph{column-pivot} denotes a matrix 
entry that is the bottommost nonzero entry in its column. 
A matrix $R$ is said to be \emph{column-reduced}  if each 
row has at most one column-pivot.  

\begin{definition} 
\label{redu}
A column-reduction of a matrix $M$ is an invertible triangular matrix $V$ such that 
$R = MV$ is column-reduced.  
\end{definition} 

\noindent  A column-reduction $V$ exists for any matrix $M$, but is not unique in general.  
Column-reduction algorithms used for persistent homology \cite{ELZ,ZC1,ZC2,DMV} 
usually prioritize computational efficiency.   
For our computational examples, we will use a column-reduction algorithm that is popular for Bruhat factorization \cite{AB,Geck}. 
This algorithm is easy to implement,  but is not very efficient computationally.    
The algorithm starts at the leftmost column of $M$ and proceeds rightward 
by successive columns as follows: 
\begin{itemize}[leftmargin=*]
\item	If the current column is zero, do nothing.
\item	If the current column is nonzero, add an appropriate multiple of the current column to each column  
to the right in order to zero the entries to the right of the column-pivot (in the same row).  
\end{itemize}
\noindent Stop if the current column is the rightmost column, otherwise proceed to the column immediately to the right and repeat.  
By design, the resulting matrix $R$ has the property that any column-pivot has only zeros to the right of it (in the same row). 
So a row of $R$ cannot contain more than one column-pivot, implying that $R$ is column-reduced.  
The invertible triangular column-reduction matrix $V$ is constructed by performing the same column operations on 
the identity matrix $I$, where $I$ has  same number of columns as $M$.   

We  briefly discuss some linear-algebraic properties of column-reduction. 
A column-reduction easily 
yields a basis for the kernel of a matrix as well as a basis for the image. 
By contrast, Gaussian elimination easily yields a basis for the image a matrix, but 
requires  additional  back-substitution to produce a basis for the kernel.    
Column-reduction algorithms are  therefore a  convenient alternative to Gaussian elimination for matrix computations in general, and 
this fact seems to be underappreciated.   
We  use a variant of the usual adapted basis for a filtered vector space, disregarding the ordering of basis elements.   
We'll say that a basis of a finite-dimensional vector space $X$ is {\it almost-adapted} to a subspace $Y \subseteq X$ if 
$Y$ is spanned by the set of basis elements that are contained in $Y$.   Proposition \ref{redu} yields:  

\begin{corollary} 
\label{redpro}
 Let $V:{\mathbb F}^m \to {\mathbb F}^m$ be a column-reduction of a 
 matrix  $M:{\mathbb F}^m \to {\mathbb F}^n$.  
Then: 
\begin{enumerate}[leftmargin=*]
\item The nonzero columns of the column-reduced matrix $R = MV$ are a basis of $\image M$.  
\item The columns of the invertible triangular matrix $V$ are a basis of ${\mathbb F}^m$, and 
this basis  is almost-adapted to $\ker M$.  
\end{enumerate}
\end{corollary} 

\begin{proof} \hfil
\begin{enumerate}[leftmargin=*]
\item The nonzero columns of $R$ span $\image M$. 
The  nonzero columns  of $R$ are 
 linearly independent because  $R$ is column-reduced.  
\item  The columns of $V$ are  a basis of ${\mathbb F}^m$ because $V$ is invertible.  
This basis is almost-adapted to $\ker M$ because the nonzero columns of $R = MV$ are linearly independent.   
\qed
\end{enumerate}
 \end{proof}

\begin{example} We  compute in detail a column-reduction of the matrix $M: {\mathbb Q}^4 \to {\mathbb Q}^3$, 
which is presented below with a column augmentation by the identity matrix $I$.

\begin{equation*}
{M \over I} = 
%*** begin augmented matrix
\begin{tikzpicture}[baseline={([yshift=0.8ex]current bounding box.center)}]

[decoration=brace]
\tikzset{
    node style ge/.style={circle,minimum size=.75cm},
}
\pgfdeclarelayer{background}
\pgfdeclarelayer{foreground}
\pgfsetlayers{background,main,foreground}
\matrix (A) [matrix of math nodes,
             nodes = {node style ge},
             left delimiter  = {[}, %brackets
             right delimiter = {]}, %brackets
              inner sep=-2pt,
		row sep=-.3cm,
		column sep=-.25cm
             ]
{
%MATRIX ENTRIES
 1 & -2 & 0 & -8 \\
 2 & -4 & 6 & 2 \\
 1 & -2 & 2 & -2 \\
 1 & 0 & 0 & 0 \\
 0 & 1 & 0 & 0 \\
 0 & 0 & 1 & 0 \\
 0 & 0 & 0 & 1 \\
};

\node (leftmid) at ($(A-3-1.south west)!0.5!(A-4-1.north west)-(.15,0)$) {};
\node (rightmid) at ($(A-3-4.south east)!0.5!(A-4-4.north east)+(.15,0)$) {};

\draw (leftmid) -- (rightmid);

\end{tikzpicture}
% end augemented matrix
\mapsto
%*** begin augmented matrix
\begin{tikzpicture}[baseline={([yshift=0.8ex]current bounding box.center)}]

[decoration=brace]
\tikzset{
    node style ge/.style={circle,minimum size=.75cm},
}
\pgfdeclarelayer{background}
\pgfdeclarelayer{foreground}
\pgfsetlayers{background,main,foreground}
\matrix (A) [matrix of math nodes,
             nodes = {node style ge},
             left delimiter  = {[}, %brackets
             right delimiter = {]}, %brackets
	 inner sep=-2pt,
		row sep=-.3cm,
		column sep=-.25cm
             ]
{
%MATRIX ENTRIES
 1 & 0 & -2 & -6 \\
 2 & 0 & 2 & 6 \\
{\mathbf 1} & 0 & 0 & 0 \\
 1 & 2 & -2 & 2 \\
 0 & 1 & 0 & 0 \\
 0 & 0 & 1 & 0 \\
 0 & 0 & 0 & 1 \\
};

\node (leftmid) at ($(A-3-1.south west)!0.5!(A-4-1.north west)-(.15,0)$) {};
\node (rightmid) at ($(A-3-4.south east)!0.5!(A-4-4.north east)+(.15,0)$) {};

\draw (leftmid) -- (rightmid);

\end{tikzpicture}
\mapsto
%*** begin augmented matrix
\begin{tikzpicture}[baseline={([yshift=0.8ex]current bounding box.center)}]

[decoration=brace]
\tikzset{
    node style ge/.style={circle,minimum size=.75cm},
}
\pgfdeclarelayer{background}
\pgfdeclarelayer{foreground}
\pgfsetlayers{background,main,foreground}
\matrix (A) [matrix of math nodes,
             nodes = {node style ge},
             left delimiter  = {[}, %brackets
             right delimiter = {]}, %brackets
            inner sep=-2pt,
		row sep=-.3cm,
		column sep=-.25cm
             ]
{
%MATRIX ENTRIES
  1 & 0 & -2 & 0 \\
 2 & 0 & { \mathbf 2} & 0 \\
{\mathbf 1} & 0 & 0 & 0 \\
 1 & 2 & -2 & 8 \\
 0 & 1 & 0 & 0 \\
 0 & 0 & 1 & -3 \\
 0 & 0 & 0 & 1 \\
};

\node (leftmid) at ($(A-3-1.south west)!0.5!(A-4-1.north west)-(.15,0)$) {};
\node (rightmid) at ($(A-3-4.south east)!0.5!(A-4-4.north east)+(.15,0)$) {};

\draw (leftmid) -- (rightmid);

\end{tikzpicture}
= {R \over V}
\end{equation*} 

\noindent  The result of the computation is a factorization $R = MV$, where 
$R$ is column-reduced and $V$ is invertible triangular (and unipotent).  We describe each 
step of the computation:    
\begin{enumerate}[leftmargin=*]
\item 
The first column of $M$ is nonzero, so it has a column-pivot. 
At the next processing step, boldface the column-pivot for clarity, and 
add an appropriate multiple of the first column to each column  
to the right in order to zero the entries to the right of the column-pivot (in the same row).   
\item At this point the second column is  zero, so requires no processing step.
\item At this point the third column is  nonzero, so it has a column-pivot.  
At the next processing step, 
 boldface the column pivot, and 
add an appropriate multiple of the third column to each column  
to the right in order to zero the entries to the right of the column-pivot (in the same row). 
\item At this point the fourth column is  zero, so requires no processing step.
\end{enumerate}

\noindent  Columns 1 and 3 of $R$ are the nonzero columns, so they are a basis of $\image M \subseteq {\mathbb Q}^3$.
The four columns of $V$ are a basis of ${\mathbb Q}^4$ that is amost-adapted to $\ker M$.  
 Columns 2 and 4 of $V$ correspond to the zero columns of $R$, so they  are a basis of $\ker M\subseteq {\mathbb Q}^4$.    
\end{example}

\subsection{Matrix Structural Theorem via Reduction}

\label{appendixB.2}
The standard algorithm of persistent homology  \cite{ELZ,ZC1,ZC2} starts with a 
differential matrix $D$ and 
constructs  a matrix $B$ satisfying the conditions of:  

 % manually number repeat proposition
 \setcounter{repsection}{1} 
\setcounter{repeat}{1}
\begin{reptheorem}
 (Ungraded Matrix  Structural Theorem) 
 Any differential matrix $D$ factors as $D = B {\underline D} B^{-1}$ where  
${\underline D}$ is an almost-Jordan differential matrix  and $B$ is a triangular matrix.   
\end{reptheorem}

\noindent    
The matrix formulation of the  {\it standard algorithm} constructs a matrix $B = \hat V$ from a 
column-reduction $V$ of a differential matrix $D$, as discussed  in \cite{DMV,BCY} for 
${\mathbb F} = {\mathbb Z}/2 {\mathbb Z}$. 
Since $R = DV$ is column-reduced,  there exists at most one nonzero column of $R$ that has its column-pivot in row $k$.     
Here $1 \le k \le m$ where $m$ is the number of rows of the square matrix $D$.   
$\hat V$ is constucted one column at a time using the following rule:   

\begin{itemize}[leftmargin=*]
\item If there exists a nonzero column of $R$ that has its column-pivot in row $k$, then column $k$ of $\hat V$ is 
 equal to this column of $R$.   
\item   If there does not exist a nonzero column of $R$ that has its column-pivot in row $k$, then column 
$k$ of $\hat V$ is  equal to column $k$ of $V$.  
\end{itemize} 

\noindent  The matrix $\hat V$ is invertible triangular, because each column is nonzero and has its column-pivot on the diagonal.  

Unlike its progenitor $V$, the matrix $\hat V$ contains all of the nonzero columns of $R$.  
We introduce the ``pivot matrix" of $R$ to encode the combinatorial data needed 
to recover  $R$.    
For any column-reduced matrix $M$, the {\it pivot matrix} of $M$ is constructed by  
replacing every column-pivot of $M$ with $1$, and every other nonzero entry of $M$ with $0$.  
It follows that the pivot matrix is is Boolean and quasi-monomial (Appendix \ref{appendixA}).

\begin{lemma}
\label{redif2} Let  $\underline D$ be the pivot matrix of  $R$. 
Then  $\hat V \underline D = R$. 
\end{lemma}

\begin{proof}  
Supposing column $k$ of $R$ is nonzero,  
let $j$ be the row number of the unique nonzero entry in column $k$ of $\underline D$.  
Then by construction, column $j$ of $\hat V$ is equal to column $k$ of $R$.  
\qed \end{proof}

But like its progenitor $V$, the matrix $\hat V$ is a column-reduction of the differential $D$:

\begin{lemma}  
\label{redifm}
$D \hat V = R$.   
\end{lemma}  

\begin{proof}  The triangular matrix $V$ is a column-reduction of $D$, as per Definition \ref{redu}. 
We now show that the triangular matrix $\hat V$ is also a column-reduction of $D$.  
Recall that $\hat V$ is  constructed as a modification of $V$, 
by replacing a  (possibly empty) subset of the columns of $V$ with columns of $R$.  
Every column of $R$ is in $\ker D$, because $D R = D^2 V = 0$.  
So $\hat R := D \hat V$ is constructed as a modification of $R = D V$, 
by replacing a (possibly empty) subset  of the nonzero columns of $R$ by zero columns.   
This particular modification preserves the column-reduced property, so 
$\hat R = D \hat V$ is column-reduced.  It follows that  $\hat V$ is a column-reduction of $D$, 
as per Definition \ref{redu}.   

Since $\hat V$ and $V$ are  column-reductions of the same matrix $D$, we see from 
Corollary \ref{redpro}  that  $\hat R$ and  $R$ have the same number (namely $\rank D$) of nonzero columns.   
It follows that in the construction of $\hat R$ as a modification of $R$, {\it none} of the nonzero columns of $R$ can 
be replaced 
by zero columns.  
This establishes the equality $\hat R = R$, and the conclusion 
$D \hat V =  R$ follows. 
\qed \end{proof}

We can now complete the constructive proof of Theorem \ref{prenon}:

\begin{proof}  Letting $B = \hat V$, Lemmas \ref{redif2} and \ref{redifm}   give
$B {\underline D} = R = D B$.   
The pivot matrix ${\underline D} =B^{-1} D B$ is a differential matrix, since it is 
conjugate to the differential matrix $D$.    
It only remains to check that the differential matrix ${\underline D}$ is  almost-Jordan.  
This requires constructing a permutation matrix $P$ such that  $P^{-1} {\underline D} P$ is Jordan. 
Since the differential matrix ${\underline D}$ is furthermore  Boolean and quasi-monomial, 
$P$ can be constructed by the procedure previously explained  
in Example \ref{nuex}.       
\qed \end{proof} 

An immediate  corollary of the proof is the important and generally known fact that the almost-Jordan differential $\underline D$ can be easily 
constructed as the pivot matrix of the column-reduced matrix $R = DV$.  
Furthermore, while $R$ may depend on the choice of column-reduction $V$, Corollary \ref{eunuch} guarantees 
that the pivot matrix $\underline D$ is an invariant of $D$  (which we call the {\it persistence canonical form} of $D$ in 
Section \ref{subsection1.2}) 
independent of the  choice of column-reduction $V$.  
The matrix $\underline D$ contains all of the data for the multiplicities of the summands in a decomposition,  
and this is independent of the choice of decomposition because of the Krull-Schmidt property.  
The data required to compute a particular decomposition is 
conveniently encoded by in the matrix $\hat V$, 
which is a  column-reduction  with additional special properties.    
These points are illustrated in the example at the 
end of the section.  
In the language of the standard framework, one says that the ``barcodes" are contained in $\underline D$,  and 
the ``creators and destroyers" of persistent homology are contained in $\hat V$.   

We also  note that the invertible triangular  matrix $B = \hat V$ produced by the standard algorithm is not in general 
normalized (see the discussion in Section \ref{subsection1.2} following Theorem \ref{prenon}).  
But it is easy 
to construct a diagonal matrix $T$  such that the invertible diagonal matrix  $B = \hat V T$ is normalized. 
This  will also be illustrated in the example at the end of the section.

The graded case is an easy modification:  

% manually number repeat proposition
 \setcounter{repsection}{1} 
\setcounter{repeat}{3}
\begin{reptheorem}
(Matrix  Structural Theorem) 
 Any block-superdiagonal differential matrix $D$ factors as $D = B {\underline D} B^{-1}$ where  
${\underline D}$ is a block-superdiagonal almost-Jordan differential matrix  and $B$ is a block-diagonal triangular matrix. 
\end{reptheorem}

\begin{proof}  Let $D$ be a block-superdiagonal differential matrix $D$.  
Then the invertible triangular column-reduction marix $V$ produced by a reduction algorithm, 
such as \cite{ELZ,ZC1,ZC2} or our Appendix B.1, is block-diagonal.  
If $V$ is block-diagonal, then so is the  invertible triangular matrix $B = \hat V$ constructed by 
the standard algorithm from $V$ and $R = D V$.      
\qed \end{proof}

The following example of a standard algorithm computation illustrates both block-structure and normalization.  

\begin{example} We  work with  block-superdiagonal differential $D: {\mathbb Q}^7 \to {\mathbb Q}^7$ of Example \ref{exfilt}, 
which is presented below with a column augmentation by the identity matrix $I$.   The identity matrix is block-diagonal with 
respect to the grading structure inherited from $D$.  We first compute a column-reduction of $D$:   
%

%************** start matrix ***********

\begin{equation*}
{D \over I} = 
\begin{tikzpicture}[baseline={([yshift=-1.4ex]current bounding box.center)}]

[decoration=brace]
\tikzset{
    node style ge/.style={circle,minimum size=.75cm},
}
\pgfdeclarelayer{background}
\pgfdeclarelayer{foreground}
\pgfsetlayers{background,main,foreground}
\matrix (A) [matrix of math nodes,
             nodes = {node style ge},
             left delimiter  = {[}, %brackets
             right delimiter = {]}, %brackets
             inner sep=-2pt,
		row sep=-.3cm,
		column sep=-.25cm
             ]
{
%MATRIX ENTRIES
0 & 0 &  0 & {-1} & 0 &{-1}  & 0 \\
0 & 0 & 0 & 1 & -1 & 0 & 0 \\
0 & 0 &  0 &  0 & 1 &   1 &  0 \\
0 & 0 & 0 & 0 & 0 & 0 &1 \\
0 & 0 & 0 & 0 & 0 & 0 & 1 \\
0 & 0 & 0 & 0 & 0 & 0 & {-1} \\
0 & 0 & 0 & 0 & 0 & 0 & 0 \\
1 & 0 & 0 & 0 & 0 & 0 & 0 \\
0 & 1 & 0 & 0 & 0 & 0 & 0 \\
0 & 0 & 1 & 0 & 0 & 0 &  0 \\
0 & 0 & 0 & 1 & 0 & 0 & 0 \\
0 & 0 & 0 & 0 & 1 & 0 & 0 \\
0 & 0 & 0 & 0 & 0 & 1 & 0 \\
0 & 0 & 0 & 0 & 0 & 0 & 1 \\
};

%here are the two anchors across the top
\node (top1) at ($(A-1-3.north)!0.5!(A-1-4.north)$) {};
\node (top2) at ($(A-1-6.north)!0.5!(A-1-7.north)$) {};

%anchors down the left
\node (left1) at ($(A-3-1.south west)!0.5!(A-4-1.north west)-(.15,0)$) {};
\node (left2) at ($(A-6-1.south west)!0.5!(A-7-1.north west)-(.15,0)$) {};

\node (leftmid) at ($(A-7-1.south west)!0.5!(A-8-1.north west)-(.15,0)$) {};

%anchors down the right
\node (right1) at ($(A-3-7.south east)!0.5!(A-4-7.north east)+(.15,0)$) {};
\node (right2) at ($(A-6-7.south east)!0.5!(A-7-7.north east)+(.15,0)$) {};

\node (rightmid) at ($(A-7-7.south east)!0.5!(A-8-7.north east)+(.15,0)$) {};

%Anchors in the middle of the matrix take two steps, as somehow averaging the center of 4 points at once didn't work.
%So these are just temporary nodes to be used below.
\node (a_1) at ($(A-3-3.south east)!0.5!(A-3-4.south west)$) {};
\node (a_2) at ($(A-4-3.north east)!0.5!(A-4-4.north west)$) {};
\node (b_1) at ($(A-3-6.south east)!0.5!(A-3-7.south west)$) {};
\node (b_2) at ($(A-4-6.north east)!0.5!(A-4-7.north west)$) {};

\node (c_1) at ($(A-10-3.south east)!0.5!(A-10-4.south west)$) {};
\node (c_2) at ($(A-11-3.north east)!0.5!(A-11-4.north west)$) {};
\node (d_1) at ($(A-13-6.south east)!0.5!(A-13-7.south west)$) {};
\node (d_2) at ($(A-14-6.north east)!0.5!(A-14-7.north west)$) {};

%the top two "intersection points"
\node (mid33) at ($(a_1)!0.5!(a_2)$) {};
\node (mid36) at ($(b_1)!0.5!(b_2)$) {};

\node (mid10_3) at ($(c_1)!0.5!(c_2)$) {};
\node (mid13_6) at ($(d_1)!0.5!(d_2)$) {};

\node (corner_br) at ($(A-14-7.south)!0.5!(A-14-7.east)$) {};

\iffalse

\node (x_1) at ($(A -    i     - (j+1).south west)!0.5!(A-    i     - j.south east)$) {};
\node (x_2) at ($(A - (i+1) - (j+1).north west)!0.5!(A- (i+1) - j.north east)$) {};

\node (mid"ij") at ($(x_1)!0.5!(x_2)$) {};

\fi

%Horizontal and vertical lines

\draw (leftmid) -- (rightmid);

\begin{pgfonlayer}{background}

%Shaded blocks in a background layer so that they don't cover lines, entries, etc

\fill[black!20!white] (mid33) rectangle (top2.south);
\fill[black!20!white] (mid36) rectangle (right2.west);

\fill[black!20!white] (leftmid.east) rectangle (mid10_3);
\fill[black!20!white] (mid13_6) rectangle (mid10_3);
\fill[black!20!white] (mid13_6) rectangle (corner_br.east);

\end{pgfonlayer}

\end{tikzpicture}
\mapsto 
\cdots \mapsto
 \begin{tikzpicture}[baseline={([yshift=-1.4ex]current bounding box.center)}]

[decoration=brace]
\tikzset{
    node style ge/.style={circle,minimum size=.75cm},
}
\pgfdeclarelayer{background}
\pgfdeclarelayer{foreground}
\pgfsetlayers{background,main,foreground}
\matrix (A) [matrix of math nodes,
             nodes = {node style ge},
             left delimiter  = {[}, %brackets
             right delimiter = {]}, %brackets
              inner sep=-2pt,
		row sep=-.3cm,
		column sep=-.25cm
             ]
{
%MATRIX ENTRIES
0 & 0 &  0 & {-1} & {-1} & 0  & 0 \\
0 & 0 & 0 & 1 &0 & 0 & 0 \\
0 & 0 &  0 &  0 & 1 &  0 &  0 \\
0 & 0 & 0 & 0 & 0 & 0 &1 \\
0 & 0 & 0 & 0 & 0 & 0 & 1 \\
0 & 0 & 0 & 0 & 0 & 0 & {-1} \\
0 & 0 & 0 & 0 & 0 & 0 & 0 \\
1 & 0 & 0 & 0 & 0 & 0 & 0 \\
0 & 1 & 0 & 0 & 0 & 0 & 0 \\
0 & 0 & 1 & 0 & 0 & 0 &  0 \\
0 & 0 & 0 & 1 & 1 & -1 & 0 \\
0 & 0 & 0 & 0 & 1 & -1 & 0 \\
0 & 0 & 0 & 0 & 0 & 1 & 0 \\
0 & 0 & 0 & 0 & 0 & 0 & 1 \\
};

%here are the two anchors across the top
\node (top1) at ($(A-1-3.north)!0.5!(A-1-4.north)$) {};
\node (top2) at ($(A-1-6.north)!0.5!(A-1-7.north)$) {};

%anchors down the left
\node (left1) at ($(A-3-1.south west)!0.5!(A-4-1.north west)-(.15,0)$) {};
\node (left2) at ($(A-6-1.south west)!0.5!(A-7-1.north west)-(.15,0)$) {};

\node (leftmid) at ($(A-7-1.south west)!0.5!(A-8-1.north west)-(.15,0)$) {};

%anchors down the right
\node (right1) at ($(A-3-7.south east)!0.5!(A-4-7.north east)+(.15,0)$) {};
\node (right2) at ($(A-6-7.south east)!0.5!(A-7-7.north east)+(.15,0)$) {};

\node (rightmid) at ($(A-7-7.south east)!0.5!(A-8-7.north east)+(.15,0)$) {};

%Anchors in the middle of the matrix take two steps, as somehow averaging the center of 4 points at once didn't work.
%So these are just temporary nodes to be used below.
\node (a_1) at ($(A-3-3.south east)!0.5!(A-3-4.south west)$) {};
\node (a_2) at ($(A-4-3.north east)!0.5!(A-4-4.north west)$) {};
\node (b_1) at ($(A-3-6.south east)!0.5!(A-3-7.south west)$) {};
\node (b_2) at ($(A-4-6.north east)!0.5!(A-4-7.north west)$) {};

\node (c_1) at ($(A-10-3.south east)!0.5!(A-10-4.south west)$) {};
\node (c_2) at ($(A-11-3.north east)!0.5!(A-11-4.north west)$) {};
\node (d_1) at ($(A-13-6.south east)!0.5!(A-13-7.south west)$) {};
\node (d_2) at ($(A-14-6.north east)!0.5!(A-14-7.north west)$) {};

%the top two "intersection points"
\node (mid33) at ($(a_1)!0.5!(a_2)$) {};
\node (mid36) at ($(b_1)!0.5!(b_2)$) {};

\node (mid10_3) at ($(c_1)!0.5!(c_2)$) {};
\node (mid13_6) at ($(d_1)!0.5!(d_2)$) {};

\node (corner_br) at ($(A-14-7.south)!0.5!(A-14-7.east)$) {};

\iffalse

\node (x_1) at ($(A -    i     - (j+1).south west)!0.5!(A-    i     - j.south east)$) {};
\node (x_2) at ($(A - (i+1) - (j+1).north west)!0.5!(A- (i+1) - j.north east)$) {};

\node (mid"ij") at ($(x_1)!0.5!(x_2)$) {};

\fi

%Horizontal and vertical lines

\draw (leftmid) -- (rightmid);

\begin{pgfonlayer}{background}

%Shaded blocks in a background layer so that they don't cover lines, entries, etc

\fill[black!20!white] (mid33) rectangle (top2.south);
\fill[black!20!white] (mid36) rectangle (right2.west);

\fill[black!20!white] (leftmid.east) rectangle (mid10_3);
\fill[black!20!white] (mid13_6) rectangle (mid10_3);
\fill[black!20!white] (mid13_6) rectangle (corner_br.east);

\end{pgfonlayer}

\end{tikzpicture}
= {R \over V}
\end{equation*}

%************** end matrix ***********

\noindent The result of the computation is a factorization $R = DV$, where 
$R$ is column-reduced and $V$ is invertible triangular (and unipotent). 
 The intervening steps are omitted for brevity.  
 The block-superdiagonal almost-Jordan differential $\underline D$ is now easily computed as the pivot matrix of $R$, 
 by setting every column-pivot to $1$ and every other nonzero entry to $0$:
 \begin{equation*}
 {\underline D} =
 \begin{tikzpicture}[baseline={([yshift=-1.4ex]current bounding box.center)}]

[decoration=brace]

\tikzset{
    node style ge/.style={circle,minimum size=.75cm},
}
\pgfdeclarelayer{background}
\pgfdeclarelayer{foreground}
\pgfsetlayers{background,main,foreground}
\matrix (A) [matrix of math nodes,
             nodes = {node style ge},
             left delimiter  = {[}, %brackets
             right delimiter = {]}, %brackets
             %inner sep=-2pt, %this determines how far the matrix brackets scale outward around the entries.
             %row sep=-.3cm, %custom row separation
             %column sep=0cm %custom column separation (currently doing nothing)
              inner sep=-2pt,
		row sep=-.3cm,
		column sep=-.25cm
             ]
{
%MATRIX ENTRIES
  0 & 0 &  0 &  {0} & 0 & {0}  & 0 \\
  0 & 0 & 0 & 1 & 0 & 0 & 0 \\
  0 & 0 &  0 &   0 & 1 &   0 &  0 \\
  0 & 0 & 0 & 0 & 0 & 0 & 0 \\
 0 & 0 & 0 & 0 & 0 & 0 & 0 \\
 0 & 0 & 0 & 0 & 0 & 0 & {1} \\
 0 & 0 & 0 & 0 & 0 & 0 & 0 \\
};

%These are ANCHORS, invisible nodes in the middle/borders of the matrix that I used to draw the lines and make the shading.
%Just for reference, the "!0.5!" stuff is used to find halfway points between nodes to ensure that the lines are drawn exactly down the middle of two adjacent columns.

%here are the two anchors across the top
\node (top1) at ($(A-1-3.north)!0.5!(A-1-4.north)$) {};
\node (top2) at ($(A-1-6.north)!0.5!(A-1-7.north)$) {};

%anchors across the bottom
\node (bot1) at ($(A-7-3.south)!0.5!(A-7-4.south)$) {};
\node (bot2) at ($(A-7-6.south)!0.5!(A-7-7.south)$) {};

%anchors down the left
\node (left1) at ($(A-3-1.south west)!0.5!(A-4-1.north west)-(.15,0)$) {};
\node (left2) at ($(A-6-1.south west)!0.5!(A-7-1.north west)-(.15,0)$) {};

%anchors down the right
\node (right1) at ($(A-3-7.south east)!0.5!(A-4-7.north east)+(.15,0)$) {};
\node (right2) at ($(A-6-7.south east)!0.5!(A-7-7.north east)+(.15,0)$) {};

%Anchors in the middle of the matrix take two steps, as somehow averaging the center of 4 points at once didn't work.
%So these are just temporary nodes to be used below.
\node (a_1) at ($(A-3-3.south east)!0.5!(A-3-4.south west)$) {};
\node (a_2) at ($(A-4-3.north east)!0.5!(A-4-4.north west)$) {};
\node (b_1) at ($(A-3-6.south east)!0.5!(A-3-7.south west)$) {};
\node (b_2) at ($(A-4-6.north east)!0.5!(A-4-7.north west)$) {};

%the top two "intersection points"
\node (mid11) at ($(a_1)!0.5!(a_2)$) {};
\node (mid12) at ($(b_1)!0.5!(b_2)$) {};

\iffalse

\node (x_1) at ($(A -    i     - (j+1).south west)!0.5!(A-    i     - j.south east)$) {};
\node (x_2) at ($(A - (i+1) - (j+1).north west)!0.5!(A- (i+1) - j.north east)$) {};

\node (mid"ij") at ($(x_1)!0.5!(x_2)$) {};

\fi

%Horizontal and vertical lines

%\draw (left1) -- (right1);
%\draw (left2) -- (right2);
%\draw (top1) -- (bot1);
%\draw (top2) -- (bot2);

\begin{pgfonlayer}{background}

%Shaded blocks in a background layer so that they don't cover lines, entries, etc

\fill[black!20!white] (mid11) rectangle (top2.south);
\fill[black!20!white] (mid12) rectangle (right2.west);

\end{pgfonlayer}

\end{tikzpicture}
.
\end{equation*}
The barcode invariants can be computed from the matrix  $\underline D$ and the filtration levels of the basis elements.  

Proceeding to compute a particular decomposition as in Example \ref{nuex}, we use the standard algorithm to construct  $\hat V$ as a modification of $V$.   
Each nonzero column of $R$ replaces the column of 
 $V$ that has its column-pivot in the same row.  Then $\hat V$ inherits the block-diagonal 
 structure of $V$:  
 %

%************** start matrix ***********

\begin{equation*}
V=
\begin{tikzpicture}[baseline={([yshift=-.5ex]current bounding box.center)}]

[decoration=brace]
\tikzset{
    node style ge/.style={circle,minimum size=.75cm},
}
\pgfdeclarelayer{background}
\pgfdeclarelayer{foreground}
\pgfsetlayers{background,main,foreground}
\matrix (A) [matrix of math nodes,
             nodes = {node style ge},
             left delimiter  = {[}, %brackets
             right delimiter = {]}, %brackets
             %inner sep=-2pt, %this determines how far the matrix brackets scale outward around the entries.
             %row sep=-.3cm, %custom row separation
             %column sep=0cm %custom column separation (currently doing nothing)
             inner sep=-2pt,
		row sep=-.3cm,
		column sep=-.25cm
             ]
{
%MATRIX ENTRIES
1 & 0 & 0 & 0 & 0 & 0 & 0 \\
0 & 1 & 0 & 0 & 0 & 0 & 0 \\
0 & 0 & 1 & 0 & 0 & 0 &  0 \\
0 & 0 & 0 & 1 & 1 & -1 & 0 \\
0 & 0 & 0 & 0 & 1 & -1 & 0 \\
0 & 0 & 0 & 0 & 0 & 1 & 0 \\
0 & 0 & 0 & 0 & 0 & 0 & 1 \\
};

%here are the two anchors across the top
\node (top1) at ($(A-1-3.north)!0.5!(A-1-4.north)$) {};
\node (top2) at ($(A-1-6.north)!0.5!(A-1-7.north)$) {};

%anchors across the bottom
\node (bot1) at ($(A-7-3.south)!0.5!(A-7-4.south)$) {};
\node (bot2) at ($(A-7-6.south)!0.5!(A-7-7.south)$) {};

%anchors down the left
\node (left1) at ($(A-3-1.south west)!0.5!(A-4-1.north west)-(.15,0)$) {};
\node (left2) at ($(A-6-1.south west)!0.5!(A-7-1.north west)-(.15,0)$) {};

%anchors down the right
\node (right1) at ($(A-3-7.south east)!0.5!(A-4-7.north east)+(.15,0)$) {};
\node (right2) at ($(A-6-7.south east)!0.5!(A-7-7.north east)+(.15,0)$) {};

%new corner anchors
\node (corner_tl) at (A-1-1.north west) {};
\node (corner_br) at (A-7-7.south east) {};

%Anchors in the middle of the matrix take two steps, as somehow averaging the center of 4 points at once didn't work.
%So these are just temporary nodes to be used below.
\node (a_1) at ($(A-3-3.south east)!0.5!(A-3-4.south west)$) {};
\node (a_2) at ($(A-4-3.north east)!0.5!(A-4-4.north west)$) {};
\node (b_1) at ($(A-6-6.south east)!0.5!(A-6-7.south west)$) {};
\node (b_2) at ($(A-7-6.north east)!0.5!(A-7-7.north west)$) {};

%the top two "intersection points"
\node (mid11) at ($(a_1)!0.5!(a_2)$) {};
\node (mid22) at ($(b_1)!0.5!(b_2)$) {};

\iffalse

\node (x_1) at ($(A -    i     - (j+1).south west)!0.5!(A-    i     - j.south east)$) {};
\node (x_2) at ($(A - (i+1) - (j+1).north west)!0.5!(A- (i+1) - j.north east)$) {};

\node (mid"ij") at ($(x_1)!0.5!(x_2)$) {};

\fi

\begin{pgfonlayer}{background}

%Shaded blocks in a background layer so that they don't cover lines, entries, etc

\fill[black!20!white] (mid11) rectangle (mid22);
\fill[black!20!white] (mid11) rectangle (corner_tl);
\fill[black!20!white] (mid22) rectangle (corner_br);

\end{pgfonlayer}

\end{tikzpicture}
%************** end matrix ***********
\mapsto \cdots \mapsto
%************** start matrix ***********
%\begin{tikzpicture}
 \begin{tikzpicture}[baseline={([yshift=-.5ex]current bounding box.center)}]

[decoration=brace]
\tikzset{
    node style ge/.style={circle,minimum size=.75cm},
}
\pgfdeclarelayer{background}
\pgfdeclarelayer{foreground}
\pgfsetlayers{background,main,foreground}
\matrix (A) [matrix of math nodes,
             nodes = {node style ge},
             left delimiter  = {[}, %brackets
             right delimiter = {]}, %brackets
              inner sep=-2pt,
		row sep=-.3cm,
		column sep=-.25cm
             ]
{
%MATRIX ENTRIES
 1 & -1 &{ -1} & 0 & 0 & 0 & 0 \\
  0 & 1& 0 & 0 &0 & 0 & 0 \\
0 & 0 &1  & 0 & 0 & 0 & 0 \\
  0 & 0 & 0 & 1 & 1 & 1 & 0 \\
 0 & 0 & 0 & 0 & 1 & 1 & 0 \\
 0 & 0 & 0 & 0 & 0 &-1  & 0 \\
 0 & 0 & 0 & 0 & 0 & 0 &  1 \\
};

%here are the two anchors across the top
\node (top1) at ($(A-1-3.north)!0.5!(A-1-4.north)$) {};
\node (top2) at ($(A-1-6.north)!0.5!(A-1-7.north)$) {};

%anchors across the bottom
\node (bot1) at ($(A-7-3.south)!0.5!(A-7-4.south)$) {};
\node (bot2) at ($(A-7-6.south)!0.5!(A-7-7.south)$) {};

%anchors down the left
\node (left1) at ($(A-3-1.south west)!0.5!(A-4-1.north west)-(.15,0)$) {};
\node (left2) at ($(A-6-1.south west)!0.5!(A-7-1.north west)-(.15,0)$) {};

%anchors down the right
\node (right1) at ($(A-3-7.south east)!0.5!(A-4-7.north east)+(.15,0)$) {};
\node (right2) at ($(A-6-7.south east)!0.5!(A-7-7.north east)+(.15,0)$) {};

%new corner anchors
\node (corner_tl) at (A-1-1.north west) {};
\node (corner_br) at (A-7-7.south east) {};

%Anchors in the middle of the matrix take two steps, as somehow averaging the center of 4 points at once didn't work.
%So these are just temporary nodes to be used below.
\node (a_1) at ($(A-3-3.south east)!0.5!(A-3-4.south west)$) {};
\node (a_2) at ($(A-4-3.north east)!0.5!(A-4-4.north west)$) {};
\node (b_1) at ($(A-6-6.south east)!0.5!(A-6-7.south west)$) {};
\node (b_2) at ($(A-7-6.north east)!0.5!(A-7-7.north west)$) {};

%the top two "intersection points"
\node (mid11) at ($(a_1)!0.5!(a_2)$) {};
\node (mid22) at ($(b_1)!0.5!(b_2)$) {};

\iffalse

\node (x_1) at ($(A -    i     - (j+1).south west)!0.5!(A-    i     - j.south east)$) {};
\node (x_2) at ($(A - (i+1) - (j+1).north west)!0.5!(A- (i+1) - j.north east)$) {};

\node (mid"ij") at ($(x_1)!0.5!(x_2)$) {};

\fi

\begin{pgfonlayer}{background}

%Shaded blocks in a background layer so that they don't cover lines, entries, etc

\fill[black!20!white] (mid11) rectangle (mid22);
\fill[black!20!white] (mid11) rectangle (corner_tl);
\fill[black!20!white] (mid22) rectangle (corner_br);

\end{pgfonlayer}

\end{tikzpicture}
= {\hat V}
\end{equation*}

%************** end matrix **********
\noindent  We list the columns of $\hat V$ that are equal to columns of $R$;   this data is 
also encoded by the nonzero entries of the pivot matrix  $\underline D$:   
\begin{itemize}[leftmargin=*]
\item	Column 2 of $\hat V$ is equal to column 4 of $R$; row 2 column 4 of $\underline D$ has entry $1$.  
\item Column 3 of $\hat V$ is equal to column 5 of $R$; row 3 column 5 of $\underline D$ has entry $1$.  
\item Column 6 of $\hat V$ is equal to column 7 of $R$; row 6 column 7 of $\underline D$ has entry $1$.    
\end{itemize}
\noindent Each of the remaining columns of $\hat V$ is equal to the corresponding column of $V$.  
%
%Now  by Corollary \ref{corcid},  
One may now check by matrix multiplication that  ${\hat V^{-1}} D \, \hat V={\underline D}$, 
where ${\underline D}$ is the pivot matrix of $R$ as above. 

The invertible triangular matrix $\hat V$ is not normalized:  
column 6 of $\hat V$ corresponds to a zero column of ${\underline D}$, but its 
diagonal entry is not equal to  
$1$.  We can normalize by scalar multiplication of the appropriate columns.  
Let $T$ be the diagonal matrix with  
$1$ in the first five diagonal entries and $-1$ in the last two.  
Then the invertible triangular matrix $B = {\hat V}T$ is normalized, and 
this is the matrix that  appears in Example \ref{exfilt}.  
Note that $B^{-1} D B = {\underline D} =  {\hat V}^{-1} D\,  {\hat V}$ by   Corollary \ref{eunuch}. 
\end{example}

%\begin{acknowledgements}
%\end{acknowledgements}



\end{document}